\documentclass[a4paper]{article}
 
\usepackage{amsmath}
\usepackage{amssymb}
\usepackage{mathtools}
\usepackage{empheq}
\usepackage{amsthm}
\usepackage{dsfont}
\usepackage{xcolor}
\usepackage{enumitem}
\usepackage{appendix}
\usepackage{authblk}
\usepackage{a4wide}
\usepackage{epigraph}
\usepackage[hidelinks]{hyperref}       
\usepackage[backend=biber]{biblatex}
\usepackage{nccmath}
\mathtoolsset{showonlyrefs}
\numberwithin{equation}{section}

\newtheorem{assumption}{Assumption}
\newtheorem{setting}{Setting}
\newtheorem{theorem}{Theorem}[section]
\newtheorem{proposition}[theorem]{Proposition}
\newtheorem{corollary}[theorem]{Corollary}
\newtheorem{lemma}[theorem]{Lemma}
\newtheorem{definition}[theorem]{Definition}
\newtheorem{remark}[theorem]{Remark}
\newtheorem{example}[theorem]{Example}

\DeclareMathOperator*{\argmin}{arg\,min}
\DeclareMathOperator{\Cov}{Cov}
\DeclareMathOperator{\diag}{diag}
\DeclareMathOperator{\erf}{erf}
\DeclareMathOperator{\essinf}{ess\,inf}
\DeclareMathOperator{\Id}{Id}
\DeclareMathOperator{\supp}{supp}

\def\biggg#1{{\hbox{$\left#1\vbox to20.5pt{}\right.$}}}

\newcommand{\IC}{\mathbb C}
\newcommand{\IN}{\mathbb N}
\newcommand{\IR}{\mathbb R}
\newcommand{\IRNN}{\IR^{N\times N}}

\newcommand{\B}{\mathcal B}
\newcommand{\calG}{\mathcal{G}}
\newcommand{\F}{\mathcal F}
\newcommand{\N}{\mathcal N}
\newcommand{\T}{\mathcal T}
\newcommand{\X}{\mathcal X}
\newcommand{\Y}{\mathcal Y}
\newcommand{\Z}{\mathcal Z}

\newcommand{\Xalpha}{\X^\alpha}
\newcommand{\XC}{\X_{\IC}}
\newcommand{\Ybeta}{\Y^\beta}

\newcommand{\OC}{O_{\IC}}

\newcommand{\BX}{\mathcal{B}(\X)}
\newcommand{\BY}{\mathcal{B}(\Y)}

\newcommand{\CzKY}{C^0(K,\Y)}

\newcommand{\LtwoCrsIPxYE}{L^2(C_R^s(\X),\IPX;\Y)}
\newcommand{\LtwoK}{L^2(K)}

\newcommand{\IE}{\mathbb E}
\newcommand{\IP}{\mathbb P}
\newcommand{\IT}{\mathbb T}
\newcommand{\IPE}{\IP_{E}}
\newcommand{\IPpE}{\IP_{\boldsymbol{E}}}
\newcommand{\IPpX}{\IP_{\pXX}}
\newcommand{\IPX}{\IP_{X}}

\newcommand{\IW}{\mathbb W}
\newcommand{\WY}{\IW_{\Y}}

\newcommand{\Ei}{E_i}

\newcommand{\bs}[1]{\boldsymbol{#1}}
\newcommand{\Bb}{{\boldsymbol{b}}}
\newcommand{\Bc}{{\boldsymbol{c}}}
\newcommand{\Brho}{{\boldsymbol{\gamma}}}
\newcommand{\pe}{\boldsymbol{e}}
\newcommand{\pEE}{\boldsymbol{E}}
\newcommand{\pZZ}{\boldsymbol{Z}}
\newcommand{\pK}{\boldsymbol{K}}
\newcommand{\pkappa}{\boldsymbol{\kappa}}
\newcommand{\pnu}{\boldsymbol{\nu}}
\newcommand{\pPhi}{\boldsymbol{\Phi}}
\newcommand{\ppsi}{\boldsymbol{\psi}}
\newcommand{\pu}{\boldsymbol{u}}

\newcommand{\px}{\boldsymbol{x}}
\newcommand{\pgamma}{\boldsymbol{\gamma}}
\newcommand{\pX}{\boldsymbol{\X}}
\newcommand{\pXX}{\boldsymbol{X}}
\newcommand{\pz}{\boldsymbol{z}}
\newcommand{\pZ}{\boldsymbol{\Z}}

\newcommand{\Hpost}{H_{\text{post}}}
\newcommand{\Hprior}{H_{\text{prior}}}

\newcommand{\lV}{\lVert}
\newcommand{\rV}{\rVert}

\newcommand{\dup}[3][]{\langle #2,#3\rangle_{#1}}
\newcommand{\dupb}[3][]{\Big\langle #2,#3\Big\rangle_{#1}}
\newcommand{\set}[2]{\{#1\,|\,#2\}}

\newcommand{\prodjNprime}{\prod_{j'\leq N}}
\newcommand{\sumin}{\sum_{i=1}^n}
\newcommand{\sumj}{\sum_{j\in\IN}}
\newcommand{\sumjlN}{\sum_{j,\ell\leq N}}
\newcommand{\sumjlNprime}{\sum_{j',\ell'\leq N}}
\newcommand{\sumjN}{\sum_{j\leq N}}

\newcommand{\sumlNnotj}{\sum_{\substack{\ell\leq N\\ \ell\neq j}}}
\newcommand{\sumjNprime}{\sum_{j'\leq N}}
\newcommand{\sumjnr}{\sum_{j\leq\lceil  n^{1/(\rr-\rho-1)}\rceil}}
\newcommand{\sumjnrgr}{\sum_{j> \lceil n^{1/(\rr-\rho-1)}\rceil}}

\newcommand{\sumjntwor}{\sum_{j\leq \lceil n^{1/(2\rr)}\rceil}}
\newcommand{\sumk}{\sum_{k\in\IN}}

\newcommand{\sumlnr}{\sum_{\ell\leq \lceil n^{1/(\rr-\rho-1)}\rceil}}
\newcommand{\sumlnrnotj}{\sum_{\substack{\ell\leq \lceil n^{1/(\rr-\rho-1)}\rceil\\ \ell\neq j}}}
\newcommand{\sumlnrgr}{\sum_{\ell> \lceil n^{1/(\rr-\rho-1)}\rceil}}
\newcommand{\sumlnrgrnotj}{\sum_{\substack{\ell> \lceil n^{1/(\rr-\rho-1)}\rceil\\ \ell\neq j}}}
\newcommand{\limNinfty}{\lim_{N\to\infty}}
\newcommand{\supxN}{\sup_{\lVert x_N\rVert_2=1}}

\newcommand{\eps}{\varepsilon}
\newcommand{\qq}{q}
\newcommand{\rr}{r}
\newcommand{\ssv}{s}
\newcommand{\ttv}{t}

\newcommand{\alphajj}{\alpha_{jj}}
\newcommand{\alphajl}{\alpha_{j\ell}}
\newcommand{\alphajlprime}{\alpha_{j'\ell'}}
\newcommand{\alphall}{\alpha_{\ell\ell}}
\newcommand{\anu}{a_{\pnu}}
\newcommand{\betaj}{\beta_{j}}
\newcommand{\betajprime}{\beta_{j'}}
\newcommand{\cnuk}{c_{\pnu,k}}
\newcommand{\deltajl}{\delta_{j\ell}}
\newcommand{\deltajlprime}{\delta_{j'\ell'}}
\newcommand{\enu}{e_{\pnu}}

\newcommand{\etak}{\eta_k}

\newcommand{\kappaj}{\kappa_{j}}
\newcommand{\kappajprime}{\kappa_{j'}}
\newcommand{\kappal}{\kappa_\ell}
\newcommand{\Lj}{L_j}
\newcommand{\Lnu}{L_{\pnu}}
\newcommand{\omegaj}{\omega_{j}}
\newcommand{\omeganu}{\omega_{\pnu}}
\newcommand{\phij}{\phi_j}
\newcommand{\psij}{\psi_{j}}
\newcommand{\psinuk}{\psi_{\pnu,k}}
\newcommand{\thetaj}{\theta_j}
\newcommand{\tomega}{\tilde\omega}

\newcommand{\gj}{g_j}
\newcommand{\gjprime}{g_{j'}}
\newcommand{\gl}{g_\ell}
\newcommand{\glprime}{g_{\ell'}}
\newcommand{\Gznuk}{g_{0,\pnu,k}}
\newcommand{\Gzj}{g_{0,j}}
\newcommand{\GzeroN}{\bs{g}_{0,N}}

\newcommand{\gnhat}{\hat{g}_n}
\newcommand{\hatgn}{\hat{g}_n}

\newcommand{\AN}{\bs{A}_N}
\newcommand{\ANin}{\AN^{-1}}
\newcommand{\bN}{\bs{b}_N}
\newcommand{\DN}{\bs{D}_N}
\newcommand{\DNhalf}{\bs{D}_N^{1/2}}
\newcommand{\DNhalfinv}{\bs{D}_N^{-1/2}}
\newcommand{\DNinv}{\bs{D}_N^{-1}}
\newcommand{\EN}{\bs{E}_N}
\newcommand{\GN}{\bs{G}_N}
\newcommand{\GNj}{(\GN)_j}
\newcommand{\GNt}{\GN^T}
\newcommand{\IIN}{\bs{I}_N}
\newcommand{\psiN}{\bs{\psi}_N}
\newcommand{\psiNt}{\psiN^T}
\newcommand{\RRN}{\bs{R}_N}
\newcommand{\uN}{\bs{u}_N}
\newcommand{\vN}{\bs{v}_N}
\newcommand{\wN}{\bs{w}_N}

\newcommand{\CRr}{C_R^{\ssv}}
\newcommand{\dd}{\;\mathrm{d}}
\newcommand{\One}{\mathds{1}}
\newcommand{\Sdeltag}{S_{>}^\delta}
\newcommand{\Sdeltal}{S_{\leq}^\delta}
\newcommand{\sigmaRr}{\sigma_{R}^{\ssv}}

\newcommand{\trm}{\textrm}
\DeclareMathOperator{\Unif}{Unif}
\usepackage{todonotes}

\title{Bayesian Operator Learning: Posterior Existence and Convergence of Point Estimates for Gaussian Priors}
\begin{document}
\author[1]{Niklas Reinhardt}
\author[1]{Jakob Zech}
\affil[1]{\small Interdisziplin\"ares Zentrum f\"ur wissenschaftliches Rechnen, Universit\"at Heidelberg, Im Neuenheimer Feld 205, 69120 Heidelberg, Germany
  
  \texttt{niklas.reinhardt@iwr.uni-heidelberg.de}\qquad
    \texttt{jakob.zech@uni-heidelberg.de}}

\maketitle
\begin{abstract}
We develop a Bayesian framework for learning nonlinear operators between infinite-dime\-nsional spaces. Given
a map $g_0:\mathcal{X}\to\mathcal{Y}$ between separable Hilbert spaces, we study the recovery of $g_0$ from $n\in\mathbb{N}$ noisy
input-output pairs $(\boldsymbol{X},\boldsymbol{Z})=(X_i,Z_i)_{i=1}^n$ with $Z_i= g_0 (X_i ) + E_i$. Here the $X_i\in\mathcal{X}$ are randomly drawn ``design'' points in a compact subset of $\X$, and the $E_i$ are assumed to be i.i.d.~draws from a Gaussian white noise process indexed by $\mathcal{Y}$. For any ``operator-valued'' prior $\IP_G$ supported on the space of continuous operators, we show existence of the posterior $\mathbb P_{G|(\boldsymbol{X},\boldsymbol{Z})}$ as a regular conditional distribution, and provide a characterization of its Radon-Nikodym derivative.
For Gaussian priors, we establish \emph{algebraic} (in the sample size $n$) convergence rates for the posterior mean towards the ground truth; this corresponds to a ridge regularized kernel estimator. Moreover, we show that the posterior mean is minimax optimal (up to logarithmic factors) over hyperrectangles when the smoothness of the prior matches that of the ground truth.
To illustrate the applicability of our analysis, we derive explicit learning rates for the Darcy flow solution operator.
\end{abstract}

\tableofcontents

\section{Introduction}
The study of nonlinear relationships from high- and infinite-dimensional data remains a fundamental problem with applications in machine learning and data science.
Learning nonlinear mappings in infinite dimensions is mainly motivated through the natural sciences, where phenomena are often described through differential equations and the ground truth $g_0$ consequently maps between (infinite-dimensional) function spaces. Important examples are solution operators of parametric partial differential equations (PDEs) or operators relating the initial state of a dynamical system to the state at some later time.

Various approaches to approximate such PDE solution operators have been proposed including compressed sensing \cite{DOOSTAN20113015,RS16_1344}, sparse grid interpolation \cite{MR3800593,NobilerandomPDE}, least-squares \cite{leastsquaresI,leastsquaresII} and reduced basis methods \cite{MR3379913,MR3408061}.
More recently, under the framework known as \emph{Operator Learning}, neural network-based approaches have gained considerable attention for approximating nonlinear mappings between 
Banach spaces \cite{li2020fourier,HESTHAVEN201855,BhattModRedNNPDE,Lu2021,raonic2023convolutional,anandkumar2019neural,batlle2023kernel,nelsen2024operator}; see \cite{kovachki2024operator} for a review. Being purely data driven, these approaches provide fast-to-evaluate surrogate models even in situations where a mathematical description of the underlying physical process is not available.

The \emph{Bayesian nonparametric} (BNP) learning framework \cite{ferguson1973bayesian,doksum1974tailfree,doksum1972decision,ferguson1974prior} offers an alternative route to learn an (infinite-dimensional) unknown parameter. By explicitly incorporating prior knowledge in form of a probability distribution, the measurements are used to update that \emph{prior} distribution by conditioning on the data resulting in the \emph{posterior} distribution. Point estimators can then be derived from the posterior.
For a given measurement model, establishing a BNP framework usually involves
\begin{enumerate}[label=(\roman*)]
    \item \label{item:intro1} choosing a \emph{prior} distribution $\IP_G$ on some (infi\-nite-dimensional) candidate space $\calG\ni g_0$ incorporating existing knowledge of $g_0$,
    \item \label{item:intro2} showing existence of a well-defined posterior $\IP_{G|(\pXX,\pZZ)}$, where 
    $(\pXX,\pZZ)$ 
    are (random) measurements generated by $g_0$, and $G\sim\IP_G$,
    \item \label{item:intro3} examining the convergence behavior of the posterior mean estimator and the maximum a posteriori (MAP) estimator, calculated from $\IP_{G|(\pXX,\pZZ)}$, towards the ground truth $g_0$,
    \item \label{item:intro4} showing contraction of the posterior measure towards the ground truth \emph{in distribution} by establishing Bernstein--von Mises-type theorems.
\end{enumerate}
Within this framework, the parameter is typically a function or distribution, see, e.g.,~\cite{hjort2010bayesian,ghosal2017fundamentals}. In this paper, we discuss \ref{item:intro1}--\ref{item:intro3} for an operator-valued Bayesian framework, while \ref{item:intro4} is left as future work. For the convergence rates \ref{item:intro3}, we resort to a Gaussian setting.

\subsection{Outline and Contributions}\label{sec:outline}
This work establishes a BNP framework for parameters that are operators between infinite-dimensional spaces. Specifically, we show existence of the posterior \ref{item:intro2} for a general class of priors \ref{item:intro1} supported on the space of continuous operators. Moreover, for Gaussian priors whose eigenfunctions form an ONB of $L^2(\X,\IP_X;\Y)$, we show that the posterior mean estimator and the MAP estimator coincide and converge towards the ground truth \ref{item:intro3} with certain rates.
Finally, we apply our theory to the parametric Darcy Flow.

To illustrate our contributions, 
we next state an informal version of our main convergence result (for full details see Theorem \ref{th:convergencetheorem} and Corollary \ref{cor:Expectationconvergence}).
Consider a \emph{supervised learning} set-up: Let $\X$ and $\Y$ be separable Hilbert spaces and let data $(X_i,Z_i)$ be generated by
$$Z_i=g_0(X_i)+ E_i,~~ i=1,\ldots,n.$$ Here the $X_i\overset{\rm iid}{\sim}\IP_X$ are i.i.d.\ with their distribution
$\IP_X$ compactly supported on $\supp(\IP_X)\coloneqq K\subset\X$, and the noise variables $E_i$ are i.i.d.~Gaussian white noise processes indexed by $\mathcal{Y}$. The aim is to infer the ground truth $g_0$ within a Bayesian framework. We therefore choose a continuous ONB $(\psi_j)_{j\in\IN}$ of $L^2(K)\coloneqq L^2(K,\IP_X;\Y)$ and introduce the (operator-valued) Gaussian prior $\IP_G$ as the law of 
\begin{align*}
    G\coloneqq\sum_{j\in\IN}\xi_j\psi_j,\quad \xi_j\sim\N(0,\kappa_j^2),
\end{align*} 
where the $\xi_j$, $j\in\IN$, are independent. For $(\kappa_j)_{j\in\IN}\in\ell^2(\IN)$ this gives a well-defined Gaussian prior on $L^2(K)$. 

\begin{theorem}[Informal]\label{th:informal}
Consider the  setup from above.  
Further let $g_0\in C^0(K,\Y)$, $\Gzj\coloneqq\langle g_0,\psij\rangle_{L^2(K)}$, $j\in\IN$, $q>1/2$, $r>1$ and  
    \begin{align}\label{eq:informaldecay}
        |\Gzj|\lesssim j^{-\qq},~~\kappaj\simeq j^{-\rr},~~\lVert\psij\rVert_{C^0(K)}\lesssim 1.
    \end{align}
    Then the posterior $\IP_{G|(\pXX,\pZZ)}$ and its mean $\hatgn$ exist and there exists a constant $C>0$ such that for $n\in\IN$, $n\geq 2$ it holds  
    \begin{align}\label{eq:informalrate}
\IE_{(X_i,Z_i)}\left[\lVert\hatgn-g_0\rVert_{L^2(K)}^2\right]\leq C\log(n)^2n^{-\frac{2\min\{\qq,\rr\}-1}{2\rr}}.
\end{align}
\end{theorem}
Theorem \ref{th:informal} is a special case of Corollary \ref{cor:Expectationconvergence}. Its most important feature
is that the convergence rate in \eqref{eq:informalrate} is algebraic in $n$ and does therefore circumvent the curse of dimensionality. This is possible since we assumed algebraic decay of the expansion coefficients $\Gzj$ in our chosen enumeration of the coefficients. In finite dimensions, this 
can, for example, be interpreted as a form of Sobolev regularity of $g_0$. In infinite dimensions, this assumption is 
in general not fulfilled for operators with finite regularity, see, e.g.,~\cite{kovachki2024data}. One relevant class of operators where \eqref{eq:informaldecay} is well-known to hold is
\emph{holomorphic operators}, see Section \ref{sec:holomorphicoperators} below. In \eqref{eq:informalrate}, 
the best convergence rate is 
achieved with $\rr=\qq$, i.e.~by matching the prior smoothness $\rr$ and the ``true'' smoothness $\qq$ of $g_0$. 

This is intuitive, since the prior then encodes the ``correct'' regularity of the ground truth. The cases $\rr<\qq$ and $\rr>\qq$ are usually referred to as the \emph{undersmoothing} and \emph{oversmoothing} regime, where the smoothness of the prior is under- or overestimated, e.g.~\cite{knapik2011bayesian}.

To prove Theorem \ref{th:informal}, we proceed in several steps. We 
first introduce the measurement model with the corresponding Gaussian white noise in Section \ref{sec:measurenementmodel}. Existence of the posterior as a \emph{regular conditional distribution} (e.g.~\cite[Section 8]{klenke2008probability}) is verified in Section \ref{sec:exposterior}. Theorem \ref{th:existenceposterior} first provides an abstract existence result independent of the noise model. Applying it to the white noise case results in Theorem \ref{th:PosteriorGaussianNoise}, which shows that the posterior exists for all priors supported on the space of continuous operators. Examples such as the Gaussian and the uniform prior are discussed in Section \ref{sec:Prior}.

In Section \ref{sec:existencealgebraicform} we prove a more general version of Theorem \ref{th:informal}. We start by showing existence of the posterior mean $\hatgn$ and explicitly calculate $\hatgn$ for the case of a linear Gaussian prior. We then prove our main result, Theorem \ref{th:convergencetheorem}, which gives a uniform concentration inequality for the convergence of $\hatgn$ to $g_0$. Theorem \ref{th:informal} is 
a special case of Corollary \ref{cor:Expectationconvergence}.  
The MAP estimator is analyzed in Section \ref{sec:MAPestimator}, and the maximum likelihood (ML) estimator is recalled from \cite{reinhardt2026statistical} in Section \ref{sec:MLestimator}. In Section \ref{sec:kernel} we
discuss the connection to vector-valued kernel methods, and we show minimax optimality (up to logarithmic factors) of the posterior mean over hyperrectangles in Section \ref{sec:minimax}.

In Section \ref{sec:holomorphicoperators}, we show that holomorphic operators fulfill the assumptions of Theorem \ref{th:convergencetheorem} and hence can be efficiently learned in a BNP framework. As the solution operator of the Darcy Flow PDE is known to be holomorphic, we provide explicit learning rates for this example, see Proposition \ref{prop:darcy} in Section \ref{sec:darcy}. We conclude our work and give an outlook to future research questions in Section \ref{sec:conclusion}.

\subsection{Existing Results}\label{sec:existingresults}
Studying the approximation properties of mappings between infinite-dimensio\-nal spaces has been studied extensively in Uncertainty Quantification, e.g., \cite{NobilerandomPDE,cohen2010convergence,schwab2011sparse,CohenDeVore,MR3379913}. More recently, various architectures for the computation of data driven surrogates have been proposed and investigated in Scientific Machine learning, see, e.g., \cite{HESTHAVEN201855,Lu2021,li2020fourier,BhattModRedNNPDE,anandkumar2019neural,raonic2023convolutional,OLEARYROSEBERRY2022114199,%
JMLR:v22:21-0806,batlle2023kernel,Adcock.2024b}, the surveys \cite{KOVACHKI2024419,SWZ26_1171}, and for practical performance comparisons \cite{LU2022114778,hoop2022,Westermann2026Performance}.
In terms of approximation properties, our analysis builds strongly on the holomorphic framework of \cite{CDS11,chkifa2015breaking,CohenDeVore,herrmann2024neural}.   

\paragraph*{Statistical Learning Theory in Infinite Dimensions}
Recently, the statistical analysis of \emph{sample complexity}, which studies how much data is needed to achieve a certain error, has gained interest. For linear and diagonalizable operators, such an analysis can be found in \cite{de2023convergence}; for lower bounds see, e.g.,~\cite{chagny2022adaptive}. Logarithmic lower bounds on learning rates are known for both Lipschitz \cite{Mhaskar.1997,adcock2025towards,adcock2024sample} and $C^k$-operators \cite{kovachki2024data} showing that finite regularity does in general not suffice to learn operators with an algebraic learning rate. For the case of holomorphic operators algebraic learning rates can be shown, see, e.g.,~\cite{ADCOCK2025106761,Adcock.2024b} and also our recent work \cite{reinhardt2026statistical} which yields algebraic learning rates for concrete PDE solution operators. Except for the linear result in \cite{de2023convergence}, none of these works examine the learning properties of operators in a Bayesian framework. 

\paragraph*{Bayesian Learning}
The Bayesian nonparametric (BNP) framework \cite{ferguson1973bayesian,doksum1974tailfree,doksum1972decision,ferguson1974prior} is well-studied and has had considerable success in e.g.~function approximation \cite{lavine1992some,muller1996bayesian,van2008rates} and density estimation \cite{lo1984class,escobar1995bayesian}. The framework has also been successfully extended to Gaussian white noise models \cite{castillo2013nonparametric,castillo2014bernstein} and inverse problems; in the latter case, well-posedness of the posterior \cite{latz2020well,stuart2010inverse}, posterior contraction \cite{dashti2013map,knapik2011bayesian} and \emph{Bernstein--von Mises theorems} \cite{nickl2023bayesian,monard2021statistical} have been studied for linear and nonlinear operators. In all of these works the prior is placed on the unknown parameter of a known forward map, whereas we place a prior on the operator itself. Moreover, the Bayesian framework has been applied for learning parametric differential systems in the context of reduced order modelling \cite{guo2022bayesian,mcquarrie2025active}. The authors establish an algorithmic framework which adaptively selects the ``best'' training samples in terms of numerical stability and accuracy. In contrast to our work, \cite{guo2022bayesian,mcquarrie2025active} are tailored to the purpose of learning PDE solution operators, the Bayesian learning problem is formulated on the discretized problem and is hence finite-dimensional and no theoretical guarantees are given.

\paragraph*{Gaussian Process Regression and Kernel Methods}
Minimax contraction rates of posterior distributions using Gaussian processes on classical $s$-smooth Hölder and Sobolev spaces were derived in \cite{van2008rates,van2011information}. In high dimensions, these rates naturally suffer from the curse of dimensionality, which can only be circumvented by assuming more ground truth regularity, e.g.~by assuming that the ground truth lies in the reproducing kernel Hilbert space of a Gaussian process with anisotropic scaling \cite{fasshauer2012dimension} (for finite dimensional input). As Gaussian Process regression (GPR) is equivalent to kernel ridge regression of the corresponding covariance kernel (see Section \ref{sec:kernel}, and, e.g., \cite{kanagawa2018gaussian,alvarez2012kernels} for this connection), the minimax rates obtained for GPR mirror the optimal convergence rates established for regularized least-squares algorithms in reproducing kernel Hilbert spaces \cite{gyorfi2002distribution,kohler2002nonparametric} for finite-dimensional $\Y$.

For the theory of vector-valued reproducing kernel Hilbert spaces with infinite-dimensional output space, see, e.g., \cite{micchelli2005learning,carmeli2006vector,caponnetto2007optimal}. In particular \cite{caponnetto2007optimal} established upper rates for operator-valued kernels in the well-specified case with $\Y$-valued noise, and matching lower bounds for finite-dimensional $\Y$. In contrast to \cite{caponnetto2007optimal}, we allow for the misspecified case, and our lower bound in Theorem \ref{th:minimax} holds for infinite-dimensional $\Y$, see Section \ref{sec:kernel} for details. Upper and lower learning rates for vector-valued ridge regression, also in the misspecified case, were recently obtained in \cite{li2024towards,li2022optimal,meunier2024optimal}; see also \cite{yang2025learning} for learning rates using a regularized stochastic gradient descent scheme. While these results are formulated for multiplicative kernels, our prior covariance is a
trace class operator (similar as in \cite{caponnetto2007optimal}). Neither assumption implies the other, see Section \ref{sec:kernel} for details. Moreover, none of the works \cite{li2024towards,li2022optimal,meunier2024optimal,caponnetto2007optimal} provide explicit learning rates for concrete PDE models.

\subsection{Notation} \label{sec:notation}
Let $\IN = \{1, 2,\dots \}$ and $\IN_0 = \{0,1, 2,\dots \}$. We write $a_n\lesssim b_n$, $a_n\gtrsim b_n$ for real sequences
$(a_n)_{n\in\IN}$, $(b_n)_{n\in\IN}$ if $a_n$ is respectively upper or lower bounded by a positive multiple of $b_n$ and $a_n\simeq b_n$ iff $a_n\lesssim b_n$ and $a_n\gtrsim b_n$.

Let $(\Omega,\Sigma,\IP)$ denote some probability space. For a measure space $(\T,\Sigma_\T,\tau)$, a Hilbert
space $(\Z,\lV\cdot\rV_{\Z})$ and $p\in [1,\infty)$, we define the Bochner space
\begin{align*}
    &L^p(\T,\tau;\Z)
    =
    \left\{
    F:\T\to\Z
    ~\middle|~
    \lVert F\rVert_{L^p(\T,\tau,\Z)}<\infty,~F\text{ is strongly measurable}\footnotemark
    \right\},
    \\
    &\lVert F\rVert_{L^p(\T,\tau;\Z)}^p
    \coloneqq
    \int_{\T}\lVert F(t)\rVert_{\Z}^p\,\dd\tau(t).
\end{align*}
and also write $L^p(\T)$, $L^p(\T,\tau)$ or $L^p(\Z)$ depending on the context. Bochner spaces are complete \cite[Theorem X.4.10]{Amann2008}, especially $L^2(\T,\tau;\Z)$ is a Hilbert space.
\footnotetext{A function $F:\T\to\mathcal Z$ is strongly measurable if there exists a sequence of simple functions converging $\tau$-a.s.\ to $F$, see, e.g., \cite[p.~41]{DiestelUhl1977}.}

Let additionally $\Z$ be measurable with $\sigma$-algebra $\Sigma_{\Z}$ and let $F:\T\to\Z$ be a measurable function. We define the \emph{pushforward} measure $F_\# \tau$ on $\Sigma_{\Z}$ as
$$\left(F_\# \tau\right)(Z)=\tau\left(F^{-1}(Z)\right),~~Z\in\Sigma_\Z. $$
Let $\mu$ be another measure on $(\T,\Sigma_\T)$. We write $\mu\ll\tau$ iff for all $T\in\Sigma_\T$
$$\tau(T)=0\implies \mu(T)=0$$
and say that $\mu$ is \emph{absolutely continuous} w.r.t.~$\tau$. If $\mu\ll\tau$, we define the \emph{Radon--Nikodym derivative} $\dd\mu/\dd\tau$ implicitly as
$$\mu(T)=\int_T \frac{\dd \mu}{\dd\tau} \dd \tau,\qquad T\in\Sigma_\T.$$
This definition is well-defined up to $\tau$-nullsets, see, e.g.,~\cite[Theorem 6.10]{rudin1987real}.
We write $\B(M)$ for the Borel $\sigma$-algebra of a metric space $(M,d)$. We use the shorthand \textit{iff} for ``if and only if''.

For two vector spaces $\X$, $\Y$ we say that $\iota:\X\to \Y$ is an embedding, if $\iota$
is a bounded, linear, and injective map. In this case we also write $\X\overset{\iota}{\hookrightarrow}\Y$.

As usual, $\delta_{j\ell}$ denotes the Kronecker delta.

\section{Problem Formulation}\label{sec:problemformulation}
In the following, let $(\X,\BX)$ and $(\Y,\BY)$ be two Borel measurable, separable Hilbert spaces with inner products $\langle\cdot,\cdot\rangle_{\X}$ and $\langle\cdot,\cdot\rangle_{\Y}$ respectively.

\subsection{Measurement Model}\label{sec:measurenementmodel}
We aim to reconstruct the Borel measurable \emph{ground truth} operator
\[ g_0:\X\to\Y\]
from noisy ``input-output'' pairs 
\noeqref{eq:xiyi1,eq:xiyi2}
\begin{subequations}\label{eq:xiyi}
\begin{equation}\label{eq:xiyi1}
    (\pXX,\pZZ)=(X_i,Z_i)_{i=1}^n,~~ n\in\IN,
    \end{equation}
generated by\footnote{A rigorous form of \eqref{eq:xiyi2}, justified in Lemma \ref{lem:white_noise_construction}, is $Z_i= d\circ g_0(X_i)+E_i$ with the isomorphism $d$ from \eqref{eq:disomorphism}.}
\begin{align}\label{eq:xiyi2}
    Z_i=g_0(X_i)+ E_i,\quad i=1,\dots,n.
\end{align}
\end{subequations}

In \eqref{eq:xiyi},
$X_i\overset{\trm{iid}}{\sim}\IPX$ are random design points, where $\IPX$ is a probability measure on $(\X,\BX)$. 
Moreover, the noise variables $(E_i)_{i=1}^n$ are independent copies of a white noise process. For a separable Hilbert space $\Y$ a white noise process $\WY:\Y\times\Omega\to\IR$ is defined as a centered Gaussian process with ``iso-normal'' covariance structure, meaning
\begin{align}\label{eq:WN}
\WY(y,\cdot)\sim\N\left(0,\lVert y\rVert_{\Y}^2\right),~\Cov\left(\WY(y,\cdot),\WY(y',\cdot)\right)=\langle y,y'\rangle_{\Y},~y,y'\in\Y.
\end{align}
In particular, $y\mapsto \WY(y,\cdot)\in L^2(\Omega,\IP)$ is a linear isometry. The above assumptions define $\WY$ uniquely \emph{in law}, see, e.g.,~\cite[p.~15--16]{Gine.op.2016}.

Crucially, white noise $\WY$ cannot be interpreted as a (Gaussian) $\Y$-valued random variable. Instead, it can be interpreted as a Gaussian random variable on some larger space 
into which $\Y$ is continuously embedded. To set up the problem rigorously, and explain what it means to ``observe'' $Z_i$ in \eqref{eq:xiyi2}, we next give a concrete construction of a version\footnote{We say $\tilde{\mathbb{W}}_{\Y}$ is a version of $\WY$, if for all $y\in \Y$ it holds $\IP\{\omega\in\Omega|\tilde{\mathbb{W}}_{\Y}(y,\omega)=\WY(y,\omega)\}=1$.} of $\WY$.

\subsubsection{Gaussian Measures}\label{sec:gaussianmeasures}
We first recall a few basic definitions and results about Gaussian measures on a separable Hilbert space $\Y$, see, e.g., \cite{bogachev1998gaussian,da2006introduction} for proofs and further details: 
A measure $\mu$ on $(\Y,\B(\Y))$ is called Gaussian if for every $y\in \Y$, the pushforward of $\mu$ under $y'\mapsto \langle y',y\rangle_{\Y}:\Y\to\IR$ is a Gaussian measure on $\IR$. Gaussian measures are uniquely determined via their mean $m=\int_{\Y}y\dd\mu(y)$ and their covariance operator $\mathcal C  :\Y\to \Y$, defined via
\begin{equation}\label{eq:covarianceoperator}
  \langle \mathcal C  y,y'\rangle_{\Y}=\int_{\Y}\langle y,h-m\rangle_{\Y}\langle y',h-m\rangle_{\Y}\dd\mu(h).
\end{equation}
It can be shown that $\mathcal C  $ must be a self-adjoint, positive semi-definite, trace class operator (hence bounded and compact), and for every such operator $\mathcal C  :\Y\to\Y$ and $m\in\Y$, there exists a unique Gaussian measure $\mu$ with mean $m$ and covariance $\mathcal C  $. We thus write $\mu=\N(m,\mathcal C  )$. A measurable map $Y:(\Omega,\Sigma,\IP)\to (\Y,\B(\Y))$ is called a (centered) Gaussian random variable, if its distribution $\IP_Y=Y_\#\IP$ is a (centered, i.e.\ $m=0$) Gaussian measure on $(\Y,\B(\Y))$.

A prominent role is played by the Cameron--Martin space of a Gaussian measure (or of a random variable whose distribution is a Gaussian measure). In the current context it can be defined as the Hilbert space
\begin{equation}
  H\coloneqq \mathcal C^{1/2}\Y=\set{\mathcal C  ^{1/2}y}{y\in\Y}\label{eq:CMspace}
\end{equation}
with inner product
\begin{equation*}
  \dup[H]{y}{y'}\coloneqq \dup[\Y]{\mathcal C  ^{-1/2}y}{\mathcal C  ^{-1/2} y'}.
\end{equation*}
Since $\mathcal C$
  is compact and self-adjoint, there exists an orthonormal basis $(\eta_j)_{j\in\IN}$ of $\Y$ and a non-negative sequence $(\lambda_j)_{j\in\IN}\in\ell^1(\IN)$, such that $\mathcal C  $ has the spectral decomposition
\begin{equation*}
  \mathcal C  y = \sum_{j\in\IN} \lambda_j \langle y,\eta_j\rangle_{\Y}\eta_j.
\end{equation*}
Here, $\mathcal C  ^{1/2}:\Y\to\Y$ is well-defined by taking the square roots of the eigenvalues in the spectral decomposition. Similarly, $\mathcal C^{-1/2}$ is well-defined (as the Moore--Penrose pseudoinverse) on the Cameron--Martin space $H$, via
$$\mathcal C  ^{-1/2}y=\sum_{\set{j}{\lambda_j>0}}\lambda_j^{-1/2}\langle y,\eta_j\rangle_{\Y}\eta_j.$$
\subsubsection{White Noise}\label{sec:whitenoise}
In this section we show in which sense $\Y$-valued white noise can be interpreted as a Gaussian measure on a Hilbert space $\tilde\Y^*$ into which $\Y$ is embedded.

 Let $(\gamma_j)_{j\in\IN}\in\ell^2(\IN)$ be an arbitrary but fixed sequence of positive numbers, and let $(\eta_j)_{j\in\IN}$ be an orthonormal basis of $\Y$. Then
  \begin{align}\label{eq:Ytilde}
    \tilde\Y\coloneqq \biggl\{~\sum_{j\in\IN}c_j\eta_j~\Big|~\sum_{j\in\IN}c_j^2\gamma_j^{-2}<\infty\biggr\},~~\dup[\tilde\Y]{y}{y'}\coloneqq \sum_{j\in\IN}\gamma_j^{-2}\dup[\Y]{y}{\eta_j}\dup[\Y]{y'}{\eta_j}
  \end{align}
  defines a separable Hilbert space, see Lemma \ref{lem:tY}. Moreover
    \begin{equation*}
      \iota: \tilde \Y\to\Y:y\mapsto y
  \end{equation*}
  is evidently an injective, linear, and continuous map (i.e.\ an embedding). Let $\tilde\Y^*$ denote the topological dual space of $\tilde\Y$. Then Lemma \ref{lem:tY} shows that $(\tilde\Y,\Y,\tilde\Y^*)$ is a Gelfand triplet, i.e.~that
\begin{equation*}
  \tilde\Y\overset{\iota}{\hookrightarrow} \Y\simeq\Y^*\overset{\iota^*}\hookrightarrow\tilde\Y^*,~~\iota,\iota^*\text{ are dense embeddings,}
\end{equation*}
where $\iota^*$ denotes the adjoint of $\iota$, and $\simeq$ stands for the identification of $\Y$ and $\Y^*$ via the Riesz isomorphism.
For details on the Gelfand triplet characterization see, e.g.,~\cite{hida1993infinite,gel2014generalized}.
\begin{example}\label{ex:gelfandH1}
Let $\mathbb T$ be the one-dimensional torus, $\Y=L^2(\mathbb T)$ with Fourier ONB and weights $\gamma_j\simeq j^{-s}$, $s>1/2$. Then $\tilde\Y =H^s(\mathbb T)$ (e.g.~\cite[Section 5.3]{saranen2001periodic} for full definitions and details) and Lemma \ref{lem:tY} yields the Gelfand triplet $H^s(\mathbb T)\hookrightarrow L^2(\mathbb T)\hookrightarrow H^{-s}(\mathbb T)$.
\end{example}

In the following we use the dual-pairing notation ${}_{\Y^*}\langle y^*, y\rangle_{\Y}\coloneqq y^*(y)$ for $y^*\in\Y^*$ and $y\in\Y$ and abbreviate $\IW_\Y(y)\coloneqq \IW_\Y(y,\cdot)\in L^2(\Omega,\IP,\IR)$.
Moreover, we denote by $$R_\Y:\Y\to\Y^*$$ the Riesz isomorphism.
We are now in position to formalize the white noise process $\IW_\Y$ within our Hilbert space setting. For a proof of the next lemma see Appendix \ref{app:white_noise_construction}.

\begin{lemma}\label{lem:white_noise_construction}
  Fix $(\eta_j)_{j\in\IN}$, $(\gamma_j)_{j\in\IN}$, and $\tilde\Y$ as in \eqref{eq:Ytilde}.
  Let $\IW_\Y$ be a white noise process on $\Y$ as in \eqref{eq:WN}, and define $\xi_j\coloneqq \IW_\Y(\eta_j)$. Then
  \begin{equation}\label{eq:limtE}
    \lim_{N\to\infty}\sum_{j=1}^N \xi_j (\iota^*\circ R_\Y) \eta_j
  \end{equation}
  converges in $L^2(\Omega,\IP,\tilde\Y^*)$ to a $\tilde\Y^*$-valued centered Gaussian random variable $E:\Omega\to \tilde\Y^*$. The corresponding Cameron--Martin space $H_E\subset\tilde\Y^*$ is isometrically isomorphic to $\Y$ via
    \begin{equation}\label{eq:disomorphism} 
      d:=\iota^*\circ R_\Y:\Y\to H_E.
  \end{equation}

  Moreover, $E$ can be identified with $\IW_\Y$ in the sense that for every $\tilde y\in \tilde\Y$
  \begin{subequations}\label{eq:IWchar}
  \begin{equation}\label{eq:IWchar1}
    \IW_\Y(\tilde y) = {}_{\tilde\Y^*}\langle E,\tilde y\rangle_{\tilde\Y}
  \end{equation}
  with equality in $L^2(\Omega,\IP,\IR)$.
  Additionally, denoting by $P_N:\Y\to {\rm span}\{\eta_1,\dots,\eta_N\} \subseteq\tilde \Y$, $N\in\IN$, the orthogonal projection in $\Y$ onto ${\rm span}\{\eta_1,\dots,\eta_N\}$, it holds for every $y\in\Y$
  \begin{equation}\label{eq:IWchar2}
    \IW_\Y(y)= \lim_{N\to\infty}{}_{\tilde\Y^*}\dup[\tilde\Y]{E}{P_N y},
  \end{equation}
  \end{subequations}
  with convergence in $L^2(\Omega,\IP,\IR)$.

\end{lemma}

Following Lemma \ref{lem:white_noise_construction}, white noise $\IW_\Y$ can be interpreted as the observation of a $\tilde\Y^*$-valued Gaussian random variable.

Let again $H_E\subset\tilde\Y^*$ be the Cameron--Martin space of the $\tilde\Y^*$-valued Gaussian random variable $E$, see \eqref{eq:CMspace}, and assume $\lVert g_0(X_i)\rVert_{L^2(\Omega,\IP,\Y)}<\infty$. Composing $g_0(X_i)\in L^2(\Omega,\IP,\Y)$ with the map $d$ in \eqref{eq:disomorphism}, we can identify it with a random variable in $\tilde\Y^*$; in fact, since $d:\Y\to H_E$ is an isometric isomorphism,
it is $H_E$-valued. We
have thus clarified that the measurement data $(\pXX,\pZZ)$ in
\eqref{eq:xiyi} can be modelled as a random variable in
$\X^n\times(\tilde\Y^{*})^n\coloneqq\pX\times\boldsymbol{\tilde\Y^*}$.

Denote the limit in \eqref{eq:IWchar2} as $\langle\cdot,y\rangle_\sim$. The map $\langle\cdot,\cdot\rangle_\sim$ appears in the log-likelihood in Theorem \ref{th:PosteriorGaussianNoise} below.
Lemma \ref{lem:white_noise_construction} implies the following properties for $\langle\cdot,\cdot\rangle_\sim$ which we summarize in the next corollary.

\begin{corollary}\label{cor:existencetildemapping}
 Consider the setting of Lemma \ref{lem:white_noise_construction}. Define the map
 \begin{align*}
  \langle\cdot,\cdot\rangle_\sim:\tilde\Y^*\times\Y\to\IR,~~
\langle z,y\rangle_\sim=
\begin{cases}
\lim_{N\to\infty}{}_{\tilde\Y^*}\dup[\tilde\Y]{z}{P_N y} & \text{if limit exists}\\
0 & \text{otherwise.}
\end{cases}
 \end{align*}
 Then $ \langle\cdot,\cdot\rangle_\sim$ is Borel measurable. Moreover, for every $y\in\Y$
  there exists $\Omega_y\subseteq\Omega$ with $\IP(\Omega_y)=0$, such that
 \begin{subequations}\label{eq:tildemapping}
 \begin{align}
\langle E(\omega),y\rangle_\sim&=\IW_{\Y}(y,\omega) 
                                 &&\forall \omega\in\Omega\backslash\Omega_y,
                                 \label{eq:tildemapping1}\\
   \langle z,\iota\tilde y\rangle_\sim&={}_{\tilde\Y^*}\langle z,\tilde  y\rangle_{\tilde\Y} &&\forall z\in\tilde\Y^*,~\tilde y\in\tilde\Y.
\label{eq:tildemapping2}
 \end{align}
\end{subequations}
Finally, for all $y\in\Y$, there exists a linear subspace $L_y\subset\tilde\Y^*$ with $\IP_E(L_y)=1$
and such that $\langle\cdot,y\rangle_{\sim}:L_y\to\IR$ is linear.
\end{corollary}
\begin{proof}
  Since the maps $f_N:\tilde\Y^*\times\Y\to\IR$, $f_N(z,y)={}_{\tilde\Y^*}\dup[\tilde\Y]{z}{P_Ny}$,
$N\in\IN$, are continuous and hence Borel measurable, \cite[Lemma 2.1.7]{bogachev2007measure} shows that $\{(z,y) \in \tilde\Y^*\times\Y: \lim_{N\to\infty}f_N(z,y)\in\IR\}$ is a Borel set. Then \cite[Theorem 2.1.5 (v)]{bogachev2007measure} yields Borel measurability of $\langle\cdot,\cdot\rangle_\sim$. For \eqref{eq:tildemapping1} let $y\in\Y$. On the one hand we have for $\omega\in\Omega$,
  \begin{equation*}
    \langle E(\omega),y\rangle_{\sim}=\begin{cases}
\lim_{N\to\infty}{}_{\tilde\Y^*}\dup[\tilde\Y]{E(\omega)}{P_N y} & \text{if limit exists}\\
0 & \text{otherwise},
\end{cases}
  \end{equation*}
  and on the other hand we have $\lim_{N\to\infty}{}_{\tilde\Y^*}\langle E(\cdot),P_N y\rangle_{\tilde\Y}=\IW_{\Y}(y,\cdot)$
  in the sense of $L^2(\Omega,\IP,\IR)$ by \eqref{eq:IWchar2}. Therefore
  $${}_{\tilde\Y^*}\langle E(\cdot),P_Ny\rangle_{\tilde\Y}=\sum_{j=1}^N
  {}_{\tilde\Y^*}\langle E(\cdot),\eta_j\rangle_{\tilde \Y} \langle y,\eta_j\rangle_{\Y}
=\sum_{j=1}^N\IW_{\Y}(\eta_j,\cdot) \langle y,\eta_j\rangle_{\Y},
  $$
  where the equality holds in $L^2(\Omega,\IP,\IR)$.
  Since the $\IW_{\Y}(\eta_j,\cdot)$, $j\in\IN$, are i.i.d.\ standard Gaussian random variables, see \eqref{eq:WN}, it holds
  $$
  \sum_{j\in\IN}\operatorname{Var}[
  \IW_{\Y}(\eta_j,\cdot)
  \langle y,\eta_j\rangle_\Y]=\sum_{j\in\IN}\langle y,\eta_j\rangle_\Y^2=\lVert y\rVert_\Y^2<\infty,\quad y\in\Y.$$
  By Kolmogorov's convergence criterion, e.g.,~\cite[Chapter 6, Theorem 5.2]{gut2012probability}, we can conclude that $\lim_{N\to\infty}{}_{\tilde\Y^*}\dup[\tilde\Y]{E(\omega)}{P_N y}$ converges a.s.~so that \eqref{eq:tildemapping1} holds. Moreover, \eqref{eq:tildemapping2} follows since for $\tilde y\in\tilde\Y$,
  \begin{equation*}
    \left|{}_{\tilde\Y^*}\dup[\tilde\Y]{z}{P_N \tilde y}-{}_{\tilde\Y^*}\dup[\tilde\Y]{z}{\tilde y}\right|\leq \lVert z\rVert_{\tilde\Y^*} \lVert \tilde y-P_N \tilde y\rVert_{\tilde\Y}\overset{N\to\infty}{\longrightarrow} 0.
  \end{equation*}
Finally, for $y\in\Y$, the map $\langle\cdot,y\rangle_{\sim}$ is linear (as a limit of linear maps) on the linear subspace
\begin{equation*}
    L_y\coloneqq\{z\in\tilde\Y^*:~~\lim_{N\to\infty}{}_{\tilde\Y^*}\dup[\tilde\Y]{z}{P_N y}\text{ exists}\}
\end{equation*}
and $\IP_E(L_y)=1$ since the sum converges $\IP_E$-a.s.~following \eqref{eq:tildemapping1}.
\end{proof}

We emphasize that \eqref{eq:tildemapping1} shows in particular that $\langle E,y\rangle_{\sim}$ does not depend on the choice of $(\gamma_j)_{j\in\IN}$ and $(\eta_j)_{j\in\IN}$ (and thus $\tilde\Y^*$) in \eqref{eq:Ytilde}. 

\begin{example}\label{ex:sequenceexample}
    The meaning of the mappings introduced in this section is best understood
    in the sequence spaces: Identify $y\in\Y$ with its coefficient sequence $\Bc=(c_j)_{j\in\IN}=(\langle y,\eta_j\rangle_{\Y})_{j\in\IN}$ and denote (with $\pgamma=(\gamma_j)_{j\in\IN}\in\ell^2$ a positive sequence as before)
    \begin{equation*}
      \ell_{\pgamma^{-1}}^2=\set{\Bc}{\sum_{j\in\IN}c_j^2\gamma_j^{-2}<\infty},\qquad
      \ell_{\pgamma}^2=\set{\Bc}{\sum_{j\in\IN}c_j^2\gamma_j^{2}<\infty}.
    \end{equation*}
    It is easily checked that $(\ell_{\pgamma^{-1}}^2)^*$ can be
    identified with $\ell_{\pgamma}^2$, $(\ell_{\pgamma}^2)^*$ with
    $\ell_{\pgamma^{-1}}^2$, and of course $(\ell^2)^*$ with $\ell^2$,
    which we do in the following. Then
    \begin{equation*}
      \tilde\Y\simeq \ell_{\pgamma^{-1}}^2\subsetneq \Y\simeq\ell^2\subsetneq \tilde\Y^*\simeq\ell_{\pgamma}^2.
    \end{equation*}
    The three maps (interpreted in the sequence spaces)
    \begin{equation*}
      \iota:\ell_{\pgamma^{-1}}^2\to\ell^2,\qquad
      R_{\Y}:\ell^2\to\ell^2,\qquad
      \iota^*:\ell^2\to \ell_{\pgamma}^2
    \end{equation*}
    all act as the identity on the sequence spaces, only the topology changes.
    In particular $d=\iota^*\circ R_\Y:\ell^2\to \ell_{\pgamma}^2$, so that
    \begin{equation*}
      \|\Bc\|^2_{\ell^2}=\sum_{j\in\IN}c_j^2,\qquad
      \|d(\Bc)\|^2_{\ell_{\pgamma}^2}=\sum_{j\in\IN}\gamma_j^2c_j^2,
    \end{equation*}
    and clearly $d$ is bounded linear and compact. With $H_E\simeq\ell^2\subsetneq\ell_{\pgamma}^2$,
    $d:\Y\simeq\ell^2\to H_E\simeq\ell^2$ is an isometric isomorphism.
  \end{example}

\begin{remark}
  In Bayesian nonparametrics, absolute continuity of the shifted noise measure (which is needed for a non-degenerate likelihood) requires the shifts to lie in the Cameron--Martin space $H_E$ of the noise, see, e.g.,~\cite[Theorem 32]{dashti2013bayesian}. 
  We work with
white noise in $\Y$ because then $H_E\simeq\Y$ via \eqref{eq:disomorphism}, so that the shifts may range over all of $\Y$. Noise distributed according to a Gaussian measure is a special case which can be obtained by replacing
$\tilde\Y^*\to\Y$ and $\Y\to H_E$.
\end{remark}

\subsection{Existence of Posterior Distribution}\label{sec:exposterior}
We aim for Bayesian inference on the ground truth $g_0$ by conditioning on the data $(\pXX,\pZZ)=(X_i,Z_i)_{i=1}^n$ and thereby consider a setup similar to the learning of linear operators in \cite{de2023convergence}.

Throughout the rest of this paper
  \begin{equation}\label{eq:GtopVec}
    \calG\subseteq\set{g:\X\to\Y}{g\text{ is Borel measurable}}
  \end{equation}
  is a Banach space equipped with the Borel $\sigma$-algebra $\B(\calG)$.
Given a random variable $G$ with prior distribution $G\sim\IP_G$ on $(\calG,\B(\calG))$, we want to define the \emph{posterior} distribution
$\IP_{G|(\pXX,\pZZ)}$ conditioned on the data. 
We 
start by recalling the notion of \emph{regular conditional distributions}, see, e.g.,~\cite[Definitions 8.25 and 8.28]{klenke2008probability}, which allows to rigorously define such conditionals.
\begin{definition}\label{def:regcond}
Let $\mathcal U$ and $\mathcal V$ be separable Banach spaces equipped with their Borel $\sigma$-algebras $\B(\mathcal U)$ and $\B(\mathcal V)$, and let $U:\Omega\to\mathcal U$ and $V:\Omega\to\mathcal V$ be two RVs with distributions $\IP_U$ and $\IP_V$, respectively. A \emph{regular conditional distribution of $U$ given $V$} is a map $\IP_{U|V}:\B(\mathcal U)\times\mathcal V\to\IR$ s.t.
\begin{enumerate}[label=(\roman*)]
    \item \label{item:conditional1} $\IP_{U|V}(A,\cdot)$ is Borel measurable for all $A\in\mathcal B(\mathcal U)$,
    \item\label{item:conditional2} $\IP_{U|V}(\cdot,v)$ is a probability measure on $(\mathcal U,\mathcal B(\mathcal U))$ for all $v\in\{V(\omega):\omega\in\Omega\}$,
    \item\label{item:conditional3} $\IP(U\in A,V\in B)=\int_B\IP_{U|V}(A,v)\dd\IP_V(v)$ for all $A\in\mathcal B(\mathcal U)$ and $B\in\mathcal B(\mathcal V)$.
\end{enumerate}
\end{definition}
Uniqueness of regular conditional distributions only holds up to
  $\IP_V$ null-sets: if $D$, $\tilde D:\B(\mathcal{U})\times\mathcal{V}\to\IR$ are two regular conditional
  distributions, then there exists a $\IP_V$ null-set
  $N_V\in\B(\mathcal{V})$ such that $D(\cdot,v)=\tilde D(\cdot,v)$
  (i.e.\ these are the same measures on $\mathcal{U}$) for all $v\in N_V^c$, e.g.,
  \cite[Lemma 3.5.23]{ScheichlZech2021}. For brevity, we do not mention this at every
  instance. In particular: if 
  $D$ satisfies the conditions
  of Definition \ref{def:regcond}, we will say that $D$ is \emph{the}
  conditional distribution of $U$ given $V$, or we will simply write
  $D=\IP_{U|V}$, i.e.\ this equality is understood in the above
  sense.

\subsubsection{General Existence Theory}\label{sec:generalexistence}
 To establish an existence result for the posterior we follow notation and generalize arguments from \cite[Section 3]{dashti2013bayesian} and \cite{stuart2010inverse}. 
 The works \cite{stuart2010inverse,dashti2013bayesian} study the Bayesian approach from an inverse problem perspective: In terms of the measurement model \eqref{eq:xiyi}, they fix $g_0$, put a prior measure on $\pXX$ and condition on the data $\pZZ$ resulting in the posterior measure $\IP_{\pXX|\pZZ}$. In contrast to their approach we ``randomize'' the operator by putting a prior 
 on $\mathcal G\ni g_0$ and condition on $(\pXX,\pZZ)$ resulting in the posterior $\IP_{G|(\pXX,\pZZ)}$.

\begin{assumption}[Data generation]\label{ass:data}   Let $n\in\IN$ and let
 \begin{enumerate}[label=(\alph*)]
\item\label{item:embedding} $\X$, $\Y$ and $\Z$ be
  separable Hilbert spaces with Borel $\sigma$-algebras $\B(\X)$, $\B(\Y)$ and $\B(\Z)$ and 
  let\footnote{Later $\Z=\tilde\Y^*$ and $d=\iota^*\circ R_{\Y}$ via the Gelfand triplet, see Section \ref{sec:measurenementmodel}; for now we keep the setting general.} $d:\Y\to \Z$ be a measurable map,
    \item  $\IP_G$ be a prior measure on $(\calG,\B(\calG))$ with the Banach space $\calG$ as in \eqref{eq:GtopVec},
    \item the data $(\pXX,\pZZ)=(X_i,Z_i)_{i=1}^n\in \X^n\times \Z^{n}\eqqcolon \pX\times\pZ$ be distributed according to
    \begin{align}\label{eq:YGX}
      Z_i=d\circ G(X_i)+\Ei,\quad i=1,\dots, n 
    \end{align}
    with $G\sim\IP_G$, $\Ei\overset{\trm{iid}}{\sim}\IPE$ and $X_i\overset{\trm{iid}}{\sim}{\IPX}$, all of them independent.
  \end{enumerate}
\end{assumption}

  \paragraph*{Joint distribution of $(G,\pXX,\pEE)$:} Since $G$, $X_i$, $E_i$ are all independent random variables, their joint distribution on $\calG\times\pX\times\pZ$ with the product $\sigma$-algebra
$\B(\calG)\otimes\B(\pX)\otimes\B(\pZ)$
  is simply
  \begin{equation*}
    \IP_{(G,\pXX,\pEE)}=\IP_G\otimes \IP_{\pXX}\otimes\IP_{\pEE},
  \end{equation*}
  where $\IP_{\pXX}=\otimes_{i=1}^n \IP_{X}$ and $\IP_{\pEE}=\otimes_{i=1}^n\IP_{E}$.

  \paragraph*{Joint distribution of $(G,\pXX,\pZZ)$:}
  To guarantee that $\pZZ$ is a well-defined random variable, we make
  the following assumption.
  \begin{assumption}[Measurability]\label{ass:meas}
    The map
      \begin{equation}
        \calG\times\X\to\Y:(g,x)\mapsto g(x)
      \end{equation}
      is Borel measurable w.r.t.\ the product $\sigma$-algebra on $\calG\times\X$.
  \end{assumption}

  With
  \begin{equation}\label{eq:T}
        T:\calG\times\pX\times\pZ\to \calG\times\pX\times\pZ:(g,\px,\pz)\mapsto (g,\px, d \circ g(\px)+\pz),
      \end{equation}
      where the application of $d\circ g$ is understood componentwise
      and $n$ (cp.~Assumption \ref{ass:data}) will always be clear from context,
      it holds $(G,\pXX,\pZZ)=T(G,\pXX,\pEE)$.
      Assumption \ref{ass:meas} directly yields measurability of $T$
      (since addition is measurable on separable Borel Hilbert
      spaces), which implies $(G,\pXX,\pZZ)$ to be a well-defined
      random variable with distribution
      \begin{equation}\label{eq:GXZ}
        \IP_{(G,\pXX,\pZZ)}=T_\sharp\IP_{(G,\pXX,\pEE)}.
      \end{equation}

      \paragraph*{Likelihood $\pZZ|(G,\pXX)$:}
      Set for $z\in\Z$
      \begin{equation}\label{eq:Tz}
        T_z:\Z\to\Z:e\mapsto z+e.
      \end{equation}
      We then have $Z_i=T_{d\circ G(X_i)}(E_i)$. The natural candidate for the distribution
      of $Z_i$ given $(g,x_i)$ is thus
      \begin{equation*}
        (T_{d\circ g(x_i)})_{\sharp}\IP_E.
      \end{equation*}

      We wish to formalize this, and describe the distribution in terms
      of its Radon--Nikodym derivative w.r.t.\ $\IP_E$. This motivates
      the next assumption.   
      \begin{assumption}[Absolute continuity of noise shifts]\label{ass:generalassump}
        Let
\begin{enumerate}[label=(\alph*)]
\item\label{item:general1} for $(\IP_G\otimes\IP_{X})$-a.e. $(g,x)\in\calG\times\X$
      \begin{equation*}
        (T_{d\circ g(x)})_\sharp\IP_E\ll \IP_E,
      \end{equation*}
      and define for each such $(g,x)$ and for $\IP_E$-a.e.\ $z\in\Z$
      \begin{equation*}
        \Phi(g,x,z):=-\log \frac{\dd (T_{d\circ g(x)})_\sharp\IP_E}{\dd \IP_E}(z),
      \end{equation*}
  \item\label{item:general2} the map thus defined
      \begin{equation}
        \Phi:\calG\times\X\times\Z\to\IR\cup \{\pm\infty\}
      \end{equation}
      be jointly Borel-measurable, i.e.\ there exists a
      jointly Borel-measurable $\Phi$ satisfying \ref{item:general1}. 
\end{enumerate} 
\end{assumption}
Here and below we denote
\begin{equation}\label{eq:pPhi}
  \pPhi(g,\px,\pz)=\sum_{i=1}^n\Phi(g,x_i,z_i),\qquad g\in\calG,~\px\in\pX,~\pz\in\pZ.
\end{equation}
The proof of the next lemma can be found in Appendix \ref{sec:lemma:GXZ}.

\begin{lemma}\label{lemma:GXZ}
  Let Assumptions \ref{ass:data}--\ref{ass:generalassump} hold. Then
  \begin{enumerate}[label=(\roman*)]
  \item\label{item:ZGX} there exists a regular conditional distribution
    $\IP_{\pZZ|(G,\pXX)}(\cdot,(g,\px))$ and it holds for $(\IP_G\otimes\IP_{\pXX})$-a.e.\ 
    $(g,\px)\in\calG\times \pX$ that $\IP_{\pZZ|(G,\pXX)}(\cdot,(g,\px))\ll \IP_{\pEE}$ and
    \begin{equation*}
      \frac{\dd \IP_{\pZZ|(G,\pXX)}(\cdot,(g,\px))}{\dd\IP_{\pEE}}= \exp(-\pPhi(g,\px,\cdot)),
    \end{equation*}
  \item\label{item:GXZ} $\IP_{(G,\pXX,\pZZ)}$ is well-defined, it holds $\IP_{(G,\pXX,\pZZ)}\ll\IP_{(G,\pXX,\pEE)}$ and
    \begin{equation*}
      \frac{\dd \IP_{(G,\pXX,\pZZ)}}{\dd\IP_{(G,\pXX,\pEE)}}= \exp(-\pPhi).
    \end{equation*}
  \end{enumerate}
\end{lemma}

\paragraph*{Posterior $G|(\pXX,\pZZ)$:}
With the joint distribution given, we now have a natural candidate for the posterior.
The next theorem formally verifies that it's a regular conditional distribution
under the given assumptions; its proof is in Appendix \ref{app:existenceposterior}.

\begin{theorem}[Existence of Posterior]\label{th:existenceposterior}
  Let Assumptions \ref{ass:data}--\ref{ass:generalassump} hold,
  and let for $\IP_{(\pXX,\pEE)}$-a.e.\ $(\px,\pz)$
\begin{align}\label{eq:Zxy}
    \mathfrak Z(\px,\pz)\coloneqq\int_{\calG}\exp\left(-\pPhi(g,\px,\pz)\right)\dd\IP_G(g)>0. 
\end{align}
Then there exists a regular conditional distribution
$\IP_{G|(\pXX,\pZZ)}$ such that for $\IP_{(\pXX,\pZZ)}$-a.e.\ $(\px,\pz)$ 
\begin{align}
    \frac{\dd \IP_{G|(\pXX,\pZZ)}(\cdot,(\px,\pz))}{\dd\IP_G}(g)=\frac{1}{\mathfrak Z(\px,\pz)}\exp\left(-\pPhi(g,\px,\pz)\right)\label{eq:generalposteriorform}.
\end{align}
\end{theorem}

\begin{remark}\label{rmk:measureequivalence}
Theorem \ref{th:existenceposterior} and its proof (cf.~\eqref{eq:Zxy} and \eqref{eq:XZXEequivalence}) shows $\dd\IP_{(\pXX,\pZZ)}/\dd\IP_{(\pXX,\pEE)}(\px,\pz)>0$ for $\IP_{(\pXX,\pEE)}$-a.e.~$(\px,\pz)$. Since Lemma \ref{lemma:GXZ} \ref{item:GXZ} yields $\IP_{(G,\pXX,\pZZ)}\ll \IP_{(G,\pXX,\pEE)}$ which implies $\IP_{(\pXX,\pZZ)}\ll \IP_{(\pXX,\pEE)}$, for example \cite[p.~179]{bogachev2007measure} shows that $\IP_{(\pXX,\pZZ)}$ and $\IP_{(\pXX,\pEE)}$ are equivalent measures. Thus all $\IP_{(\pXX,\pEE)}$-a.s.~statements can be exchanged with $\IP_{(\pXX,\pZZ)}$-a.s.~statements in the setting of Theorem \ref{th:existenceposterior}.
\end{remark}

\begin{remark}\label{rmk:singularnoise}
Suppose $T$ is Borel-measurable but $(T_{d\circ g(x)})_\sharp\IP_E\ll\IP_E$ is not true
(i.e.\ Assumption \ref{ass:meas} holds but Assumption \ref{ass:generalassump} not).
Then the joint distribution $\IP_{(G,\pXX,\pZZ)}=T_\# \IP_{(G,\pXX,\pEE)}$ is a probability measure on $\calG\times\pX\times\pZ$. If $\calG$ is Polish, so is the product space, and the disintegration theorem \cite[Theorem~10.4.14]{bogachev2007measurevol2} yields a posterior $\IP_{G|(\pXX,\pZZ)}$ as a $\IP_{(\pXX,\pZZ)}$-a.s.~unique regular conditional distribution in the sense of Definition~\ref{def:regcond}. However, without absolute continuity, see Assumption~\ref{ass:generalassump}~\ref{item:general1}, the posterior can in general not be represented as a Radon--Nikodym density w.r.t.~the prior $\IP_G$. 
\end{remark}

\subsubsection{Additive White Noise for Continuous Operators on Compact Sets}
We now specify the setting to an important generic problem class:
  additive white noise, continuous operators, and approximation on compact sets.
Our goal is to apply Theorem \ref{th:existenceposterior} to show existence
of posteriors.

Denote by $\supp(\IPX)$ the (topological) support of $\IPX$
\begin{align*}
   \supp(\IPX)\coloneqq \left\{x\in\X\,|\,\IPX(N_x)>0~ \text{for all open neighborhoods } N_x\in\BX\text{ of }x\right\}.
\end{align*}
For a compact set $K\subset \X$, we denote the space of all bounded, continuous operators from $K$ to $\Y$ by $C^0(K)\coloneqq \CzKY$. Then %
$(C^0(K),\lVert\cdot\rVert_\infty)$ with the supremum norm $\lVert g\rVert_\infty\coloneqq \sup_{x\in K}\lVert g(x) \rVert_\Y$ is separable and complete, see \cite[Theorem 4.19]{kechris1995classical} and \cite[p.~97--98]{treves2016topological}.
\begin{setting}\label{set:gaussiansetup} 
  Let $\X$, $\Y$ be two separable Hilbert spaces, $\tilde\Y$ as in \eqref{eq:Ytilde}, $d:\Y\to\tilde\Y^*$ as in \eqref{eq:disomorphism}, and
  \begin{equation*}
    Z_i = d\circ G(X_i)+E_i\in\tilde\Y^*,\qquad i=1,\dots,n
  \end{equation*}
  where $G$, $X_i$, $E_i$ are all independent random variables such that:
  \begin{enumerate}[label=(\alph*)]
  \item\label{item:Gaussiansetupnoise} {\bf white noise}:
  $\Ei\overset{\rm iid}{\sim}\IP_E$ for a centered Gaussian measure on $(\tilde\Y^*,\B(\tilde\Y^*))$ with covariance operator $\mathcal{C}:\tilde\Y^*\to\tilde\Y^*$ and eigen\-values $(\gamma_j^2)_{j\in\IN}$ $\in\ell^1(\IN)$.  
\item\label{item:gaussiansetupnu} {\bf design points in compact set}:
    $X_i\overset{\rm iid}{\sim}\IP_X$ for a 
  Borel probability measure on $\X$ with compact support $K\coloneqq\supp(\IPX)\subseteq\X$. 
    
\item\label{item:GaussianG} {\bf continuous operator class}:
  $G\sim\IP_G$ for a Borel probability measure on $\calG=C^0(K,\Y)$
  endowed with the supremum norm $$\lV g\rV_\infty=\sup_{x\in K}~\lVert   g(x)\rVert_{\Y},~~g\in\calG.$$
\end{enumerate}
\end{setting}

Item \ref{item:Gaussiansetupnoise} formalises the white noise model introduced in Section \ref{sec:whitenoise}. Assuming $\calG=C^0(K)$ in \ref{item:GaussianG} ensures a continuous and hence Borel measurable translation map $T$ (cf.~Assumption \ref{ass:meas}).%

The next theorem establishes well-definedness of posteriors; its proof can be found in Appendix \ref{app:proofposteriorgaussian}.
\begin{theorem}[Posterior under white noise]\label{th:PosteriorGaussianNoise}
 Consider Setting \ref{set:gaussiansetup} with $\pXX=(X_i)_{i=1}^n$, $\pZZ=(Z_i)_{i=1}^n$.
  Then the posterior $\IP_{G|(\pXX,\pZZ)}$ is well-defined.
  It is given for $\IP_{(\pXX,\pZZ)}$-a.e.\ $(\px,\pz)\in\pX\times\bs{\tilde{\Y}^*}$
  via \eqref{eq:generalposteriorform}, with negative log-likelihood
  \begin{align}\label{eq:phigeneral}
        \pPhi(g,\px,\pz)=\sum_{i=1}^n\Big(\frac{1}{2}\lVert g(x_i)\rVert_{\Y}^2-\dup[\sim]{z_i}{g(x_i)}\Big)
    \end{align}  
and $\dup[\sim]{\cdot}{\cdot}$ as in Corollary \ref{cor:existencetildemapping}. 
\end{theorem}

\begin{remark}\label{rmk:compactness}
  Compactness of $K$ in Setting \ref{set:gaussiansetup} \ref{item:gaussiansetupnu} is assumed to ensure separability of $C^0(K)$, see \cite[Theorem 4.19]{kechris1995classical}. This assumption on $K$ can be weakened to $K\subseteq\X$ being a closed subset:
  for a bounded continuous ONB $(\psij)_{j\in\IN}$ of $L^2(K)$ as in Section \ref{sec:Prior} below, $$\calG_{\ppsi}\coloneqq\overline{\operatorname{span}\{\psij:j\in\IN\}}^{\,\lV\cdot\rV_\infty}\subseteq C^0(K)$$ is then a separable Banach space and Theorem \ref{th:PosteriorGaussianNoise}
remains valid with $\calG$ in Setting \ref{set:gaussiansetup} \ref{item:GaussianG} replaced by $\calG_{\ppsi}$.
\end{remark}

\subsection{Prior Distribution}\label{sec:Prior}
Let $\ppsi=(\psi_j)_{j\in\IN}$ be a continuous ONB of $L^2(K)= L^2(K,\IP_X;\Y)$ with compact $K=\supp(\IP_X)\subset\X$.
The prior distribution $\IP_G$ should encode a priori knowledge of the ground truth $g_0$.
Frequently used priors include for example Gaussian and uniform priors,
e.g.\ \cite[Chapter 2]{dashti2013bayesian}. 
\begin{example}[Gaussian Prior]\label{ex:GaussianPrior}
  Consider $\IP_G$ to be a Gaussian measure on
  $(L^2(K),\B(L^2(K)))$ with covariance operator
\begin{align}\label{eq:L2Gaus}
     \mathcal{C}_{\rm{prior}}:L^2(K)\to L^2(K),\quad
    \mathcal{C}_{\rm{prior}}(g)=\sumj \kappaj^2\langle g,\psij\rangle_{L^2(K)}~\psij
\end{align}
where $(\kappaj)_{j\in\IN}\in\ell^2(\IN)$.
\end{example}
\begin{example}[Uniform Prior]\label{ex:UniformPrior}
    Let $\pi=\bigotimes_{j\in\IN}\Unif[-1,1]$ be the uniform distribution on $U=[-1,1]^{\IN}$, see, e.g.,~\cite[Section 3.5]{bogachev2007measure}. 
    For $\pkappa\in\IR^{\IN}$, we define
    \begin{align}
        &\sigma_{\pkappa}: U\to L^2(K),\qquad \sigma_{\pkappa}(\boldsymbol{a})=\sumj a_{j}\kappaj\psij\label{eq:uniformkappa}
        \end{align}
        which is a well-defined, continuous and hence Borel measurable mapping iff $(\kappaj)_{j\in\IN}\in\ell^2(\IN)$. Thus, setting $\IP_G=(\sigma_{\pkappa})_{\#}\pi$, we get a well-defined (uniform) prior measure on $(L^2(K),\B(L^2(K)))$.
\end{example}

 In order to apply Theorem \ref{th:PosteriorGaussianNoise}, the support of $\IP_G$ has to lie in $C^0(K)\subset L^2(K)$ and hence continuity has to be %
 enforced into the priors from Examples \ref{ex:GaussianPrior} and \ref{ex:UniformPrior}. As the next lemma shows (and as is well-known)
 this can be done by assuming continuity of the $\psi_j$ and sufficient decay of the sequence $\kappaj\lVert\psi_j\rVert_{\infty}$;
 for the proof of the Lemma see Appendix \ref{app:decayholo}.
\begin{lemma}\label{lem:decay}
  Let $\IP_G$ be as in Example \ref{ex:GaussianPrior} or Example
  \ref{ex:UniformPrior} for certain $\psi_j\in C^0(K)$ and define
 \begin{align}
    \label{eq:omegaj}
    \omegaj\coloneqq\sup_{x\in K}\lVert\psij(x)\rVert_{\Y}.
\end{align}
Assume that $\sum_{j\in\IN}\kappaj\omegaj<\infty$ and $\kappa_j>0$, $j\in\IN$.
Then draws $G$ from $\IP_G$ are a.s.\ bounded continuous operators, i.e.~$\IP_G(C^0(K))=1$,
and $\IP_G$ induces a Borel probability measure on $C^0(K)$. 
\end{lemma}

\section{Posterior Analysis for Gaussian Model}\label{sec:contraction}
In this section we analyze the posterior mean, the maximum a posteriori and the maximum likelihood (ML) estimators for the white noise model from Section \ref{sec:whitenoise} with Gaussian prior from Example \ref{ex:GaussianPrior}. We do not have convergence results for the uniform prior from Example \ref{ex:UniformPrior} as the expression for the posterior mean gets more involved in that case.

\subsection{Posterior Mean Estimator}\label{sec:existencealgebraicform}
We define the posterior mean $\hatgn$ as
 $$\hatgn:\pX\times\bs{\tilde\Y^*}\to\calG,~~\hatgn(\px,\pz)=\IE_{G\sim\IP_{G|(\pXX,\pZZ)}(\cdot,\px,\pz)}\left[G\right].$$ We start by showing that $\hatgn$ is well-defined. The proof of Lemma \ref{lem:posteriorexpectationexistence} can be found in Appendix \ref{app:proofexistenceGnhat}.

\begin{lemma}[Existence of Posterior Expectation]
\label{lem:posteriorexpectationexistence}
Consider Setting \ref{set:gaussiansetup}, and let $(\psij)_{j\in\IN}$ be a continuous ONB of $L^2(K)$. Consider either the Gaussian prior from Example \ref{ex:GaussianPrior} or the uniform prior from Example \ref{ex:UniformPrior} and let 
\begin{align}\label{eq:strongerdecayproperty}
    \sum_{j\in\IN}\kappa_j\omega_j\max\left\{1,\log(j)^{1/2}\right\}<\infty.
\end{align}
Then the posterior expectation
\begin{align*}
\hatgn(\px,\pz)
=\mathfrak Z(\px,\pz)^{-1}\int_{\calG} g\exp\left(-\pPhi(g,\px,\pz) \right)\dd\IP_G(g),
\end{align*}
with the likelihood $\exp(-\pPhi)$ from Theorem \ref{th:PosteriorGaussianNoise}, exists $(\px,\pz)$ $(\IPpX\otimes\IPpE)$-a.s.~as a $\calG$-valued Bochner integral. Moreover, $\hatgn:\pX\times\bs{\tilde\Y^*}\to\calG$ is Borel measurable.
\end{lemma}
In the following we give an algebraic expression for $\hatgn$ in the case of the Gaussian prior from Example \ref{ex:GaussianPrior}. The resulting formula generalizes the finite-dimensional expression for the posterior mean in conjugate Gaussian linear models, see, e.g.,~\cite[Section~2.1]{rasmussen2006gaussian}. The proof of Lemma \ref{lem:calcgaussiangnhat} can be found in Appendix \ref{app:proofgnhatgaussian}.

\begin{lemma}\label{lem:calcgaussiangnhat}
  Consider Setting \ref{set:gaussiansetup} with the Gaussian prior from Example \ref{ex:GaussianPrior}
 built with a fixed continuous ONB $(\psij)_{j\in\IN}$ of $L^2(K)$. Let the decay property \eqref{eq:strongerdecayproperty} hold, and define, for $N\in\IN$,
    \begin{align*}
      \AN&=(A_{N,j\ell})_{j,\ell=1}^N:\pX\to\IR^{N\times N},\\ \bN &=(b_{N,j})_{j=1}^N:\pX\times\bs{\tilde\Y^*}\to\IR^N
    \end{align*}
with components (cp.~Corollary \ref{cor:existencetildemapping})
\begin{align}
A_{N,j\ell}(\px)&=\alphajl(\px)+\deltajl n^{-1}\kappaj^{-2},\qquad
\alphajl(\px)=\frac{1}{n}\sumin\langle\psij(x_i),\psi_\ell(x_i)\rangle_{\Y},\label{eq:alphajk}\\
b_{N,j}(\px,\pz)&=\frac{1}{n}\sumin\dup[\sim]{z_i}{\psij(x_i)}.\label{eq:betanuj}
\end{align}

Then it holds $(\px,\pz)$ $(\IPpX\otimes\IPpE)$-a.s.
    \begin{align}\label{eq:hatgngaussianprior}
    \hatgn(\px,\pz)=\limNinfty\psiNt\AN^{-1}(\px)\bN(\px,\pz)
    \end{align}
    with convergence in $C^0(K)$ where $\psiN=(\psij)_{j=1}^N\in L^2(K)^N$.
\end{lemma}
Having calculated the specific form of $\hatgn$ in Lemma \ref{lem:calcgaussiangnhat} for the Gaussian prior, we are able to formulate a convergence theorem. We therefore assume a ground truth $g_0\in\calG$ and consider $\bs{\tilde\Y^*}$-data distributed according to $\pZZ\sim d\circ g_0(\pXX)+\pEE$, where $\pXX\sim\IPpX$ and $\pEE\sim\IPpE$. 
The proof of Theorem \ref{th:convergencetheorem} can be found in Appendix \ref{app:proofconvergencetheorem}.

\begin{theorem}\label{th:convergencetheorem}
    Consider Setting \ref{set:gaussiansetup} with the Gaussian prior from Example \ref{ex:GaussianPrior} and let $g_0\in C^0(K)$. Moreover, consider a continuous ONB $(\psij)_{j\in\IN}$ of $L^2(K)$ and let
\begin{equation}
\omegaj=\lVert\psij\rVert_{C^0(K)}\lesssim j^\rho,\qquad \rho\geq0.\label{eq:omegaassumppolynomial}
\end{equation} 
Moreover, let $\Gzj\coloneqq\langle g_0,\psij\rangle_{L^2(K)}$, $j\in\IN$, $\qq>1/2$, $\rr>1+2\rho$ and 
    \begin{align}
        |\Gzj|\lesssim j^{-\qq},~~\kappaj\simeq j^{-\rr}.\label{eq:rtassump}
    \end{align}
Define the sequence $(\delta_n)_{n\in\IN}$ with 
\begin{align}
\delta_n=\max\{1,\log(n)\}~n^{-\frac{2\min\{\qq,\rr-2\rho\}-1}{4\rr}},\qquad n\in\IN.
    \label{eq:deltan}
\end{align}
Then there exists a constant $C>0$, depending on $g_0$ only through $\lVert g_0\rVert_\infty$ and the hidden constant in \eqref{eq:rtassump},
s.t.~for all $n\in\IN$ and $\delta\geq \delta_n$ 
\begin{subequations}\label{eq:convergencegaussian}
\begin{fleqn}[3em]
\begin{align}
 &\left(\IPpX\otimes\IP_{\pEE}\right)\left(\left\lVert\hatgn(\pXX,d\circ g_0(\pXX)+\pEE)-g_0\right\rVert_{L^2(K)}^2\geq\delta^2\right)\notag
\end{align}
\end{fleqn}
\vspace{-\abovedisplayskip}
\begin{empheq}[left={\leq\empheqlbrace}]{alignat=2}
    &C\exp\left(-C^{-1}n^{(2\rr-4\rho-1)/(2\rr)}\delta^2\right),
      &&\;\;\delta_n\leq\delta\leq n^{-\frac{1}{4\rr}},\label{eq:convergencegaussian1}\\
    &C\exp\left(-C^{-1}n^{(2\rr-4\rho-2)/(2\rr)}\right),
      &&\;\;\max\{\delta_n,n^{-\frac{1}{4\rr}}\}\leq\delta\leq n^{3/2},\label{eq:convergencegaussian2}\\
    &C\exp\left(-C^{-1}\delta^2n^{-(2\rr+2\rho+1)/\rr}\right),
      &&\;\;\delta\geq n^{3/2}\label{eq:convergencegaussian3}.
\end{empheq}
\end{subequations}
\end{theorem}
Theorem \ref{th:convergencetheorem} provides non-asymptotic concentration bounds for the posterior mean $\hatgn$ around the ground truth $g_0$. The concentration results are stated in three regimes: 
\begin{enumerate}[label=(\roman*)]
\item A small-deviation regime in \eqref{eq:convergencegaussian1} which gives sub-Gaussian tails for small deviations. This regime determines the sequence $\delta_n$ via Hoeffding bounds of the dominating components. Note that due to \eqref{eq:deltan}, this case only occurs if $\min\{q,r-2\rho\}>1$.
\item A ``plateau'' regime in \eqref{eq:convergencegaussian2} for moderately large deviations $\delta$, where the exponential bound does not depend on $\delta$. The reason is a Neumann series argument used to bound $\bs{A}_N^{-1}$ in the proof of Theorem \ref{th:convergencetheorem}. This leads to a $\delta$-independent condition which dominates in the plateau regime. 
\item A high-deviation regime in \eqref{eq:convergencegaussian3} using a rough bound on $\AN^{-1}$  to ensure integrability for $\delta\to\infty$.
\end{enumerate} 

The convergence rate $\delta_n$ depends on the regularity of the ground truth (via the parameter $\qq$) and the regularity of the prior (via the parameter $\rr$). 
The fluctuations of the basis functions, described by the parameter $\rho$ in \eqref{eq:omegaassumppolynomial}, reduce the prior regularity to an effective regularity $\rr'=\rr-2\rho$.  
Choosing the ``optimal'' effective smoothness $\rr'=\qq$ for the prior maximizes the convergence rate in \eqref{eq:deltan}.

\begin{remark}\label{rmk:logarithmiccase}
Assume that instead of \eqref{eq:omegaassumppolynomial}, the logarithmic bound $\omega_j\lesssim \log(j)$, $j\geq 2$, holds. Then, similar to the proof of Theorem \ref{th:convergencetheorem}, one can show that the concentration bounds in \eqref{eq:convergencegaussian} hold with $\rho=0$, $\delta_n=\log(n)^3 n^{-(2\min\{\qq,\rr\}-1)/(4\rr)}$ and $C$ replaced by $\log(n)^4C_0$, $n\geq 2$, where $C_0>0$ is a constant. 
\end{remark}

Integrating the concentration inequality from Theorem \ref{th:convergencetheorem} shows a convergence result in expectation, see the following corollary.
\begin{corollary}\label{cor:Expectationconvergence}
Consider the setting of Theorem \ref{th:convergencetheorem}. Then there exists a constant $C>0$ depending on $g_0$ only through $\lVert g_0\rVert_\infty$ and the hidden constant in \eqref{eq:rtassump},
s.t.~for all $n\in\IN$
    \begin{align}
\IE_{(\pXX,\pEE)}\left[\left\lVert\hatgn(\pXX,d\circ g_0(\pXX)+\pEE)-g_0\right\rVert_{L^2(K)}^2\right]\leq C\delta_n^2.\label{eq:convergencegaussianIE}
    \end{align}
\end{corollary}
\begin{proof} 
 We use \eqref{eq:convergencegaussian} to estimate
    \begin{align*}
  &\IE_{(\pXX,\pEE)}\left[\left\lVert\hatgn(\pXX,d\circ g_0(\pXX)+\pEE)-g_0\right\rVert_{L^2(K)}^2\right]\notag\\
=~&\int_0^\infty\left(\IPpX\otimes\IP_{\pEE}\right)\left(\left\lVert\hatgn(\pXX,d\circ g_0(\pXX)+\pEE)-g_0\right\rVert_{L^2(K)}^2\geq t\right)\dd t\notag\\
\lesssim~& \delta_n^2+\int_{\delta_n^2}^{\max\{\delta_n^2,n^{-1/(2\rr)}\}} \exp\left(-C^{-1}n^{(2\rr-4\rho-1)/(2\rr)} t\right)\dd t \notag\\
&\quad +\int_{\max\{\delta_n^2,n^{-1/(2\rr)}\}}^{n^{3}}\exp\left(-C^{-1}n^{(2\rr-4\rho-2)/(2\rr)}\right)\dd t+\int_{n^{3}}^\infty \exp\left(-C^{-1} t n^{-(2\rr+2\rho+1)/\rr}\right)\dd t\notag\\
\lesssim~& \delta_n^2+\underbrace{Cn^{-\frac{2\rr-4\rho-1}{2\rr}}}_{\lesssim\delta_n^2}\underbrace{\exp\left(-C^{-1}n^{(2\rr-4\rho-1)/(2\rr)}\delta_n^2\right)}_{\leq 1}\notag\\
&\quad+\underbrace{n^{3}\exp\left(-C^{-1}n^{(2\rr-4\rho-2)/(2\rr)}\right)}_{\lesssim\delta_n^2}+\underbrace{Cn^{\frac{2\rr+2\rho+1}{\rr}}\exp\left(-C^{-1}n^{(\rr-2\rho-1)/\rr}\right)}_{\lesssim\delta_n^2}\lesssim \delta_n^2
    \end{align*}
    which shows the claim.
\end{proof}

In Section \ref{sec:minimax} below we show that the rate in Corollary \ref{cor:Expectationconvergence} is minimax optimal up to logarithmic factors for $\rho=0$ and $\rr=\qq$.

The proof strategy of Theorem \ref{th:convergencetheorem} can be used to obtain concentration results in the $C^0(K)$-norm; see Appendix \ref{app:C0Kconvergence} for details. For $\rho=0$ in \eqref{eq:omegaassumppolynomial} the squared convergence rate equals (up to logarithmic factors) $\delta_n^2=n^{-(2\min\{\qq,\rr\}-2)/(2\rr)}$, see \eqref{eq:deltanC0}. In particular, this does not recover the minimax rate $(\log(n)/n)^{2k/(2k+d)}$ in $(C^k(\IT^d),\lVert\cdot\rVert_\infty)$ (e.g.~\cite[Theorem 1]{stone1982optimal}) for $k,d\in\IN$,  $k>d/2$. This suboptimality is an artefact of the coefficientwise estimate of the $C^0(K)$-norm used in the proof of Theorem \ref{th:concentrationC0}. Getting sharp bounds in the supremum norm requires different techniques (e.g.~\cite{stone1982optimal,tsybakov2009introduction}) and is left for future work.

\subsection{Maximum a Posteriori Estimator}\label{sec:MAPestimator}
Another popular estimator in the Bayesian framework is the so-called \emph{maximum a posteriori (MAP)} estimator. The infinite-dimensional case was first discussed in the inverse problem setting, see \cite{dashti2013bayesian,dashti2013map}. We follow \cite{ayanbayev2022gamma} in this section.
\begin{definition}\label{def:MAP}
Let $(V,\lVert\cdot\rVert)$ be a normed space and $\mu$ a probability measure on $V$.
Then, for $\delta>0$ and $x\in V$, denote by $B_\delta(x)$ the $\delta$-ball around $x$ and define the functional
\begin{align}\label{eq:Jdeltafunctional}
 J_\delta: V\to\IR,\quad J_{\delta}(x)=\mu(B_{\delta}(x)).
\end{align}
We %
call %
$x_{\text{MAP}}\in V$ with $\lim_{\delta\to 0}\frac{J_{\delta}(x_{\text{MAP}})}{\sup_{x\in V} J_{\delta}(x)}=1$ \emph{maximum a posteriori (MAP)} estimator of $\mu$.
\end{definition}

There exists a general relationship (e.g.~\cite[Proposition 4.1]{ayanbayev2022gamma}) between MAP estimators and minimizers of the so-called \emph{Onsager--Machlup (OM) functional}. Let therefore $\emptyset\neq M\subseteq \supp(\mu)\subseteq V$ and $I:M\to\IR$. Then $I$ is an Onsager--Machlup functional for $\mu$ if
\begin{align}
    \lim_{\delta\to 0}\frac{\mu\left(B_\delta(x_1)\right)}{\mu\left(B_\delta(x_2)\right)}=\exp\left(I(x_2)-I(x_1)\right),\quad x_1,x_2\in M.\label{eq:OMfunctional}
\end{align}
Under assumptions on $M$ (see \cite[Remark 3.2 and Lemma B.1]{ayanbayev2022gamma}) all MAP estimators are minimizers of the Onsager--Machlup functional $I$.\footnote{By \cite[Proposition 4.1]{ayanbayev2022gamma} the minimizers of $I$ are exactly the \emph{global weak modes} (see \cite[Definition 3.7]{ayanbayev2022gamma}) of $\mu$.} 

We denote by
\begin{equation}\label{eq:Hprior}
  H_{\rm{prior}}\subset \calG
\end{equation}
  the CM space of the (Gaussian) prior $\IP_G$ from Example \ref{ex:GaussianPrior}. In order to define the OM functional of the Gaussian posterior, we require a representation of the cross term
$\dup[\sim]{z_i}{g(x_i)}$ in \eqref{eq:phigeneral} as a bounded linear functional
on $H_{\rm{prior}}$. We give this representation in Lemma \ref{lem:abstractlmap} below; for a proof, see Appendix \ref{app:abstractlmap}.

\begin{lemma}\label{lem:abstractlmap}
Consider the continuous ONB $(\psi_j)_{j\in\IN}$ from Example
\ref{ex:GaussianPrior} and let the decay property \eqref{eq:strongerdecayproperty} hold.
Then, $(x,z)$ $(\IPX\otimes\IPE)$-a.s., the set
\begin{align}\label{eq:lmapdomain}
  \Lambda_{x,z}\coloneqq\Bigl\{g\in L^2(K):
  \sum_{j=1}^\infty\bigl|\langle g,\psi_j\rangle_{L^2(K)}
  \langle z,\psi_j(x)\rangle_\sim\bigr|<\infty\Bigr\}
\end{align}
is a linear subspace of $L^2(K)$ with $H_{\rm{prior}}\subseteq\Lambda_{x,z}$ and
$\IP_G(\Lambda_{x,z})=1$, and the map
\begin{align}\label{eq:lmap}
  \ell_{x,z}:\Lambda_{x,z}\to\IR,\quad
  \ell_{x,z}(g)=\sum_{j=1}^\infty\langle g,\psi_j\rangle_{L^2(K)}\langle z,\psi_j(x)\rangle_\sim
\end{align}
is a well-defined linear functional which is bounded on $H_{\rm{prior}}$.
Moreover, $H_{\rm{prior}}\subseteq C^0(K)$ and for every $g\in \Lambda_{x,z}$
there exists a Borel set $N_{g}\subset\X\times\tilde\Y^*$ with $(\IPX\otimes\IPE)(N_{g})=0$ such that
\begin{align}\label{eq:lmapidentity}
  \ell_{x,z}(g)=\dup[\sim]{z}{g(x)},\quad (x,z)\in\X\times\tilde\Y^*\setminus N_{g}.
\end{align}
\end{lemma}
We are now able to classify MAP estimators of the posterior $\IP_{G|(\pXX,\pZZ)}$ and do so in Theorem \ref{th:MAPestimator} below; for a proof, see Appendix \ref{app:MAPestimator}.

\begin{theorem}\label{th:MAPestimator}
Consider Setting \ref{set:gaussiansetup} with the Gaussian prior from Example \ref{ex:GaussianPrior}. Moreover, let the decay property \eqref{eq:strongerdecayproperty} hold. Then the posterior $\IP_{G|(\pXX,\pZZ)}$ has for $(\px,\pz)$ $(\IPpX\otimes\IPpE)$-a.s.~the (Tikhonov) OM functional
\begin{align}
   I: \mathcal G\to \IR\cup\{\infty\},~~I(g)=\begin{cases}
    \frac{1}{2}\left(\sum_{i=1}^n\left(\lVert g(x_i)\rVert_\Y^2-2\ell_{x_i,z_i}(g)\right)+\lVert g\rVert_{H_{\rm{prior}}}^2\right), &g\in H_{\rm{prior}},\\
   \infty,&\text{otherwise}
   \end{cases}\label{eq:tikhonovOM}
\end{align}
with $\ell_{x,z}$ from \eqref{eq:lmap}. %
{Moreover, the posterior mean $\hatgn(\px,\pz)$ from Lemma \ref{lem:calcgaussiangnhat} is the unique MAP estimator of $\IP_{G|(\pXX,\pZZ)}(\cdot,\px,\pz)$ in the sense of Definition \ref{def:MAP}.}
\end{theorem}
Theorem \ref{th:MAPestimator} shows that the posterior mean and the MAP estimator coincide in the Gaussian case. In particular Theorems \ref{th:convergencetheorem} and \ref{th:concentrationC0} yield algebraic convergence rates for the MAP estimator.  

\subsection{Maximum Likelihood Estimator}\label{sec:MLestimator}
We conclude this section with a short discussion of the \emph{maximum likelihood (ML)} estimator.
If it exists, the maximum likelihood estimator $g_{n}^{\text{ML}}$ of the measurement model \eqref{eq:xiyi} with white noise minimizes the negative log-likelihood $\pPhi$ from \eqref{eq:phigeneral}, i.e.
\begin{align*}
   g_{n}^{ML}(\px,\pz)\in\argmin_{g\in\calG}\pPhi(g,\px,\pz)=\argmin_{g\in\calG} \sum_{i=1}^n\frac{1}{2}\lVert g(x_i)\rVert_{\Y}^2-\langle z_i, g(x_i)\rangle_{\sim}. 
\end{align*}
Convergence analysis for the expectation
$$\IE_{(\pXX,\pEE)}\left[\left\lVert g_{n}^{ML}(\pXX,d\circ g_0(\pXX)+\pEE)-g_0\right\rVert_{L^2(K)}^2\right]$$
in the case of compact operator classes $\calG$ can be found in \cite[Corollary 7]{reinhardt2026statistical}. In that work, it is shown that the convergence rate (in terms of $n$) depends on both the entropy of $\calG$ (which measures how ``big'' $\calG$ is) and the bias of $\calG$ towards $g_0$ (which measures how well $\calG$ approximates $g_0$). Establishing convergence theorems for the ML estimator in the case of non-compact regression classes is ongoing work.

\subsection{Vector-Valued Kernel Methods}\label{sec:kernel}
In this section we show the connection of our Bayesian operator learning framework to vector-valued kernel methods, see, e.g.,~\cite{caponnetto2007optimal,micchelli2005learning,carmeli2006vector,li2024towards}. 

Denote by $\mathcal L(\Y)$ the space of all bounded linear operators $A:\Y\to\Y$ and 
define $\odot:\Y\times\Y\to\mathcal L(\Y)$ via $(y\odot y')y''=\langle y',y''\rangle_{\Y}y$ for $y,y',y''\in\Y$. Then we have the following standard characterization of the Cameron-Martin space
as an RKHS:
\begin{proposition}\label{prop:kernel}
Consider Setting \ref{set:gaussiansetup} with the Gaussian prior from
Example \ref{ex:GaussianPrior} with continuous ONB $(\psij)_{j\in\IN}$ of $L^2(K)$ and let
\eqref{eq:strongerdecayproperty} hold. Then
\begin{align}\label{eq:kernel}
  \Gamma:K\times K\to\mathcal L(\Y),\qquad
  \Gamma(x,x')=\sum_{j\in\IN}\kappa_j^2\,\psi_j(x)\odot\psi_j(x')
\end{align}
is the reproducing kernel of the Cameron--Martin space $H_{\rm{prior}}$ (cp.~\eqref{eq:Hprior}).
\end{proposition}
\begin{proof}
  Since $H_{\rm prior}$ is the Cameron-Martin space of the Gaussian measure,
  whose covariance has eigenfunctions $\psi_j$ and eigenvalues $\kappa_j^2$,
  the sequence $(\kappa_j\psi_j)_{j\in\IN}$ forms an ONB of $H_{\rm prior}$.

  Next, we observe that $\Gamma$ is well-defined since $\sum_{j\in\IN}\kappaj^2\omegaj^2<\infty$ by \eqref{eq:strongerdecayproperty}.
For every $x\in K$ and $y\in\Y$, we also have
$\Gamma(\cdot,x)y\in H_{\rm{prior}}$, since
\[
  \|\Gamma(\cdot,x)y\|_{H_{\rm{prior}}}^2
  =\Bigl\lVert\sum_{j\in\IN}\kappa_j^2
    \langle\psi_j(x),y\rangle_{\Y}\psij\Bigr\rVert_{\Hprior}^2
  =\sum_{j\in\IN}\kappa_j^2
    |\langle\psi_j(x),y\rangle_{\Y}|^2
  \leq \|y\|_{\Y}^2
    \sum_{j\in\IN}\kappa_j^2\omega_j^2<\infty.
\]
 Moreover, since $\Hprior\subseteq C^0(K)$ (see Lemma \ref{lem:abstractlmap}) we have the reproducing property
\begin{align*}
  \langle g,\Gamma(\cdot,x)y\rangle_{H_{\rm{prior}}}
  &=\sum_{j\in\IN}\kappa_j^{-2}\langle g,\psi_j\rangle_{L^2(K)}\,\langle\psij,\Gamma(\cdot,x)y\rangle_{L^2(K)}\\
    &=\sum_{j\in\IN}\kappa_j^{-2}\langle g,\psi_j\rangle_{L^2(K)}\,\kappa_j^{2}\langle\psi_j(x),y\rangle_{\Y}\\
  &=\Bigl\langle \sum_{j\in\IN}\langle g,\psi_j\rangle_{L^2(K)}\psi_j(x),y\Bigr\rangle_{\Y}=\langle g(x),y\rangle_{\Y},\qquad g\in H_{\rm{prior}},~~x\in K,~~y\in\Y
\end{align*}
which shows the claim.
\end{proof} 
By Theorem \ref{th:MAPestimator} and \cite[p.~9]{li2024towards}, the maximum a posteriori estimator (and the posterior mean) coincides\footnote{In \cite{li2024towards}, the noise is $\Y$-valued, which gives a regularized least-squares objective. In contrast, we assume a white noise model leading to the modified least-squares objective in \eqref{eq:tikhonovOM} by formally subtracting the infinite term $\frac 12 \sum_{i=1}^n\lVert z_i\rVert_\Y^2$, see also \cite[Section 2.1.1]{reinhardt2026statistical}.} with the (vector-valued) kernel ridge regression estimator over $H_{\rm{prior}}$ with regularization parameter $\lambda=1/n$.
In the following we compare the kernel in \eqref{eq:kernel} with the literature on vector-valued kernel ridge regression.

\begin{enumerate}[label=(\roman*)]
\item {\bf Multiplicative kernels:} The kernel in \eqref{eq:kernel} has a different structure from the multiplicative kernels underlying the analysis in \cite{li2024towards,li2022optimal,yang2025learning}. In these works, the structure
$\Gamma=KT$ with a scalar-valued kernel $K$ and a positive semi-definite self-adjoint operator $T\in\mathcal L(\Y)$ is assumed (with $T=\Id_{\Y}$ in \cite{li2022optimal,yang2025learning}). Such a multiplicative structure holds in \eqref{eq:kernel} if the basis $(\psi_j)_{j\in\IN}$ factorizes as $\psi_{(k,\ell)}=\varphi_{k}\eta_\ell$, where $(\varphi_{k})_{k\in\IN}$ is an ONB of $L^2(K,\IPX;\IR)$ and $(\eta_\ell)_{\ell\in\IN}$ is an ONB of $\Y$, and the weights $(\kappa_j)_{j\in\IN}$ factorize as $\kappa_{(k,\ell)}=s_kt_\ell$ for sequences $(s_k)_{k\in\IN}$ and $(t_\ell)_{\ell\in\IN}$. On the other hand, in the present work the operator $\mathcal C_{\rm{prior}}$ is trace class. Hence, for example, the choice $T=\Id_{\Y}$, permitted in
\cite{li2024towards,li2022optimal,yang2025learning}, is not admissible in our setting if $\dim(\Y)=\infty$.
\item {\bf Trace-class kernels:} The analysis in \cite{caponnetto2007optimal} allows general trace-class operator-valued kernels $\Gamma$, i.e.~kernels with $\sup_{x\in K}\operatorname{tr}(\Gamma(x,x))<\infty$, which includes \eqref{eq:kernel}, since
\[
\operatorname{tr}(\Gamma(x,x))=\sum_{j\in\IN}\kappa_j^2\lVert\psij(x)\rVert_\Y^2\leq\sum_{j\in\IN}\kappa_j^2\omega_j^2<\infty
\]
by \eqref{eq:strongerdecayproperty}. However, \cite{caponnetto2007optimal} consider only the well-specified case $g_0\in H_{\rm{prior}}$; this corresponds to the choice $\rr<\qq-1/2$ in our setup. In that case, the upper bounds in \cite{caponnetto2007optimal} are better than ours. The reason is that \cite{caponnetto2007optimal} choose the regularization parameter $\lambda=\lambda_n(\qq,\rr)$ depending on $\qq$ and $\rr$ (the regularity of $g_0$ and the eigenvalue decay of the kernel, respectively), whereas our setup corresponds to $\lambda=1/n$, independent of $\qq$ and $\rr$. Moreover, \cite{caponnetto2007optimal} only have lower bounds for the case $\dim(\Y)<\infty$,
also see Section \ref{sec:minimax} ahead.
\end{enumerate}
\begin{remark}\label{rmk:mercer}
Under suitable assumptions, certain positive semi-definite trace-class kernels $\Gamma:K\times K\to\mathcal L(\Y)$ can be shown to be of the form \eqref{eq:kernel} by Mercer's theorem for operator-valued kernels (see, e.g.,~\cite[Proposition 4]{carmeli2010vector}).
\end{remark}

\subsection{\texorpdfstring{Minimax Optimality}{Minimax Optimality}}\label{sec:minimax}
Let $\qq>1$ and $(\psij)_{j\in\IN}$ be a continuous ONB of $L^2(K)$. We consider the hyperrectangle
    \begin{align*}
        \Theta_{\qq}\coloneqq\left\{g\in L^2(K):~|\langle g,\psij\rangle_{L^2(K)}|\leq j^{-\qq}~\text{for all }j\in\IN\right\}.
    \end{align*}
We now show that the rate in Corollary \ref{cor:Expectationconvergence} is minimax optimal up to logarithmic factors over $\Theta_q$ if $\rho=0$ in \eqref{eq:omegaassumppolynomial} and the prior decay matches the decay of the expansion coefficients of $g_0$, i.e.~$\rr=\qq$ in \eqref{eq:rtassump}.  For a proof of Theorem \ref{th:minimax}, see Appendix \ref{app:lowerbound}.

\begin{theorem}\label{th:minimax}
Consider Setting \ref{set:gaussiansetup}, let $\qq>1$ and let $(\psij)_{j\in\IN}$ be a continuous ONB of $L^2(K)$ with $M\coloneqq\sup_{j\in\IN}\lVert\psij\rVert_{C^0(K)}<\infty$, i.e.~\eqref{eq:omegaassumppolynomial} holds with $\rho=0$.
\begin{enumerate}[label=(\roman*)]
\item\label{item:minimaxupper} Consider the Gaussian prior from Example \ref{ex:GaussianPrior} with prior decay $\rr=\qq$ in \eqref{eq:rtassump}. Then there exists a constant $C>0$ such that for all $n\in\IN$
    \begin{align}\label{eq:upperboundhyperrect}
        \sup_{g_0\in\Theta_{\qq}}\IE_{(\pXX,\pEE)}\left[\lVert\hatgn(\pXX,d\circ g_0(\pXX)+\pEE)-g_0\rVert_{L^2(K)}^2\right]\leq C\max\{1,\log(n)\}^2n^{-\frac{2\qq-1}{2\qq}}.
    \end{align}
\item\label{item:minimaxlower} There exists a constant $c>0$ such that for all $n\in\IN$
    \begin{align}\label{eq:lowerboundhyperrect}
        \inf_{\tilde g_n}\sup_{g_0\in\Theta_{\qq}}\IE_{(\pXX,\pEE)}\left[\lVert\tilde g_n(\pXX,d\circ g_0(\pXX)+\pEE)-g_0\rVert_{L^2(K)}^2\right]\geq c\,n^{-\frac{2\qq-1}{2\qq}},
    \end{align}
    where the infimum ranges over all measurable estimators $\tilde g_n:\pX\times\bs{\tilde\Y^*}\to L^2(K)$.
\end{enumerate}
\end{theorem}
The upper and lower bounds from \eqref{eq:upperboundhyperrect} and \eqref{eq:lowerboundhyperrect} show that the rate in \eqref{eq:deltan} cannot be improved beyond logarithmic factors for $\rho=0$ and $\rr=\qq$. 

\section{Applications}\label{sec:applications}
In this section we apply the Bayesian operator framework to the learning of holomorphic operators and show that Theorem \ref{th:convergencetheorem} implies posterior mean convergence for the solution operator of the Darcy flow PDE.

\subsection{Bayesian Learning of Holomorphic Operators}\label{sec:holomorphicoperators}
 We consider a setup similar to \cite{reinhardt2026statistical,herrmann2024neural}. Let $(\phi_j)_{j\in\IN}$ and $(\eta_j)_{j\in\IN}$ be ONBs of the Hilbert spaces $\X$ and $\Y$ respectively.
 We start by introducing the Hilbert spaces $\Xalpha$ and $\Ybeta$ for $\alpha,\beta\geq 0$ as
 \begin{align}
        \X^{\alpha}\coloneqq\{x\in\X:\,\lVert x\rVert_{\X^{\alpha}}<\infty\}\quad\text{and}\quad
        \Ybeta\coloneqq\{y\in\Y:\,\lVert y\rVert_{\Ybeta}<\infty\}
    \end{align}
    where
\begin{align}
    \lVert x\rVert_{\X^{\alpha}}^2\coloneqq\sum_{j\in\IN}\langle x,\phi_j\rangle_{\X}^2 \theta_j^{-2\alpha}\qquad\text{and}\qquad
    \lVert y\rVert_{\Ybeta}^2\coloneqq\sum_{j\in\IN}\langle y,\eta_j\rangle_{\Y} ^2 \theta_j^{-2\beta}
\end{align}
and $(\theta_j)_{j\in\IN}$ with
\begin{equation}\label{eq:XrYt}
  \theta_j\lesssim j^{-1}
\end{equation}
is positive and monotonically decreasing.
With the scalar product 
\begin{align}\label{eq:scalarXt}
    \langle y,y'\rangle_{\Ybeta}\coloneqq\sum_{j\in\IN}\theta_j^{-2\beta}\langle y,\eta_j\rangle_{\Y}\langle y',\eta_j\rangle_{\Y}
\end{align}
the spaces $\Ybeta$ (and similarly $\X^{\alpha}$) become Hilbert spaces, see \cite[Lemma 1]{herrmann2024neural}. These spaces exhibit ``extra smoothness'' which is encoded via higher summability of the expansion coefficients w.r.t.~their ONBs. With $U=[-1,1]^\IN$, $\ssv>1/2$ and a scaling factor $R>0$, we define
\begin{align}\label{eq:sigmarr}
\sigmaRr:&~U\to\X,~~\pu\mapsto R\sumj \theta_j^{\ssv}u_j\phij.
\end{align}
The map $\sigmaRr$ is well-defined, since %
$(\phij)_{j\in\IN}$ is an orthonormal basis
and $\ssv>1/2$ yields $\left(\thetaj^{\ssv} u_j\right)_{j\in\IN}\in\ell^{2}(\IN)$ for all $\pu\in U$. Following \cite{herrmann2024neural} we further define the ``cubes''
\begin{align}
    \CRr(\X)\coloneqq \sigmaRr(U)\subset\X,
\end{align}
which encode ``input decay'' w.r.t.~the ONB $(\phij)_{j\in\IN}$. Note that $\CRr(\X)$ is compact in $\X$ for all $s>1/2$ and $R>0$, since the set $\bigtimes_{j\in\IN}[-R\thetaj^s,R\thetaj^s]$ is compact in $\ell^2(\IN)$. 

In the following we will make use of the property that bounded holomorphic operators with sufficient input decay have $\ell^p$-summable ($p<1$) expansion coefficients w.r.t.~general polynomial chaos-based $L^2$-bases, see \cite{cohen2010convergence,chkifa2015breaking,CDS11}.
Similar to \cite{reinhardt2026statistical,herrmann2024neural}, we first define an ONB of $\LtwoCrsIPxYE$ based on multivariate Legendre polynomials.

Denote by $\Lj$ the $L^2([-1,1],\Unif[-1,1])$-normalized $j$-th Legendre polynomial. Let $\F$ be the set of multiindices with finite support, i.e.~
$$\F=\left\{(\nu_j)\in\IN_0^{\IN}~\big|~|\supp\pnu|<\infty   \right\}$$
and define the multivariate Legendre polynomials as
\begin{align}\label{eq:Lnu}
\Lnu:~U\to\IR,~~\Lnu(\pu)=\prod_{j\in\supp\pnu} L_{\nu_j}(u_j).
\end{align}
By e.g.~\cite[§18.2(iii) and §18.3]{Olver.2010} we obtain the uniform bound
\begin{align}
\label{eq:omeganu}
\omeganu\coloneqq\sup_{\pu\in U}\lvert\Lnu(\pu)\rvert\leq\prod_{j\in\supp\pnu}\left(1+2\nu_j\right)^{1/2}.
\end{align}

We now use \eqref{eq:Lnu} to construct an ONB of $\LtwoCrsIPxYE$. Recall $\pi=\otimes_{j\in\IN}\Unif[-1,1]$. 
For $\pnu\in\F$ and $k\in\IN$, define 
\begin{align}\label{eq:legendrel2basis}
    \psi_{\pnu,k}:~\CRr(\X)\to\Y,~~\psi_{\pnu,k}(x)=\Lnu\circ\left(\sigmaRr\right)^{-1}(x)~\etak.
\end{align}
The following lemma shows that $(\psi_{\pnu,k})_{\pnu\in\F,k\in\IN}$ is an ONB of $\LtwoCrsIPxYE$. For a proof of Lemma \ref{lem:ltwobasis}, see Appendix \ref{app:Ltwobasisproof}.
\begin{lemma}
\label{lem:ltwobasis}
Let $\IPX=(\sigmaRr)_\#\pi$. Then $\left(\psinuk\right)_{\pnu\in\F,k\in\IN}$ in \eqref{eq:legendrel2basis} is an ONB of $\LtwoCrsIPxYE$. 
\end{lemma}
We now recall the (well-known) fact that holomorphic operators have fast decaying expansion coefficients 
\begin{align*}
    \Gznuk\coloneqq \langle g_0,\psinuk\rangle_{\LtwoCrsIPxYE}.
\end{align*}
The following assumption is formulated similarly to \cite[Assumption 1]{herrmann2024neural}. We denote the complexification of a Hilbert space $\X$ by $\X_{\IC}$, see, e.g.,~\cite{munoz1999complexifications}
or \cite[p.~18]{zech2018sparse}.
\begin{assumption}[Holomorphy Assumption]\label{ass:holo}
  For some $M$, $R>0$, $\ssv>1$ and $\ttv>0$, there exists an open set $\OC\subset\XC$ containing $\CRr(\X)$ s.t.
    \begin{enumerate}[label=(\alph*)]
        \item $\sup_{a\in \OC} \lVert g_0(a)\rVert_{\Y_{\IC}^t}\le M$,\label{item:holoassumptionboundedness}
        \item $g_0:\OC\to\Y_{\mathbb C}$ is holomorphic.\label{item:holoassumptionholo}
    \end{enumerate}
\end{assumption}
Under Assumption \ref{ass:holo} one can show the decay property mentioned above, as stated in the next theorem, which is a variant of
\cite[Theorem 2.2.10]{zech2018sparse} or \cite[Theorem
5]{herrmann2024neural}. For details see Appendix
\ref{app:proofholodecay}.

\begin{theorem}[{\cite[Theorem 5 and Proposition 2]{herrmann2024neural}}]\label{th:holodecay}
Let $O_{\IC}\subseteq\X_{\IC}$, $M$, $R>0$, $\ssv>1$ and $\ttv>0$ and let $ p\in(1/\ssv,1)$ and $t'\in(0,t)$. %
Let $(\theta_j)_{j\in\IN}$ 
be as in \eqref{eq:XrYt}.
Then there exists a constant $C>0$ and a sequence $(\anu)_{\pnu\in\F}\in\ell^p(\F)$ of positive numbers
s.t.\ 
\begin{enumerate}[label=(\roman*)]
\item\label{item:holodecaylp}
  for every $g_0$ satisfying Assumption \ref{ass:holo} (with $O_{\IC}$, $M$, $R>0$, $\ssv>1$ and $\ttv>0$) and
      all $\pnu\in\F$ 
    \begin{align*}
        \sumk \theta_k^{-2\ttv'}\Gznuk^2\leq C\anu^2,
    \end{align*}
    \item\label{item:holodecayseq} there exist constants $T_1>1$ and $T_2>0$ s.t.~with 
    $$\tau=\left(\tau_j\right)_{j\in\IN},~~\tau_j\coloneqq\max\left\{T_1,T_2\thetaj^{\ssv(p-1)}\right\}$$
    and $\enu\coloneqq\tau^{\pnu}$ it holds $(\anu\enu)_{\pnu\in\F}\in\ell^1(\F)$ and $(\enu^{-1})_{\pnu\in\F}\in\ell^{\frac{p}{1-p}}(\F)$.
\end{enumerate}
\end{theorem}
We now show that the assumptions for posterior mean convergence, \eqref{eq:omegaassumppolynomial} and \eqref{eq:rtassump} in Theorem \ref{th:convergencetheorem}, hold under Assumption \ref{ass:holo} using a suitable enumeration of $\Gznuk$. For a proof of Lemma \ref{lem:holomakesconvergence}, see Appendix \ref{app:holomakesconvergence}.

\begin{lemma}\label{lem:holomakesconvergence}
    Let $O_{\IC}\subseteq\X_{\IC}$, $M$, $R>0$, $\ssv>1$, $\ttv>0$
    and $\eps>0$. Then
    \begin{enumerate}[label=(\roman*)]
    \item\label{item:g0rate} for every $g_0$ as in Assumption \ref{ass:holo} there exists
      $C>0$ and an enumeration $i:\IN\to\F\times\IN$ (meaning $i$ is bijective) s.t.\
      for all $j\in\IN$
      \begin{equation}\label{eq:g0rate}
        \omegaj=\lVert\psi_{i(j)}\rVert_{C^0(K)}\lesssim  j^{\eps}\qquad\text{and}\qquad
        |g_{0,i(j)}|\lesssim j^{-\min\{\ssv,\ttv+1/2\}+\eps},
      \end{equation}
    \item\label{item:generalrate} there exists $C>0$ and an enumeration
      $i:\IN\to\F\times\IN$ s.t.\ for all
      $g_0$ as in Assumption \ref{ass:holo} and all $j\in\IN$
      \begin{equation}\label{eq:generalrate}
        \omegaj=\lVert\psi_{i(j)}\rVert_{C^0(K)}\lesssim  j^{\eps}\qquad\text{and}\qquad
        |g_{0,i(j)}|\lesssim j^{-\min\{\ssv,\ttv\}+\eps}.
      \end{equation}      
    \end{enumerate}
\end{lemma}

For simplicity we write in the following $\psi_j$, $g_{0,j}$ for $j\in\IN$.
From Lemma \ref{lem:holomakesconvergence} we conclude that holomorphic operators (see Assumption \ref{ass:holo}) can be efficiently learned in a Bayesian framework (Setting \ref{set:gaussiansetup}) with Gaussian prior (Example \ref{ex:GaussianPrior}). We formalize this in the next Corollary and provide two statements: The first has a better convergence, but uses a prior that requires explicit information on $g_0$. The second uses a prior which only depends on the class of operators in Assumption \ref{ass:holo}, but not on the specific $g_0$; for a proof of Corollary \ref{cor:holomorphlearning}, see Appendix \ref{app:holomorphlearning}.

\begin{corollary}\label{cor:holomorphlearning}
  Consider Setting \ref{set:gaussiansetup}, and let $\rr>1$, $\ssv>1$, and $\eps>0$. Let Assumption \ref{ass:holo} hold, and let $\IP_{X}=(\sigma_R^{\ssv})_{\#}\pi$.

    \begin{enumerate}[label=(\roman*)]
    \item\label{item:g0explicitrate} Let $\ttv>0$ and consider the Gaussian prior from Ex.~\ref{ex:GaussianPrior} with the enumeration \ref{item:g0rate}. Then
          \begin{align*}
\IE_{(\pXX,\pEE)}\left[\left\lVert\hatgn(\pXX,d\circ g_0(\pXX)+\pEE)-g_0\right\rVert_{L^2(K)}^2\right]\leq Cn^{-\frac{2\min\{\rr,\ssv,\ttv+1/2\}-1}{2\rr}+\eps}.
    \end{align*}
  \item\label{item:generalexplicitrate} Let $\ttv>1/2$ and consider the Gaussian prior from Ex.~\ref{ex:GaussianPrior} with the enumeration \ref{item:generalrate}. Then
          \begin{align*}
\IE_{(\pXX,\pEE)}\left[\left\lVert\hatgn(\pXX,d\circ g_0(\pXX)+\pEE)-g_0\right\rVert_{L^2(K)}^2\right]\leq Cn^{-\frac{2\min\{\rr,\ssv,\ttv\}-1}{2\rr}+\eps}.
          \end{align*}
    \end{enumerate}
  \end{corollary}

Similarly to Corollary \ref{cor:holomorphlearning} bounds in the $C^0(K)$-norm can be established using Theorem \ref{th:concentrationC0} instead of Corollary \ref{cor:Expectationconvergence}.

\begin{remark}\label{rmk:jacobi}
  Lemma \ref{lem:holomakesconvergence} and Corollary \ref{cor:holomorphlearning} can also be shown using Jacobi instead of Legendre polynomials as $L^2$-basis:
    \begin{enumerate}[label=(\alph*)]
    \item One can construct an ONB $\psi_{\pnu,k}^{\alpha,\beta}$ of $L^2(\CRr(\X),(\sigma_R^s)_{\#}\lambda_{\alpha,\beta},\Y)$ using the probability measures $\lambda_{\alpha,\beta}$ induced by the weight function 
    $$\omega_{\alpha,\beta}:[-1,1]\to\IR,~~ \omega_{\alpha,\beta}(u)=(1-u)^{\alpha}(1+u)^{\beta},\quad \alpha,\beta >-1,$$ for details see \cite[Example 2.1.15]{zech2018sparse}. 
    \item  Similar uniform bounds as in \eqref{eq:omeganu} hold for multivariate Jacobi polynomials, see \cite[§18.3]{Olver.2010}. Thus the weight bounds in Lemma \ref{lem:holomakesconvergence} hold.
    \item Decay of the expansion coefficients $\Gznuk=\langle g_0,\psi_{\pnu,k}^{\alpha,\beta}\rangle_{L^2(K)}$ as in Theorem \ref{th:holodecay} can be shown, see \cite[Theorem 2.2.10]{zech2018sparse}. Thus the coefficient bounds in Lemma \ref{lem:holomakesconvergence} hold.
    \end{enumerate}
\end{remark}
\subsection{Diffusion Equation on the Torus}\label{sec:darcy}
In this section we discuss the application of our results to the solution operator of the diffusion equation. 
The setting discussed here is similar to \cite[Section 7.1]{herrmann2024neural} and \cite[Section 4.2]{reinhardt2026statistical}. 
\noindent 
Let $\IT^2$ denote the $2$-dimensional torus, $\bar{a}\in L^\infty(\IT^2)$ and $f\in C^\infty(\IT^2)/\IR$.
We consider the Darcy Flow PDE and thus aim to find $u\in H^1(\IT^2)$ s.t.~in the weak sense,
\begin{align*}
     - \nabla \cdot ((\bar{a} + a )\nabla u) = f 
\text{ on } \IT^2\qquad
\text{ and }\qquad \int_{\IT^2}u(x)\dd x =0.
\end{align*}
Let $a_{\min}>0$. Then the Lax--Milgram lemma shows well-posedness of the solution operator 
\begin{align} 
g_0:\left\{a\in L^\infty(\IT^2):~\essinf_{x\in \IT^2} (\bar{a}(x)+a(x)) > a_{\min}\right\}\to H^1(\IT^2)\label{eq:darcyg0}.
\end{align}
For a (free) parameter $s_0>0$, we set $\X=H^{s_0}(\IT^2)$ and $\Y=H^1(\IT^2)$. 

Applying elliptic regularity theory (e.g.~\cite{agmon1959estimates}) one can show that if $s_0>1$ the operator $g_0$ from \eqref{eq:darcyg0} satisfies the holomorphy Assumption \ref{ass:holo} for $R>0$,
$$\X^\ssv=H^{s_0+2\ssv}(\IT^2),~~\ssv>1\text{ and }\Y^\ttv=H^{1+2\ttv}(\IT^2),~~\ttv\in [0,(s_0-1)/2),$$ see \cite[Proof of Proposition 3]{herrmann2024neural} for more details.

Defining the input regularity
$\beta\coloneqq s_0+2\ssv$, Corollary \ref{cor:holomorphlearning} 
shows the following proposition.
\begin{proposition}\label{prop:darcy}
Let $R>0$ and $\eps>0$ (arbitrarily small), and set
\begin{enumerate}[label=(\roman*)]
\item\label{item:darcyg0dependent} for the $g_0$-dependent prior, i.e.~with the enumeration from Lemma \ref{lem:holomakesconvergence} \ref{item:g0rate},
\begin{align*}
(s,s_0,\kappa)=\begin{cases}
\left(1+\frac{\eps}{2},\ \beta-2-\eps,\ \frac{\beta-3}{5-\beta}\right) & \text{if } \beta\in(3,4] \text{ and } \eps<\beta-3,\\[1ex]
\left(\frac{\beta}{4},\ \frac{\beta}{2},\ \frac{\beta-2}{2}\right) & \text{if } \beta>4,
\end{cases}
\end{align*}
\item\label{item:darcyg0independent} for the $g_0$-independent prior, i.e.~with the enumeration from Lemma \ref{lem:holomakesconvergence} \ref{item:generalrate},
\begin{align*}
(s,s_0,\kappa)=\begin{cases}
\left(1+\frac{\eps}{2},\ \beta-2-\eps,\ \frac{\beta-4}{6-\beta}\right) & \text{if } \beta\in(4,5] \text{ and } \eps<\beta-4,\\[1ex]
\left(\frac{\beta-1}{4},\ \frac{\beta+1}{2},\ \frac{\beta-3}{2}\right) & \text{if } \beta>5.
\end{cases}
\end{align*}
\end{enumerate}
Moreover, let
\begin{enumerate}[label=(\alph*)]
\item {\bf ground truth}: $g_0:C_R^{s}(H^{s_0}(\IT^2))\to H^{1}(\IT^2)$ be given through \eqref{eq:darcyg0}, with $\bar a$ such that $g_0$ is well-defined,
\item {\bf data}: $\pXX$ be i.i.d.~sampled from $\IP_X=(\sigma_R^{s})_{\#}\pi$ with white noise as in Setting \ref{set:gaussiansetup} \ref{item:Gaussiansetupnoise},
\item {\bf prior}: $\IP_G$ be the linear Gaussian prior from Example \ref{ex:GaussianPrior} with prior decay $\rr=s$ in \eqref{eq:rtassump} and the multivariate Legendre basis from Lemma \ref{lem:ltwobasis}.
\end{enumerate}
Then there exists a constant $C>0$ such that for all $n\in\IN$
\begin{align*}
\IE_{(\pXX,\pEE)}\left[\left\lVert\hatgn(\pXX,d\circ g_0(\pXX)+\pEE)-g_0\right\rVert_{L^2(K)}^2\right]\leq Cn^{-\frac{\kappa}{\kappa+1}+\eps}.
\end{align*}
\end{proposition}
\begin{proof}
Let $\eps'>0$ be arbitrarily small. The proof is an application of Corollary \ref{cor:holomorphlearning}. In case \ref{item:darcyg0dependent} we choose $s_0=\beta-2-\eps$, $\rr=\ssv=1+\eps/2$ and $\ttv=(s_0-1)/2-\eps'$ for $\beta\in(3,4)$, and $s_0=\beta/2$, $\rr=\ssv=\ttv+1/2+\eps'=\beta/4$ for $\beta>4$.

In case \ref{item:darcyg0independent} we choose $s_0=\beta-2-\eps$, $\rr=\ssv=1+\eps/2$ and $\ttv=(s_0-1)/2-\eps'$ for $\beta\in(4,5)$, and $s_0=(\beta+1)/2$, $\rr=\ssv=\ttv+\eps'=(\beta-1)/4$ for $\beta>5$.
\end{proof}
\begin{remark}
  For the $g_0$-dependent prior in \ref{item:darcyg0dependent}, the convergence rate for the posterior mean estimator in Proposition \ref{prop:darcy} is identical to the rate obtained for the maximum likelihood estimator established in \cite[Theorem 27]{reinhardt2026statistical} in case $\beta>4$. For $\beta\in(3,4)$, the restriction $\rr>1$ on the prior in Theorem \ref{th:convergencetheorem} makes our rate suboptimal w.r.t.~\cite[Theorem 27]{reinhardt2026statistical}.\footnote{In \cite[Theorem 27]{reinhardt2026statistical}, a different (suboptimal) $s_0$ was chosen for the case $\beta\in(3,4)$.} The general prior in \ref{item:darcyg0independent} (which does not depend on $g_0$ but only only on the parameters in Assumption \ref{ass:holo}) yields a worse rate. We suspect that this gap can be closed by replacing the pointwise decay condition $|\Gzj|\lesssim j^{-\qq}$ in \eqref{eq:rtassump} by a weaker tail condition.
\end{remark}

\section{Conclusion and Outlook}\label{sec:conclusion}
In this work, we developed a Bayesian framework for learning operator-valued maps from noisy point evaluations and gave a characterization of the posterior measure in a general setting. For Gaussian priors and a white noise measurement model, we derived an explicit representation of the posterior mean, %
which is equal to the maximum a posteriori estimator
and both are equal to ridge regularized kernel regression; we
proved concentration bounds for the convergence towards the ground truth in the large data limit. Finally, we applied
our theory to the learning of the solution operator of the
standard parametric diffusion equation.

So far, we have established concentration bounds for the posterior
mean towards the ground truth in the $L^2(K)$-norm, see Theorem
\ref{th:convergencetheorem}. The same technique gives
  convergence rates in the topology of $C^0(K)$ as we detail
  in Appendix \ref{app:C0Kconvergence}; these estimates are likely not sharp though, and
  deriving sharp convergence statements in this topology is open.
Moreover, our convergence
  results so far only treat the posterior mean. The full Bayesian
  analysis, i.e.\ the derivation of concentration results for the
posterior distribution in the sense of Bernstein--von Mises theorems
are also left for future work; see, e.g.,~\cite{nickl2023bayesian,monard2021statistical} for
corresponding results in the case of inverse problems.

Finally, our convergence analysis so far only treats linear priors, i.e.~priors of the form $\IP_G=T_\# \mu$ where $\mu$ is a probability measure on $\IR^{\IN}$ and $T:\IR^{\IN}\to\calG$ is a linear map. The study of nonlinear (e.g.~neural-network) priors is ongoing research.

\section*{Acknowledgements}
We thank Claudia Strauch for helpful discussions and Mina Chavelle Tschoua for assistance with preliminary numerical investigations, which are not included in the %
manuscript.
\section*{LLM Usage Declaration}
Large language models (LLMs) were used for language polishing and proofreading of the manu\-script. In addition, during the proofreading of Theorem \ref{th:PosteriorGaussianNoise}, Theorem \ref{th:MAPestimator}, Lemma \ref{lem:holomakesconvergence}, Corollary \ref{cor:holomorphlearning} and Theorem \ref{th:minimax}, the models \emph{Claude Opus 5.5} and \emph{Claude Fable 5.1} identified minor technical errors and assisted in formulating corrections. All resulting corrections were independently checked and verified by the authors.

\printbibliography[heading=bibintoc,title=References]

\begin{appendix}
\section{Proofs of Section \ref{sec:problemformulation}}\label{app:proofsproblemformulation}
\subsection{White Noise Construction}\label{app:auxmeaslemmas}

\begin{lemma}\label{lem:tY}
  Let $(\gamma_j)_{j\in\IN}\in\ell^2(\IN)$ be a sequence of positive numbers, and let $(\eta_j)_{j\in\IN}$ be an orthonormal basis of $\Y$. Then
  \begin{equation*}
    \tilde\Y\coloneqq \biggl\{\sum_{j\in\IN}c_j\eta_j\Big|\sum_{j\in\IN}c_j^2\gamma_j^{-2}<\infty\biggr\}
  \end{equation*}
  defines a separable Hilbert space with inner product
  \begin{equation}\label{eq:tYinp}
    \dup[\tilde\Y]{y}{y'}\coloneqq \sum_{j\in\IN}\gamma_j^{-2}\dup[\Y]{y}{\eta_j}\dup[\Y]{y'}{\eta_j}.
  \end{equation}
  The inclusion map $\iota:\tilde\Y\to\Y:y\mapsto y$ is bounded and $\iota\tilde\Y$ is a dense subset of $\Y$. Moreover, $\iota^*:\Y^*\to \tilde\Y^*$ defines a dense embedding.
\end{lemma}
\begin{proof}
  It is clear that $\tilde\Y$ is a vector space, and that
  \eqref{eq:tYinp} defines an inner product on $\tilde\Y$. Define
  $\tau:\Y\to\tilde\Y$ via
  $\tau y \coloneqq \sum_{j\in\IN}\gamma_j\dup[\Y]{y}{\eta_j}\eta_j$. One checks
  directly that $\tau$ is an isometric isomorphism. Since
  $\Y$ is a Hilbert space, this shows that $\tilde\Y$ is complete, and
  thus $\tilde\Y$ is a Hilbert space. Finally note that
  $\set{\sum_{j=1}^Nc_j\eta_j}{c_j\in\IR,~N\in\IN}\subset\tilde\Y$ is
  a dense subset of $\Y$ and also $\iota^*\Y^*$
  is dense in $\tilde\Y^*$. To see the latter, apply \cite[Corollary III.6.14]{conway2019course} and use that $\tilde\Y$ is reflexive.
\end{proof}

\begin{lemma}\label{lem:tYstar}
  Consider $\Brho\coloneqq(\gamma_j)_{j\in\IN}\in\ell^2(\IN)$ and  $(\eta_j)_{j\in\IN}$ as in Lemma \ref{lem:tY}. Define the inner product space
  \begin{equation}\label{eq:tYdual}
    \ell_{\Brho}^2\coloneqq \biggl\{(c_j)_{j\in\IN}\in\IR^{\IN}\Big|\sum_{j\in\IN}c_j^2\gamma_j^2<\infty\biggr\},\quad
    \dup[\ell_{\Brho}^2]{\Bc}{\Bc'}\coloneqq \sum_{j\in\IN}\gamma_j^{2}c_j c_j'
  \end{equation}
  and the map
    \begin{equation}\label{eq:kappa}
    \zeta:\ell_{\Brho}^2\to\tilde\Y^*:\Bc \mapsto \Big(\tilde y\mapsto\sum_{j\in\IN}c_j \dup[\Y]{\tilde y}{\eta_j}\Big).
  \end{equation}

  Then $\ell_{\Brho}^2$ is a Hilbert space and $\zeta:\ell_{\Brho}^2\to\tilde\Y^*$ is an isometric isomorphism.
\end{lemma}
\begin{proof}
  Fix $\Bc\coloneqq(c_j)_{j\in\IN}\in\ell_{\Brho}^2$. Then for all $\tilde y\in\tilde \Y$ we have
  \begin{equation}\label{eq:Hnorm}
    \Big|\sum_{j\in\IN}c_j \dup[\Y]{\tilde y}{\eta_j}\Big|\le\Big(
    \sum_{j\in\IN}(c_j\gamma_j)^2 \sum_{j\in\IN} (\gamma_j^{-1}\dup[\Y]{\tilde y}{\eta_j})^2\Big)^{1/2}
    =\|\Bc\|_{\ell_{\Brho}^2}\|\tilde y\|_{\tilde\Y}.
  \end{equation}
  Hence $\zeta:\ell_{\Brho}^2\to\tilde\Y^*$ is well-defined, clearly injective,
  and $\|\zeta(\Bc)\|_{\tilde\Y^*}\le\|\Bc\|_{\ell_{\Brho}^2}$.

  Set $\tilde z \coloneqq \sum_{j\in\IN} c_j\gamma_j^2 \eta_j$. Due to $\sum_{j\in\IN} (c_j\gamma_j^2)^2\gamma_j^{-2}=\sum_{j\in\IN}c_j^2\gamma_j^{2}<\infty$ it holds $\tilde z\in\tilde \Y$. Plugging into \eqref{eq:Hnorm}, the inequality becomes an equality, which shows $\|\zeta(\Bc)\|_{\tilde\Y^*}\ge \|\Bc\|_{\ell_{\Brho}^2}$. Thus $\zeta$ is an isometry.

  It remains to show that $\zeta$ is surjective. Fix $\varphi\in\tilde\Y^*$ and let $\tau:\Y\to\tilde\Y$, $y\mapsto\sum_{j\in\IN}\gamma_j\dup[\Y]{y}{\eta_j}\eta_j$ be the isometric isomorphism from Lemma \ref{lem:tY}. Then also the dual map $\tau^*:\tilde\Y^*\to\Y^*$ is an isometric isomorphism, and especially $\tau^*\varphi\in\Y^*$. Thus there exists a sequence $(d_j)_{j\in\IN}\in\ell^2(\IN)$ such that for all $y\in\Y$
  \begin{equation*}
    (\tau^*\varphi)y = \sum_{j\in\IN} d_j \dup[\Y]{y}{\eta_j} = \sum_{j\in\IN}d_j \gamma_j^{-1}\dup[\Y]{\tau y}{\eta_j}.
  \end{equation*}
  Hence, setting $b_j\coloneqq d_j\gamma_j^{-1}$, $j\in\IN$, we have $\Bb\in\ell_{\Brho}^2$ and $\varphi$ corresponds to $\zeta(\Bb)$.
\end{proof}

\subsubsection{Proof of Lemma \ref{lem:white_noise_construction}}\label{app:white_noise_construction}
 Throughout this proof, we set
  \begin{equation*}
    \eta_j^*\coloneqq (\iota^*\circ R_\Y)\eta_j \in\tilde\Y^*,
  \end{equation*}
  and note that its action on some $\tilde y \in \tilde\Y$ is given as
  \begin{equation*}
    \eta_j^*(\tilde y) = \iota^*\circ R_\Y\eta_j(\tilde y) = R_\Y\eta_j (\iota\tilde y)
    =\dup[\Y]{\eta_j}{\tilde y}.
  \end{equation*}

  {\bf Step 1.} We first show that \eqref{eq:limtE} converges in $L^2(\Omega,\IP,\tilde\Y^*)$. By definition of a white noise process, the $\xi_j$ are i.i.d.\ $\N(0,1)$ distributed random variables. For any $\tilde y = \sum_{j\in\IN}c_j\eta_j$ with $\|\tilde y\|_{\tilde\Y}^2 = \sum_{j\in\IN} c_j^2\gamma_j^{-2}\le 1$, it holds $\eta_j^*(\tilde y)=c_j$ where $|c_j|\le \gamma_j$. Thus
  \begin{equation}\label{eq:normejst}
    \|\eta_j^*\|_{\tilde \Y^*}=\gamma_j.
  \end{equation}
    Since $(\gamma_j)_{j\in\IN}\in\ell^2(\IN)$, the sequence $\sum_{j\in\IN}\xi_j\eta_j^*$ converges in $L^2(\Omega,\IP,\tilde\Y^*)$, and by Parseval's identity its squared norm equals $\sum_{j\in\IN}\gamma_j^2<\infty$. Hence $E\in L^2(\Omega,\IP,\tilde\Y^*)$ is well defined.

{\bf Step 2.} We show \eqref{eq:IWchar}. As mentioned above, the map $y\mapsto\IW_\Y(y)\in L^2(\Omega,\IP,\IR)$ is a linear isometry (in particular continuous). Hence for every $y\in\Y$
  \begin{equation}\label{eq:IWYconv}
    \IW_\Y(y)=\lim_{N\to\infty} \IW_\Y(P_Ny)=
    \lim_{N\to\infty} \sum_{j=1}^N \dup[\Y]{y}{\eta_j} \IW_\Y(\eta_j)
    =
    \lim_{N\to\infty} \sum_{j=1}^N \dup[\Y]{y}{\eta_j} \xi_j,
  \end{equation}
  with convergence in $L^2(\Omega,\IP,\IR)$.

  Since $E\in L^2(\Omega,\IP,\tilde\Y^*)$ is defined as the $L^2(\Omega,\IP,\tilde\Y^*)$ limit of \eqref{eq:limtE}, we may exchange the (continuous) duality pairing with the limit to obtain for $\tilde y\in\tilde \Y$
  \begin{equation*}
    {}_{\tilde \Y^*}\dup[\tilde\Y]{E}{\tilde y}=
    \lim_{N\to\infty}\sum_{j=1}^N \xi_j \eta_j^*(\tilde y)
    =
    \lim_{N\to\infty}\sum_{j=1}^N \dup[\Y]{\tilde y}{\eta_j}\xi_j.
  \end{equation*}
  Together with \eqref{eq:IWYconv} this shows $\IW_\Y(\tilde y)={}_{\tilde\Y^*}\dup[\tilde\Y]{E}{\tilde y}$ for all $\tilde y\in \tilde\Y$.

  Finally, if $y\in\Y$, then $P_Ny\in\tilde\Y$ and $\lim_{N\to\infty}P_N y = y$ in $\Y$. Thus using again the continuity of $\IW_\Y:\Y\to L^2(\Omega,\IP,\IR)$,
  \begin{equation*}
    \IW_\Y(y)=\lim_{N\to\infty} \IW_\Y(P_N y) = \lim_{N\to\infty}{}_{\tilde\Y^*}\dup[\tilde\Y]{E}{P_N y}
  \end{equation*}
  with convergence in $L^2(\Omega,\IP,\IR)$.

  {\bf Step 3.} Lastly we consider the Cameron--Martin space $H_E$ of the measure induced by $E$. To this end we first note that with $\zeta:\ell^2_{\Brho}\to\tilde\Y^*$ as in Lemma \ref{lem:tYstar} and $\boldsymbol{\delta}_j\coloneqq(\delta_{jk})_{k\in\IN} $,
  \begin{equation*}
    \eta_j^* = \zeta (\boldsymbol{\delta}_j)\quad\text{and thus}\quad E = \sum_{j\in\IN}\xi_j \zeta (\boldsymbol{\delta}_j).
  \end{equation*}

  Therefore the covariance operator $\mathcal C:\tilde\Y^*\to\tilde\Y^*$ satisfies for $\Bb,\Bc,\in\ell^2_{\Brho}$,
  \begin{align*}
    \dup[\tilde\Y^*]{\mathcal C \zeta\Bb}{\zeta\Bc}&=\int_{\Omega} \dup[\tilde\Y^*]{E(\omega)}{\zeta\Bb}
                                  \dup[\tilde\Y^*]{E(\omega)}{\zeta\Bc}\dd\IP(\omega)\\
                                &=\int_\Omega \Big(\sum_{j\in\IN} \xi_j(\omega) b_j\gamma_j^2\Big)\Big(\sum_{i\in\IN} \xi_i(\omega) c_i\gamma_i^2\Big)\dd\IP(\omega)\\
                                &=\sum_{j\in\IN}b_jc_j\gamma_j^4\\
                                    &=\sum_{j\in\IN}\dup[\ell^2_\Brho]{(\gamma_j^{-2}\delta_{jk})_{k\in\IN}}{\Bb}\dup[\ell^2_\Brho]{(\gamma_j^{-2}\delta_{jk})_{k\in\IN}}{\Bc}\gamma_j^4\\
                                &=\sum_{j\in\IN}\dup[\tilde\Y^*]{\eta_j^*}{\zeta\Bb}\dup[\tilde\Y^*]{\eta_j^*}{\zeta\Bc}
  \end{align*}
  where we used the isometry property of $\zeta$ to write inner products in $\tilde\Y^*$ as inner products in $\ell^2_\Brho$.
  
  Since this holds for all $\Bb, \Bc \in \ell^2_{\Brho}$, the operator $\mathcal C  $ acting on any $v \in \tilde\Y^*$ is given by
  \begin{equation*}
    \mathcal C   v = \sum_{j\in\IN} \dup[\tilde\Y^*]{\eta_j^*}{v}\eta_j^*
    = \sum_{j\in\IN} \gamma_j^2\dupb[\tilde\Y^*]{\eta_j^*\gamma_j^{-1}}{v}\eta_j^*\gamma_j^{-1}.
  \end{equation*}
  By \eqref{eq:normejst}, this is the spectral decomposition of $\mathcal C  $, and the eigenvalues $(\gamma_j^2)_{j\in\IN}$ are summable.
  
  Denote $e_j^*\coloneqq \eta_j^*/\gamma_j$, which gives a set of orthonormal vectors in $\tilde\Y^*$. The Cameron--Martin (CM) space $H_E$ is the image of $\mathcal C  ^{1/2}:\tilde\Y^*\to\tilde\Y^*$, which is given by
  \begin{equation*}
      \mathcal C  ^{1/2} v = \sum_{j\in\IN} \gamma_j \dup[\tilde\Y^*]{e_j^*}{v} e_j^* = \sum_{j\in\IN} \dup[\tilde\Y^*]{e_j^*}{v} \eta_j^*,
    \end{equation*}
    equipped with the inner product $\dup[H_E]{\cdot}{\cdot}\coloneqq \dup[\tilde\Y^*]{\mathcal C  ^{-1/2}\cdot}{\mathcal C  ^{-1/2}\cdot}$. Thus $\mathcal C  ^{1/2}e_j^*=\eta_j^*$, $j\in\IN$, forms an orthonormal basis of $H_E$.

    Finally, consider again the continuous dense embedding $d:=\iota^*\circ R_\Y : \Y \to \tilde\Y^*$. Since it maps the orthonormal basis $(\eta_j)_{j\in\IN}$ of $\Y$ to the orthonormal basis $(\eta_j^*)_{j\in\IN}$ of $H_E$, it defines a bijective isometry from $\Y$ onto $H_E$.

    \subsection{Posterior Existence}

    \subsubsection{Proof of Lemma \ref{lemma:GXZ}}\label{sec:lemma:GXZ}

      Throughout this proof we will use that $T$ in \eqref{eq:T} is a measurable map as a composition
  of measurable maps, due to Assumptions \ref{ass:data} and \ref{ass:meas} and the fact that all spaces are separable
  Hilbert spaces equipped with the Borel $\sigma$-algebras, so that additions are continuous.
  
  {\bf \hypertarget{hyp:step1lemjoint}{Step 1}.} We show \ref{item:ZGX}. Let $\Phi$ and $\pPhi$
  be as in Assumption \ref{ass:generalassump} \ref{item:general1} and \eqref{eq:pPhi},
  and define
  \begin{equation}\label{eq:Ddef}
    D:\B(\pZ)\times (\calG\times\pX)\to\IR: (B_{\pZ},(g,\px))\mapsto\int_{B_{\pZ}} \exp(-\pPhi(g,\px,\pz))\dd\IP_{\pEE}(\pz).
  \end{equation}
  To show that $D=\IP_{\pZZ|(G,\pXX)}$ as desired, we verify that $D$ satisfies
  Definition \ref{def:regcond}
  \ref{item:conditional1}-\ref{item:conditional3}:

  \begin{enumerate}[label=(\roman*)]
  \item Due to Assumption \ref{ass:generalassump} \ref{item:general2},
    $\exp(-\pPhi):\calG\times\pX\times\pZ\to\IR$ is jointly measurable. %
    Tonelli's theorem then implies that
    $D(B_{\pZ},\cdot):(\calG\times\pX)\to\IR$ in
    \eqref{eq:Ddef} is measurable for all $B_{\pZ}\in\B(\pZ)$.
\item By Assumption \ref{ass:generalassump} \ref{item:general1}, there exists a
  $(\IP_G\otimes \IP_{X})$-null-set $N_{\calG,\X}\subseteq \calG\times \X$ such that
  \begin{equation*}
     \frac{\dd (T_{d\circ g(x)})_\sharp\IP_E}{\dd \IP_E}=\exp(-\Phi(g,x,\cdot))
  \end{equation*}
  for all $(g,x)\notin N_{\calG,\X}$. Using independence of the $E_i$ and Tonelli's theorem we easily get $(T_{d\circ g(\px)})_\sharp\IP_{\pEE} = \otimes_{i=1}^n (T_{d\circ g(x_i)})_\sharp \IP_E\ll \IP_{\pEE}$
  and
  \begin{equation}\label{eq:exppPhi}
    \frac{\dd (T_{d\circ g(x)})_\sharp\IP_{\pEE}}{\dd \IP_{\pEE}}(\pz)=
    \prod_{i=1}^n\exp(-\Phi(g,x_i,z_i))
    =\exp(-\pPhi(g,\px,\pz))
  \end{equation}
for each $(g,\px)\notin N_{\calG,\pX}\coloneqq\bigcup_{i=1}^n\pi_i^{-1}(N_{\calG,\X})$, where
  $\pi_i:\calG\times\pX\to\calG\times\X$, $(g,\px)\mapsto (g,x_i)$.
  Equation \eqref{eq:exppPhi} implies that $D(\cdot,(g,\px))$ in \eqref{eq:Ddef} is a probability
  measure for $(g,\px)\notin N_{\calG,\pX}$. Since each $\pi_i$ is measurable with
  $(\pi_i)_\sharp\IP_{(G,\pXX)}=\IP_G\otimes\IPX$, the set $N_{\calG,\pX}$ is a measurable nullset.  Redefining $D(\cdot,(g,\px))\coloneqq\IP_{\pEE}$ for $(g,\px)\in N_{\calG,\pX}$, which affects
  neither the measurability shown in (i) nor the identity in (iii) below,
  shows that $D(\cdot,(g,\px))$ is a probability measure for \emph{every} $(g,\px)\in\calG\times\pX$. This gives \ref{item:conditional2}.
\item
  Fix $B_{\pZ}\in\B(\pZ)$, define $\IP_{(G,\pXX)}\coloneqq \IP_G\otimes\IPpX$ and let $B_{\calG,\pX}\in\B(\calG\times\pX)$. Then
  by \eqref{eq:exppPhi}
  \begin{align*}
    &\int_{B_{\calG,\pX}}\underbrace{\int_{B_{\pZ}} \exp(-\pPhi(g,\px,\pz)) \dd\IP_{\pEE}(\pz)}_{=D(B_{\pZ},(g,\px))}\dd\IP_{(G,\pXX)}(g,\px)\\
    &\qquad= \int_{B_{\calG,\pX}\backslash N_{\calG,\pX}}\int_{\pZ}\One_{B_{\pZ}}(\pz)
      \dd (T_{d\circ g(\px)})_{\sharp}\IP_{\pEE}(\pz)\dd\IP_{(G,\pXX)}(g,\px)\\
    &\qquad= \int_{B_{\calG,\pX}}
      \int_{\pZ}\One_{B_{\pZ}}(d\circ g(\px)+\pe)\dd\IP_{\pEE}(\pe)\dd\IP_{(G,\pXX)}(g,\px)\\
    &\qquad= \int_{\calG\times \pX}\One_{B_{\calG,\pX}}(g,\px)
      \int_{\pZ}\One_{B_{\pZ}}(d\circ g(\px)+\pe)\dd\IP_{\pEE}(\pe)\dd\IP_{(G,\pXX)}(g,\px)\\
    &\qquad = \int_{\calG\times\pX\times\pZ}\One_{B_{\calG,\pX}\times B_{\pZ}}(T(g,\px,\pe)) \dd\IP_{(G,\pXX,\pEE)}(g,\px,\pe),
  \end{align*}
  where in the last equality we used Tonelli's theorem and the fact
  that $T^{-1}(B_{\calG,\pX}\times B_{\pZ})$ is a measurable set.
  By \eqref{eq:GXZ}, the integral thus evaluates to
  \begin{align*}
    \IP_{(G,\pXX,\pEE)}\set{(g,\px,\pe)}{T(g,\px,\pe)\in B_{\calG,\pX}\times B_{\pZ}}&= T_\sharp\IP_{(G,\pXX,\pEE)}(B_{\calG,\pX}\times B_{\pZ})\\
    &= \IP_{(G,\pXX,\pZZ)}(B_{\calG,\pX}\times B_{\pZ}).
    \end{align*}
  \end{enumerate}
  This concludes the proof of \ref{item:ZGX}.

  {\bf \hypertarget{hyp:step2lemjoint}{Step 2}.} For \ref{item:GXZ}, we use that $\Phi$ satisfies both
  items of Assumption \ref{ass:generalassump}, so that in particular $\pPhi$
  is jointly measurable. By \hyperlink{hyp:step1lemjoint}{Step 1} we have for any $B_{\calG,\pX}\in\B(\calG\times\pX)$,
  $B_{\pZ}\in\B(\pZ)$
  \begin{align*}
    \IP_{(G,\pXX,\pZZ)}(B_{\calG,\pXX}\times B_{\pZ}) &= \int_{B_{\calG,\pX}}\int_{B_{\pZ}} \exp(-\pPhi(g,\px,\pz)) \dd\IP_{\pEE}(\pz)\dd\IP_{(G,\pXX)}(g,\px)\\
    &=\int_{B_{\calG,\pX}\times B_{\pZ}}\exp(-\pPhi(g,\px,\pz))\dd\IP_{(G,\pXX,\pEE)}(g,\px,\pz)
  \end{align*}
  where the second equality follows again by Tonelli's theorem and
  the fact that $\exp(-\pPhi)$ is jointly measurable. This proves \ref{item:GXZ}.

\subsubsection{Proof of Theorem \ref{th:existenceposterior}}\label{app:existenceposterior}
By Assumption \ref{ass:generalassump} \ref{item:general2}, Lemma \ref{lemma:GXZ} \ref{item:GXZ}, and Tonelli's theorem, we have
\begin{align}\label{eq:Zfin}
  \int_{\pX\times\boldsymbol{\Z}}\mathfrak Z(\px,\pz)\dd(\IPpX\otimes\IPpE)(\px,\pz)&=\int_{\calG\times \pX\times\boldsymbol{\Z}}
      \exp\left(-\pPhi(g,\px,\pz)\right)\dd\IP_{(G,\pXX,\pEE)}(g,\px,\pz)\nonumber\\
  &=\IP_{(G,\pXX,\pZZ)}(\calG\times \pX\times\boldsymbol{\Z})=1.
\end{align}
Moreover, the joint measurability of $\pPhi$ implies measurability of $\mathfrak Z$ in \eqref{eq:Zxy},
so that
\begin{equation}\label{eq:NXZ}
  N_{\pX,\pZ}:=(\pX\times\pZ)\backslash \set{(\px,\pz)\in\pX\times\pZ}{0<\mathfrak Z(\px,\pz)<\infty}\subset \pX\times\pZ
\end{equation}
is a measurable set. Due to \eqref{eq:Zfin} and the assumption that $\mathfrak Z(\px,\pz)>0$
$\IP_{(\pXX,\pEE)}$-a.s., $N_{\pX,\pZ}$ is a $\IP_{(\pXX,\pEE)}$-null set.

Define for $B_{\calG}\in \B(\calG)$ and $(\px,\pz)\in\pX\times\pZ$
\begin{subequations}\label{eq:Ddefconditional}
\begin{equation}\label{eq:Ddefint}
  I(g,(\px,\pz)):=
    \begin{cases}
      \frac{1}{\mathfrak Z(\px,\pz)}\exp\left(-\pPhi(g,\px,\pz)\right)&\text{if }%
                                                              (\px,\pz)\notin N_{\pX,\pZ}
      \\
      1&\text{if }%
                                                              (\px,\pz)\in N_{\pX,\pZ}
  \end{cases}
\end{equation}
and
\begin{equation}\label{eq:Ddefcond}
  D(B_{\calG},(\px,\pz)):=\int_{\calG} \One_{B_{\calG}}(g) I(g,(\px,\pz)) \dd\IP_G(g).
\end{equation}
\end{subequations}
To show that $D=\IP_{G|(\pXX,\pZZ)}$ as desired, we verify Definition \ref{def:regcond} \ref{item:conditional1}--\ref{item:conditional3}:

\begin{enumerate}[label=(\roman*)]
\item Let $B_{\calG}\in\B(\calG)$. By assumption $\exp(-\pPhi)$ is
  jointly measurable and non-negative. Thus (as mentioned) $\mathfrak Z$ in
  \eqref{eq:Zxy} is jointly measurable by Tonelli's
  theorem, and consequently $N_{\pX,\pZ}$ in \eqref{eq:NXZ} is a
  measurable set. %
  Since $B_{\calG}\in\B(\calG)$, the integrand $\One_{B_{\calG}}(g)I(g,\px,\pz)$ in
  \eqref{eq:Ddefcond} is non-negative and jointly measurable. Another
  application of Tonelli's theorem implies measurability of
  $(\px,\pz)\mapsto D(B_{\calG},(\px,\pz))$.
\item Since $I:\calG\times \pX\times\pZ\to\IR$ in \eqref{eq:Ddefint} is jointly
  measurable, for every fixed
  $(\px,\pz)\in\pX\times\pZ$, it is measurable as a function of $g\in\calG$. Hence, for
  each $(\px,\pz)$, $D(\cdot,(\px,\pz))$ in \eqref{eq:Ddefconditional}
  defines a probability measure on $(\calG,\B(\calG))$.
\item 
  Lemma \ref{lemma:GXZ} \ref{item:GXZ} implies that $\IP_{(\pXX,\pEE)}$-a.e.
  \begin{align}\label{eq:XZXEequivalence}
    \frac{\dd\IP_{(\pXX,\pZZ)}}{\dd\IP_{(\pXX,\pEE)}}(\px,\pz)=\int_{\cal{G}}\exp\left(-\pPhi(g,\px,\pz)\right)\dd\IP_G(g)=\mathfrak Z(\px,\pz).
  \end{align}  
  Since $\IP_{(\pXX,\pZZ)}\ll \IP_{(\pXX,\pEE)}=(\IPpX\otimes\IPpE)$ by Lemma \ref{lemma:GXZ} \ref{item:GXZ},
$N_{\pX,\pZ}$ in \eqref{eq:NXZ} is a $\IP_{(\pXX,\pZZ)}$-null set.
  
  Hence, for $B_{\calG}\in\B(\calG)$, $B_{\pX,\pZ}\in\B(\pX\times\pZ)$
\begin{align*}
    &\int_{B_{\pX,\pZ}}D(B_{\calG},(\px,\pz))\dd \IP_{(\pXX,\pZZ)}(\px,\pz)\\
    =~&\int_{B_{\pX,\pZ}\backslash N_{\pX,\pZ}}\mathfrak Z(\px,\pz)^{-1}\int_{B_\calG}\exp\left(-\pPhi(g,\px,\pz)\right)\dd\IP_G(g)\underbrace{\dd \IP_{(\pXX,\pZZ)}(\px,\pz)}_{=\mathfrak Z(\px,\pz)\dd\IP_{(\pXX,\pEE)}(\px,\pz)}\\
    =~&\int_{B_{\calG}\times B_{\pX,\pZ}}\exp\left(-\pPhi(g,\px,\pz)\right)\dd\IP_{(G,\pXX,\pEE)}(g,\px,\pz)\\
    =~&\IP_{(G,\pXX,\pZZ)}(B_{\calG}\times B_{\pX,\pZ}),
\end{align*}
where in the last equality we used the characterization of the joint distribution $\IP_{(G,\pXX,\pZZ)}$ from Lemma \ref{lemma:GXZ} \ref{item:GXZ}. 
\end{enumerate}

In all this shows $D=\IP_{G|(\pXX,\pZZ)}$. The formula for the Radon--Nikodym derivative
\eqref{eq:generalposteriorform} follows directly by definition \eqref{eq:Ddefconditional}
and the fact that $N_{\pX,\pZ}$ is a $\IP_{(\pXX,\pZZ)}$-null set.

\subsubsection{Proof of Theorem \ref{th:PosteriorGaussianNoise}}\label{app:proofposteriorgaussian}
We aim to apply Theorem \ref{th:existenceposterior} and therefore need to verify Assumptions \ref{ass:data}--\ref{ass:generalassump} as well as \eqref{eq:Zxy}.

  {\bf \hypertarget{hyp:step1postexistence}{Step 1} (Assumptions \ref{ass:data} and \ref{ass:meas}).}
  Assumption \ref{ass:data} (on the data generation) is implied by the more specified Setting \ref{set:gaussiansetup}  and the fact that $d:\Y\to\tilde\Y^*$ is bounded linear (hence continuous), and thus Borel-measurable.

  Assumption \ref{ass:meas} requires that $\calG\times\X\to\Y:(g,x)\mapsto g(x)$ is measurable. Since $\X$, $\Y$ are metric spaces, $\calG=C^0(K)$ and $K\subset\X$ is compact (by Setting \ref{set:gaussiansetup}), the evaluation map $\calG\times\X\to\Y:(g,x)\mapsto g(x)$ is
  continuous and thus
  Borel-measurable w.r.t.\ the product
  topology on $\calG\times\X$.

{\bf \hypertarget{hyp:step2postexistence}{Step 2} (Assumption \ref{ass:generalassump} \ref{item:general1}).} 
Since $\IPpE$ is a Gaussian measure on the separable Hilbert space $\bs{\tilde\Y^*}$, the Cameron--Martin theorem \cite[Corollary 2.4.3]{bogachev1998gaussian} shows
$ (T_{d\circ g(x)})_\sharp\IP_E\ll \IPpE$ iff $d\circ g(\px)\in\ H_{\boldsymbol E}$ which is the case for all $g\in\calG=C^0(K)$ and $\px\in \pK\coloneqq\supp(\IPpX)$ 
since $d:\Y\to H_E$ by Lemma \ref{lem:white_noise_construction}.
This shows Assumption \ref{ass:generalassump} \ref{item:general1}.

{\bf \hypertarget{hyp:step3postexistence}{Step 3} (expression for $\pPhi$).} 
The Cameron--Martin Theorem \cite[Corollary 2.4.3]{bogachev1998gaussian} shows that for $(\IP_G\otimes\IPpX)$-a.e.\ $(g,\px)$
\begin{align*}
    \frac{\dd\IP_{\pZZ|(G,\pXX)}(\cdot,g,\px)}{\dd\IPpE}(\pz)=\exp\left(- \frac{1}{2}\biggl(\sumin\lVert g(x_i)\rVert_{\Y}^2-2R_{\IPE}^{-1}(d\circ g(x_i))(z_i)\biggr)\right)
\end{align*}
with $d=\iota^*\circ R_{\Y}:\Y\to\tilde\Y^*$ from \eqref{eq:disomorphism},
\begin{equation*}
R_{\IPE}:\tilde\Y_{\IPE}^{**}\to\tilde\Y^*, ~~R_{\IPE}f=\int_{\tilde\Y^*}zf(z)\dd\IPE(z),
\end{equation*}
and where $\tilde\Y_{\IPE}^{**}$ is the $L^2(\IPE)$-closure of $\tilde\Y^{**}$, see \cite[p.~44]{bogachev1998gaussian}. 
To conclude the claim we show $R_{\IPE}\dup[\sim]{\cdot}{y}= d(y)\in\tilde\Y^*$ for all $y\in\Y$ which is equivalent to  
$${}_{\tilde\Y^*}\dup[\tilde\Y]{R_{\IPE}\dup[\sim]{\cdot}{y}}{\tilde y}={}_{\tilde\Y^*}\dup[\tilde\Y]{d(y)}{\tilde y}=\dup[\Y]{y}{\tilde y}\qquad\text{for all }y\in\Y,~\tilde y\in\tilde\Y,$$
where $\dup[\Y]{\cdot}{\cdot}$ denotes the scalar product in $\Y$.
Using the definition of $R_{\IPE}$, we have for $y\in\Y$ and $\tilde y\in\tilde\Y$,
\begin{align}\label{eq:prooftildeloglikelihood}
    {}_{\tilde\Y^*}\dup[\tilde\Y]{R_{\IPE}\dup[\sim]{\cdot}{y}}{\tilde y}&={}_{\tilde\Y^*}\left\langle  \int_{\tilde\Y^*}z \dup[\sim]{z }{y}\dd\IPE(z )\,,\,\tilde y\right\rangle{}_{\tilde\Y}\notag\\
&=  \int_{\tilde\Y^*}{}_{\tilde\Y^*}\dup[\tilde\Y]{z }{\tilde y}\dup[\sim]{z }{y}\dd\IPE(z )\notag\\
&=\lim_{N\to\infty}\int_{\tilde\Y^*}{}_{\tilde\Y^*}\dup[\tilde\Y]{z }{\tilde y}\,{}_{\tilde\Y^*}\dup[\tilde\Y]{z }{P_N y}\dd\IPE(z ).
\end{align}
In \eqref{eq:prooftildeloglikelihood} we used \eqref{eq:IWchar2} and that for fixed $y\in\Y$
it holds that for $\IP_E$-a.e.~$z\in\tilde\Y^*$ we have
$\langle z ,y \rangle_\sim=\lim_{N\to\infty} {}_{\tilde\Y^*}\langle z ,P_Ny\rangle_{\tilde\Y}$ 
by Corollary \ref{cor:existencetildemapping}. 
Using \eqref{eq:WN} and Corollary \ref{cor:existencetildemapping}, we get for $y\in\Y$ and $\tilde y\in\tilde\Y$,

\begin{align*}
\lim_{N\to\infty} \int_{\tilde\Y^*}{}_{\tilde\Y^*}\dup[\tilde\Y]{z }{\tilde y}\,{}_{\tilde\Y^*}\dup[\tilde\Y]{z }{P_N y}\dd\IPE(z )&=\lim_{N\to\infty} \IE[\IW_\Y(\tilde y) \IW_\Y(P_N y)]\\
&=\lim_{N\to\infty}\dup[\Y]{\tilde y}{P_Ny}=\dup[\Y]{\tilde y}{y}.
\end{align*}

{\bf \hypertarget{hyp:step4postexistence}{Step 4} (Assumption \ref{ass:generalassump} \ref{item:general2}).} 
It suffices to show Borel measurability of the map
$$\phi:\calG\times\X\times\tilde\Y^*\to\IR,~~\phi(g,x,z )=\frac 12\lVert g(x)\rVert_{\Y}^2-\langle z ,g(x)\rangle_{\sim}$$ as summation and exponentiation are Borel measurable in $\IR$. The first summand is measurable as the norm $\lVert\cdot\rVert_\Y$ and the map $\calG\times\X\to\IR$, $(g,x)\mapsto g(x)$ are continuous and thus Borel measurable. The second summand is measurable, as measurability of the map $\langle\cdot,\cdot\rangle_\sim$ is guaranteed by Corollary \ref{cor:existencetildemapping}.

{\bf \hypertarget{hyp:step5postexistence}{Step 5} ($\mathfrak Z(\px,\pz)>0$).}
We have $\pPhi(g,\px,\pz)<\infty$ on $\calG\times\pK\times\boldsymbol{\tilde\Y^*}$ by the definition of $\langle\cdot,\cdot \rangle_\sim$ in Corollary \ref{cor:existencetildemapping}. Thus, for all $(\px,\pz)\in\pK\times\boldsymbol{\tilde\Y^*}$, we have
\begin{align*}
    \mathfrak Z(\px,\pz)=\int_{\calG}\exp(-\pPhi(g,\px,\pz))\dd\IP_G(g)>0.
\end{align*}

\subsection{Prior Measure}\label{app:priormeasure}
\subsubsection{Proof of Lemma \ref{lem:decay}}\label{app:decayholo}
We divide the proof into two steps.

{\bf Step 1 ($\boldsymbol{G\in C^0(K)}$ a.s.).}
Let $G\sim \IP_G$ and denote the (random) basis coefficients of $G$ w.r.t.~the ONB $(\psij)_{j\in\IN}$ by $G_j$, i.e.~$G_j=\langle G,\psij\rangle_{L^2(K)}$
and let $a\coloneqq \sum_{j\in\IN}|G_j|\omegaj$. 
In the Gaussian case from Example \ref{ex:GaussianPrior}, the RVs $a_j\coloneqq \omega_j| G_j|$ are half-normal with $\mu_j\coloneqq\IE[a_j]=\omega_j\kappaj\sqrt{2/\pi}$. 
Then $\sum_{j\in\IN}\kappaj\omegaj<\infty$ implies $\sumj\mu_j<\infty$. In the uniform case from Example \ref{ex:UniformPrior} we have $|G_j|\sim\text{Unif}([0,\kappaj])$, i.e.~$\mu_j=\frac{\kappaj\omegaj}{2}$ and thus also $\sumj\mu_j<\infty$ in this case. The monotone convergence theorem then shows $a<\infty$ a.s.~in both cases.
Now we define the operator sequence
    \begin{align*}
        \tilde{Q}_N: K\to\Y,\quad \tilde{Q}_N(x)=\sumjN  G_j \psij(x),~~N\in\IN.
    \end{align*}
    Each $\tilde{Q}_N$ is continuous and bounded, since the basis elements $\psij$ are continuous by assumption (and thus bounded on every compact set).
    It holds for $N,M\in\IN$, $M<N$,
    \begin{align}
        \sup_{x\in K}\lV \tilde{Q}_N(x)-\tilde{Q}_M(x)\rV_{\Y}&\leq \sup_{x\in K} \left\lV \sum_{M<j\leq N} \psij G_j\right\rV_{\Y}\leq \sum_{M<j\leq N} \omegaj | G_j| \overset{M,N\to\infty}{\xrightarrow{\hspace{1cm}}} 0\label{eq:C0convergencePn}
    \end{align}
    where we used that a.s.~$a<\infty$. Since $(C^0(K),\lVert\cdot\rVert_{\infty})$ is a Banach space, we get $G\in C^0(K)$ a.s.

{\bf Step 2 ($\IP_G$ induces a measure on $\boldsymbol{C^0(K)}$).}
For $N\in\IN$, let $g_j\coloneqq \langle g,\psi_j\rangle_{L^2(K)}$ and 
\begin{align}\label{eq:QN}
Q_N:L^2(K)\to C^0(K),\qquad Q_Ng\coloneqq \sum_{j\le N}g_j\psi_j 
\end{align}
which is bounded linear with
$\lVert Q_Ng\rVert_\infty\le\big(\sum_{j\le N}\omega_j^2\big)^{1/2}
\lVert g\rVert_{L^2(K)}$. Now define
\begin{align*}
\Lambda\coloneqq\biggl\{g\in L^2(K):\sum_{j\in\IN}\omega_j\lvert g_j\rvert<\infty\biggr\}.
\end{align*}
Since $g\mapsto\sum_{j\in\IN}\omega_j\lvert g_j\rvert$ is
a pointwise limit of continuous functions and hence Borel measurable we have
$\Lambda\in\B(L^2(K))$. Also $\IP_G(\Lambda)=1$ following Step 1. 
For $g\in\Lambda$ and $M,N\in\IN$, $M<N$, we
have 
$$\lVert Q_Ng-Q_Mg\rVert_\infty\le\sum_{M<j\le N}\omega_j\lvert g_j\rvert\to0,\qquad M,N\to\infty,$$
 see \eqref{eq:C0convergencePn}, so $(Q_Ng)_{N\in\IN}$ converges in the Banach space
$(C^0(K),\lVert\cdot\rVert_\infty)$. Thus
\begin{align*}
    Q:L^2(K)\to C^0(K),\qquad
    Qg\coloneqq \lim_{N\to\infty}Q_Ng\ \text{ for }g\in\Lambda,\qquad Qg\coloneqq 0\ \text{ else},
\end{align*}
is Borel measurable and
$\IP_G^{C^0}\coloneqq Q_\#\IP_G$ is a Borel probability measure on $(C^0(K),\B(C^0(K)))$.
Finally, $\lVert\cdot\rVert_{L^2(K)}\le\lVert\cdot\rVert_\infty$ and
$Q_Ng\to g$ in $L^2(K)$ give $Qg=g$ in $L^2(K)$ for every $g\in\Lambda$, so
$\IP_G^{C^0}$ is a version of $\IP_G$ supported on $C^0(K)$. 
This shows the claim.
  
\section{Proofs of Section \ref{sec:contraction}}\label{app:proofscontraction}
\subsection{Proof of Lemma \ref{lem:posteriorexpectationexistence}}\label{app:proofexistenceGnhat}
    Recall $\mathcal G=C^0(K)$ and abbreviate $\lVert\cdot\rVert_{\infty}\coloneqq \lVert\cdot\rVert_{C^0(K)}$. We show that the operator
    \begin{align*}
    F:\calG\to\calG,~~F(g)=g\exp\left(-\pPhi(g,\px,\pz)\right)
    \end{align*}
    belongs $(\px,\pz)$ $(\IPpX\otimes\IPpE)$-a.s.~to the Bochner space $L^1(\calG,\IP_G,\calG)$. 
    
    Set $\tomega_j\coloneqq\omega_j\max\{1,\log(j)^{1/2}\}$, so \eqref{eq:strongerdecayproperty} reads $\sum_{j\in\IN}\kappa_j\tomega_j<\infty$. Applying Lemma \ref{lem:envelopelemma} we have for $x_i\in K$, $z_i\in\tilde\Y^*$ and $i=1,\dots,n$,
        \begin{align}\label{eq:envelopebound}
        |\langle z_i,\psi_j(x_i)\rangle_\sim|\le \tfrac1n C(\px,\pz)\,\tomega_j,\qquad j\in\IN
    \end{align}
    with the  $(\IPpX\otimes\IPpE)$-a.s.~finite constant 
    $$C(\px,\pz)\coloneqq n\max_{1\le i\le n}M(x_i,z_i)$$
    and $M$ from Lemma \ref{lem:envelopelemma}. 

    Fix $g\in\calG_0\coloneqq\{g\in C^0(K):\sumj\tomega_j|\gj|<\infty\}$, which is a $\IP_G$-full set by \eqref{eq:strongerdecayproperty}, so that $g=\sum_{j\in\IN}g_j\psi_j$ with $g_j=\langle g,\psi_j\rangle_{L^2(K)}$ converges in $\lVert\cdot\rVert_\infty$, see \eqref{eq:C0convergencePn}. The log-likelihood from \eqref{eq:phigeneral} can be bounded by
    \begin{align}\label{eq:Phicrossbound}
        -\pPhi(g,\px,\pz)=\sum_{i=1}^n\left(-\tfrac12\lVert g(x_i)\rVert_\Y^2+\langle z_i,g(x_i)\rangle_\sim\right)\le\sum_{i=1}^n\langle z_i,g(x_i)\rangle_\sim.
    \end{align}
   Fix $i\in\{1,\dots,n\}$ and consider the partial sums
\begin{align*}
    Q_N^i\coloneqq\sumjN\gj\dup[\sim]{z_i}{\psij(x_i)},\quad N\in\IN.
\end{align*}
We first show that $Q_N^i$ converges absolutely $g$ $\IP_G$-a.s. Using \eqref{eq:envelopebound} we have
\begin{align}\label{eq:SNiestimate}
    |Q_N^i|\leq\sumj|\gj|\left\lvert\dup[\sim]{z_i}{\psij(x_i)}\right\rvert\leq\tfrac1n C(\px,\pz)\sumj\tomega_j|\gj|.
\end{align}
For the uniform prior from Example \ref{ex:UniformPrior} we have $|\gj|\leq\kappaj$, and for the Gaussian prior from Example \ref{ex:GaussianPrior} we have $\IE\sumj\tomega_j|\gj|=\sqrt{2/\pi}\sumj\kappaj\tomega_j<\infty$, both by \eqref{eq:strongerdecayproperty}. Thus the r.h.s.~of \eqref{eq:SNiestimate} is finite $(g,\px,\pz)$ $(\IP_G\otimes\IPpX\otimes\IPpE)$-a.s.~for both priors. 

On $\Omega\setminus\Omega_i$, $\IP(\Omega_i)=0$, we have $\dup[\sim]{E_i}{\psij(x_i)}=\WY(\psij(x_i))$ for all $j\in\IN$. Since $\WY:\Y\to L^2(\Omega,\IP)$ is linear, see \eqref{eq:WN}, we have
\begin{align}
    Q_N^i=\sumjN\gj\WY(\psij(x_i))=\WY(Q_Ng(x_i))\label{eq:linearitywhitenoise}
\end{align}
with equality in $L^2(\Omega,\IP)$.
Since $Q_Ng(x_i)\to g(x_i)$ in $\Y$ and $\WY$ is an $L^2(\Omega,\IP)$-isometry, see \eqref{eq:WN}, we obtain
\begin{align*}
    Q_N^i\to\WY(g(x_i))=\dup[\sim]{E_i}{g(x_i)}\quad\text{in }L^2(\Omega,\IP).
\end{align*}
$L^2$-convergence yields a subsequence $(N_k)_{k\in\IN}$ with $Q_{N_k}^i\to\WY(g(x_i))$ $\IP$-a.s.~for $k\to\infty$. 
For fixed $g$ in the $\IP_G$-full set $\calG_0\coloneqq\{g\in\calG:\sumj\tomega_j|\gj|<\infty\}$ this identifies, $(\IPpX\otimes\IPpE)$-a.s.,
\begin{align}\label{eq:pointwiseexpansion}
    \dup[\sim]{z_i}{g(x_i)}=\sumj\gj\dup[\sim]{z_i}{\psij(x_i)},\quad i=1,\dots,n,
\end{align}
the exceptional set depending on $g$. Set
\begin{align*}
    B\coloneqq\left\{(\px,\pz,g):\ \eqref{eq:pointwiseexpansion}\text{ fails for some }i\right\}\subseteq\pX\times\bs{\tilde\Y^*}\times\calG
\end{align*}
and denote the $g$-sections and $(\px,\pz)$-sections by $B_g$ and $B_{(\px,\pz)}$ respectively.
Since $(g,\px)\mapsto g(x_i)$ is continuous and $\dup[\sim]{\cdot}{\cdot}$ is Borel measurable (Corollary \ref{cor:existencetildemapping}), the partial sums $\sumjN\gj\dup[\sim]{z_i}{\psij(x_i)}$ and hence their limit are jointly Borel measurable. As $\calG$, $\X$ and $\tilde\Y^*$ are separable metric spaces, we have $B\in\B(\pX)\otimes\B(\bs{\tilde\Y^*})\otimes\B(\calG)$, see, e.g.,~\cite[Lemma 6.4.2]{bogachev2007measurevol2}. Following \eqref{eq:pointwiseexpansion}, the $g$-sections satisfy $(\IPpX\otimes\IPpE)(B_g)=0$ for $g\in\calG_0$. Moreover, $\IP_G(\calG\setminus\calG_0)=0$, so Tonelli gives $(\IPpX\otimes\IPpE\otimes\IP_G)(B)=0$ and thus $\IP_G(B_{(\px,\pz)})=0$ for $(\IPpX\otimes\IPpE)$-a.e.~$(\px,\pz)$. Fix such $(\px,\pz)$. Then \eqref{eq:pointwiseexpansion} holds $\IP_G$-a.s., so \eqref{eq:envelopebound} and \eqref{eq:Phicrossbound} yield the $\IP_G$-a.s.~majorant
\begin{align}
    \lVert F(g)\rVert_{\infty}\leq\left(\sumj\omegaj|\gj|\right)\exp\left(\sumin\dup[\sim]{z_i}{g(x_i)}\right)\leq\left(\sumj\omegaj|\gj|\right)\exp\left(C(\px,\pz)\sumj\tomega_j|\gj|\right).\label{eq:majorant}
\end{align}
By monotone convergence,
\begin{align*}
    \lVert F\rVert_{L^1(\calG,\IP_G,\calG)}
    &=\int_\calG\lVert g\exp(-\pPhi(g,\px,\pz))\rVert_\infty\,\dd\IP_G(g)\\
    &\leq\limNinfty\int_\calG\left(\sumjN\omegaj|\gj|\right)\exp\left(C(\px,\pz)\sumjN\tomega_j|\gj|\right)\dd\IP_G(g).
\end{align*}

For the uniform prior from Example \ref{ex:UniformPrior} we have $|g_j|\leq\kappaj$ and thus
$$\lVert F\rVert_{L^1(\calG,\IP_G,\calG)}\leq \left(\sumj\omegaj\kappaj\right)\exp\left(C(\px,\pz)\sumj\tomega_j\kappaj\right)<\infty$$
which shows the claim for this case.

For the Gaussian prior from Example \ref{ex:GaussianPrior} we have $\gj\sim\N(0,\kappaj^2)$. We let $C=C(\px,\pz)$ change throughout the proof and calculate
\begin{align}
I_N&=\int_{\calG}\left(\sumjN\omega_j|g_j|\right)\exp\left(C(\px,\pz)\sumjN\tomega_j|g_j|\right)\dd\IP_G(g)\notag\\
&=\sumjN\omega_j\biggg(\frac{2}{\sqrt{2\pi\kappaj^2}}\int_0^\infty\gj\exp\left(C\tomega_j\gj-\frac{\gj^2}{2\kappaj^2}\right)\dd\gj\notag\\
&~~\times\prod_{\substack{\ell\leq N\\\ell\neq j}}\frac{2}{\sqrt{2\pi\kappal^2}}\int_0^\infty\exp\left(C\tomega_\ell g_\ell-\frac{ g_\ell^2}{2\kappa_\ell^2}\right)\dd g_\ell \biggg)\label{eq:gaussianmajorant1}.
\end{align}
Consider the integral estimates
\begin{align}
    \int_0^\infty\exp(ax-bx^2)\dd x&=\frac{\sqrt{\pi}\left(\erf\left(\frac{a}{2\sqrt{b}}\right)+1\right)\exp\left(\frac{a^2}{4b}\right)}{2\sqrt{b}}\label{eq:gaussint1}\\
    &\leq \frac{\sqrt{\pi}}{2\sqrt{b}}\exp\left( \frac{a^2}{4b}+\frac{a}{\sqrt{\pi b}}\right),~~a,b>0,\notag\\
      \int_0^\infty x \exp(ax-bx^2)\dd x&=\frac{1}{2b}+\frac{\sqrt{\pi}a}{4b^{3/2}}\exp\left(\frac{a^2}{4b}\right)\left(1+\erf\left(\frac{a}{2\sqrt{b}}\right)\right)\label{eq:gaussint2}\\
    &\leq \frac{1}{2b}+\frac{\sqrt{\pi}a}{4b^{3/2}}\exp\left( \frac{a^2}{4b}+\frac{a}{\sqrt{\pi b}}\right),~~a,b>0,\notag
\end{align}
where we used $1+\erf(x)\leq 1+cx\leq\exp(cx)$ for $x>0$ and $c= 2/\sqrt\pi$. With \eqref{eq:gaussint1} and \eqref{eq:gaussint2} we get 
\begin{align}
    &\frac{2}{\sqrt{2\pi\kappaj^2}}\int_0^\infty\gj\exp\left(C\tomega_j\gj-\frac{\gj^2}{2\kappaj^2}\right)\dd\gj\label{eq:gaussianmajorant2}\\
    \leq~&\frac{2\kappaj}{\sqrt{2\pi}}\left(1 +\sqrt{\frac{C^2\pi\tomega_j^2\kappaj^2}{2}}\exp\left(\frac{C^2\tomega_j^2\kappaj^2}{2}+\sqrt{\frac{2C^2\tomega_j^2\kappaj^2}{\pi}}\right)\right)\leq C\kappaj,\notag\\
    &\frac{2}{\sqrt{2\pi\kappal^2}}\int_0^\infty\exp\left(C\tomega_\ell g_\ell-\frac{ g_\ell^2}{2\kappa_\ell^2}\right)\dd g_\ell 
    \leq\exp\left(\frac{C^2\tomega_\ell^2\kappa_\ell^2}{2}+\sqrt{\frac{2C^2\tomega_\ell^2\kappa_\ell^2}{\pi}}\right)\leq\exp\left(C\tomega_\ell\kappa_\ell\right).  \notag
\end{align}
In \eqref{eq:gaussianmajorant2} we additionally used that $\tomega_j\kappaj\to 0$ for $j\to\infty$ following \eqref{eq:strongerdecayproperty} and changed $C$ from line to line. Inserting \eqref{eq:gaussianmajorant2} into \eqref{eq:gaussianmajorant1} yields
\begin{align*}
    I_N&\leq C\sumjN\omegaj\kappaj\left(\prod_{\ell\leq N}\exp\left(C\tomega_\ell\kappa_\ell\right)\right)\\
    &\leq C \sumj \omegaj\kappaj\exp\left(C\sum_{\ell\in\IN}\tomega_\ell\kappa_\ell \right)<\infty
\end{align*}
which shows the claim for the Gaussian case.
As Borel measurability implies weak measurability and $\calG=C^0(K)$ is separable, the Pettis measurability theorem (e.g.~\cite[Theorem II.1.2]{DiestelUhl1977}) shows that $g\mapsto g\exp(-\pPhi(g,\px,\pz))$ is $(\px,\pz)$ $(\IPpX\otimes\IPpE)$-a.s.~strongly measurable and thus Bochner integrable. In particular $\hatgn$ is Borel measurable as a parametric Bochner integral.

\subsection{Proof of Lemma \ref{lem:calcgaussiangnhat}}\label{app:proofgnhatgaussian}
For $N\in\IN$ we consider the projection $Q_N$ from \eqref{eq:QN}.
Since $(\psij)_{j\in\IN}$ is a continuous ONB of $\LtwoK$ we have $Q_N g\to g$ for all $g\in\calG$ w.r.t.~$\lVert\cdot\rVert_{\LtwoK}$. Moreover, the decay property \eqref{eq:strongerdecayproperty} implies that this convergence also holds for $\IP_G$-a.e.~$g$ w.r.t.~$\lVert\cdot\rVert_{\infty}$, see \eqref{eq:C0convergencePn}.
 
The dominated convergence theorem (integrand convergence and majorant by \eqref{eq:pointwiseexpansion} and \eqref{eq:majorant}) shows that
$(\px,\pz)$ $(\IPpX\otimes\IPpE)$-a.s.
\begin{align}
    \hatgn(\px,\pz)&=\mathfrak Z(\px,\pz)^{-1}\int_{\calG}\lim_{N\to\infty} Q_N g\exp\left(-\pPhi(Q_Ng,\px,\pz)\right)\dd\IP_G(g)\notag\\
   &=\mathfrak Z(\px,\pz)^{-1}\lim_{N\to\infty} \int_{\calG}Q_N g\exp\left(-\pPhi(Q_Ng,\px,\pz)\right)\dd\IP_G(g)\label{eq:domconvgaussian}.
\end{align}
 
With the likelihood $\pPhi$ from \eqref{eq:phigeneral} we have 
\begin{align}
    \pPhi\left(Q_Ng,\px,\pz\right)&=\sumin\frac{1}{2}\lVert Q_Ng(x_i)\rVert_{\Y}^2-\langle z_i,Q_Ng(x_i)\rangle_{\sim}.\label{eq:PhiPNG}
\end{align}
We write out the terms in \eqref{eq:PhiPNG} as
\begin{align}
    \frac{1}{2}\lVert Q_Ng(x_i)\rVert_{\Y}^2&=\frac{1}{2}\left\lVert \sumjN \gj\psij(x_i)\right\rVert_{\Y}^2=\frac{1}{2}\sumjlN\gj\gl\langle\psij(x_i),\psi_\ell(x_i)\rangle_{\Y} \notag,\\
\langle z_i,Q_Ng(x_i)\rangle_{\sim}&=\sumjN\gj\langle z_i,\psij(x_i)\rangle_{\sim}\label{eq:PNxiYalph},
\end{align}
where the last identity holds $(g,\px,\pz)$ $(\IP_G\otimes\IPpX\otimes\IPpE)$-a.s.~by \eqref{eq:linearitywhitenoise} and Fubini.
Then the definitions
\begin{align*}
    &\alphajl:\pX\to\IR,~~\alphajl(\px)=\frac{1}{n}\sumin\langle\psij(x_i),\psi_\ell(x_i)\rangle_{\Y},~~1\leq j,\ell\leq N,\\
    &\betaj:\pX\times\bs{\tilde\Y^*}\to\IR,~~\betaj(\px,\pz)=\frac{1}{n}\sumin \dup[\sim]{z_i}{\psij(x_i)},~~1\leq j\leq N,
\end{align*}
are Borel measurable and finite $(\IPpX\otimes\IPpE)$-a.s.\ by Lemma \ref{lem:envelopelemma}.
Thus we get
\begin{align*}
    &\exp\left(-\pPhi(Q_Ng,\px,\pz)\right)\\
    =~&\exp\left(-\sumin\left( \frac{1}{2}\sumjlN\gj\gl\langle\psij(x_i),\psi_\ell(x_i)\rangle_{\Y}-\sumjN\gj\dup[\sim]{z_i}{\psij(x_i)}\right)\right)\\
    =~&\exp\left(-n\left( \frac{1}{2}\sumjlN\alphajl(\px)\gj\gl-\sumjN\betaj(\px,\pz)\gj\right)\right).
\end{align*}
Writing $\alphajl=\alphajl(\px)$ and $\betaj=\betaj(\px,\pz)$ we get $(\IPpX\otimes\IPpE)$-a.s.
\begin{align}
    \hatgn(\px,\pz)&=\mathfrak Z(\px,\pz)^{-1}\lim_{N\to\infty} \int_{\calG}Q_N g\exp\left(-\pPhi(Q_Ng,\px,\pz)\right)\dd\IP_G(g)\label{eq:gnatlim}\\
    &=\mathfrak Z(\px,\pz)^{-1}\lim_{N\to\infty}\sumjN\psij\int_{\calG}\gj\exp\left(-n\left( \frac{1}{2}\sumjlNprime \alphajlprime\gjprime\glprime-\sumjNprime\betajprime\gjprime\right)\right)\dd\IP_G(g)\notag.
\end{align}
Since $\gjprime\sim\N(0,\kappajprime^2)$, $j'\in\IN$, are independent under $\IP_G$ (Example \ref{ex:GaussianPrior}), the integral reduces to the $N$-dimensional marginal, i.e.~for $j\in\IN$,
\begin{align}
    &I_j\coloneqq\int_{\calG}\gj\exp\left(-n\left( \frac{1}{2}\sumjlNprime \alphajlprime\gjprime\glprime-\sumjNprime\betajprime\gjprime\right)\right)\dd\IP_G(g)\label{eq:gaussianEcalc1}\\
   &= \prodjNprime\frac{1}{\sqrt{2\pi\kappajprime^2}}\int_{\IR^N}\gj\exp\biggg(-\frac{1}{2}\sumjlNprime \left(n\alphajlprime+\deltajlprime\kappajprime^{-2}\right)\gjprime\glprime+\sumjNprime n\betajprime\gjprime\biggg)\dd g_{1}\dots\dd g_N\notag.
\end{align}
Let $\tilde{\bs{A}}_N:\pX\to\IR^{N\times N}$, $\tilde{\bs{b}}_N:\pX\times\bs{\tilde\Y^*}\to\IR^N$ and $\GN\in\IR^N$ with
\begin{align*}
    (\tilde{\bs{A}}_N)_{j\ell}(\px)=n\alphajl(\px) +\deltajl\kappaj^{-2},~(\tilde{\bs{b}}_N)_{j}(\px,\pz)=n\betaj(\px,\pz),~\GNj=\gj
\end{align*}
where $1\leq j,\ell\leq N$.
For $0\neq\GN\in\IR^N$ we have for $\px\in\pX$ and $\pz\in\bs{\tilde\Y^*}$,
\begin{align*}
    \GNt\tilde{\bs{A}}_N\GN&=\sumjlN \gj\gl(n\alphajl+\deltajl\kappaj^{-2})=\sumin\lVert Q_Ng(x_i)\rVert_{\Y}^2+\sumjN\gj^2\kappaj^{-2}>0,
\end{align*}
 so $\tilde{\bs{A}}_N$ is symmetric positive definite and in particular $\det\tilde{\bs{A}}_N>0$ for all $N\in\IN$. Using Leibniz rule for differentiation under the integral yields
\begin{align}
    I_j&=\prodjNprime\frac{1}{\sqrt{2\pi\kappajprime^2}}\int_{\IR^N}g_{j}\exp\left(-\frac{1}{2}\GN^T\tilde{\bs{A}}_N\GN+\tilde{\bs{b}}_N^T\GN\right)\dd \GN\notag\\
    &=\prodjNprime\frac{1}{\sqrt{2\pi\kappajprime^2}}\int_{\IR^N}\frac{\partial}{\partial \tilde{b}_{j}}\exp\left(-\frac{1}{2}\GN^T\tilde{\bs{A}}_N\GN+\tilde{\bs{b}}_N^T\GN\right)\dd \GN\notag\\
    &=\prodjNprime\frac{1}{\sqrt{2\pi\kappajprime^2}}\frac{\partial}{\partial \tilde{b}_{j}}\int_{\IR^N}\exp\left(-\frac{1}{2}\GN^T\tilde{\bs{A}}_N\GN+\tilde{\bs{b}}_N^T\GN\right)\dd \GN.\label{eq:multigaussintegral}
\end{align}
The remaining integral in \eqref{eq:multigaussintegral} evaluates to
\begin{align*}
    \int_{\IR^N}\exp\left(-\frac{1}{2}\GN^T\tilde{\bs{A}}_N\GN+\tilde{\bs{b}}_N^T\GN\right)\dd \GN=\sqrt{\frac{(2\pi)^N}{\det \tilde{\bs{A}}_N}}\exp\left(\frac{1}{2}\tilde{\bs{b}}_N^T\tilde{\bs{A}}_N^{-1}\tilde{\bs{b}}_N\right).
\end{align*}
Hence
\begin{align}
    I_j&=\prodjNprime\frac{1}{\sqrt{2\pi\kappajprime^2}}\sqrt{\frac{(2\pi)^N}{\det \tilde{\bs{A}}_N}}\frac{\partial}{\partial \tilde{b}_j}\exp\left(\frac{1}{2}\tilde{\bs{b}}_N^T\tilde{\bs{A}}_N^{-1}\tilde{\bs{b}}_N\right)\notag\\
    &=\prodjNprime\frac{1}{\sqrt{2\pi\kappajprime^2}}\sqrt{\frac{(2\pi)^N}{\det \tilde{\bs{A}}_N}}\left(\tilde{\bs{A}}_N^{-1}\tilde{\bs{b}}_N\right)_j\exp\left(\frac{1}{2}\tilde{\bs{b}}_N^T\tilde{\bs{A}}_N^{-1}\tilde{\bs{b}}_N\right)\label{eq:Ij}.
\end{align}
For the normalization factor $\mathfrak Z(\px,\pz)$ similar argumentation (same integrand without the factor $g_j$) shows $(\IPpX\otimes\IPpE)$-a.s.
\begin{align}
    \mathfrak Z(\px,\pz)=\lim_{N\to\infty}\prodjNprime\frac{1}{\sqrt{2\pi\kappajprime^2}}\sqrt{\frac{(2\pi)^N}{\det\tilde{\bs{A}}_N}}\exp\left(\frac{1}{2}\tilde{\bs{b}}_N^T\tilde{\bs{A}}_N^{-1}\tilde{\bs{b}}_N\right)\label{eq:Zxygaussian}.
\end{align}
Then $\mathfrak Z(\px,\pz)\in(0,\infty)$ for $(\px,\pz)$ $(\IPpX\otimes\IPpE)$-a.s.~following Theorem \ref{th:PosteriorGaussianNoise}.
Inserting \eqref{eq:Ij} and \eqref{eq:Zxygaussian} into \eqref{eq:gnatlim} and then \eqref{eq:domconvgaussian} shows
\begin{align*}
    \hatgn(\px,\pz)&=\lim_{N\to\infty} \sumjN \psij (\tilde{\bs{A}}_N^{-1}(\px)\tilde{\bs{b}}_N(\px,\pz))_j= \lim_{N\to\infty} \psiNt\tilde{\bs{A}}_N^{-1}(\px)\tilde{\bs{b}}_N(\px,\pz).
\end{align*}
Defining $\AN\coloneqq n^{-1}\tilde{\bs{A}}_N$ and $\bN\coloneqq n^{-1}\tilde{\bs{b}}_N$ shows the claim.

\subsection{Proof of Theorem \ref{th:convergencetheorem}} \label{app:proofconvergencetheorem}
Since the proof is rather long, we first give the general idea
and outline the main steps:

By Lemma \ref{lem:calcgaussiangnhat}, the posterior mean is the $C^0(K)$-limit of $\bs{\psi}_N^T\ANin\bN$, where $\AN=(\alphajl+\delta_{j\ell}n^{-1}\kappaj^{-2})_{j,\ell=1}^N$, $\bN=(\betaj)_{j=1}^N$ and
\begin{align*}
    \alphajl=\frac1n\sumin\langle\psij(x_i),\psi_\ell(x_i)\rangle_{\Y},\qquad
\betaj=\frac1n\sumin\dup[\sim]{z_i}{\psij(x_i)}.
\end{align*}
The idea in the following is to
(i) establish concentration bounds for $\alphajl$ and $\betaj$ for $$j,\ell\leq \lceil n^{1/(\rr-\rho-1)}\rceil \eqqcolon N_H$$ using Hoeffding's inequality and Gaussian tail bounds, (ii) expand $\ANin$ in a Neumann series around $\DNinv$ (the inverse of the diagonal of $\AN$) and bounding $\RRN=\AN-\DN$, and (iii) split all infinite sums into a head part $\{j\leq N_H\}$, controlled by the concentration bound from (i), and a tail part $\{j>N_H\}$, controlled by the decay of the prior and the ground truth from \eqref{eq:rtassump}. 

The head/tail split in (iii) is used in \cite[Theorems 3.3 and 3.9]{de2023convergence}. In contrast, (ii) has no counterpart in \cite{de2023convergence} as the matrix $\AN$ is diagonal in their linear analysis. The argument is divided into five steps:

\begin{itemize}[leftmargin=!, labelwidth=\widthof{\bfseries Step 1:}, labelindent=0pt, align=left]    
    \item[\bf \hyperlink{hyp:step1}{Step 1}] We show that Lemma \ref{lem:calcgaussiangnhat} can be applied to obtain an algebraic expression for the posterior mean $\hatgn$, see \eqref{eq:hatgngaussianprior}. We then define the central objects of the proof: the set $S^\delta$ in \eqref{eq:Sdelta} and the sequence $({d_n^\delta})_{n\in\IN}$ in \eqref{eq:dn}.
    \item[\bf \hyperlink{hyp:step2}{Step 2}] We show that there exists a constant $C>0$ such that for $\delta\gtrsim\delta_n$
    \begin{align*}
        (\IPpX\otimes\IP_{\pEE})\left(S^\delta\right)\geq 1-C\exp\left(-\frac{n{d_n^\delta}}{C}\right).
    \end{align*}
    \item[\bf \hyperlink{hyp:step3}{Step 3}] We show $\lVert\hatgn-g_0\rVert_{L^2(K)}^2\lesssim\delta^2$ on $S^\delta$ in the regime $\delta_n\lesssim\delta\leq n^{3/2}$. Writing $\AN=\DN+\RRN$ splits the error into an off-diagonal contribution \hyperlink{hyp:firstterm}{$(i)$}, treated in \hyperlink{hyp:step3a}{Step 3a} via a Neumann series argument, and a diagonal contribution \hyperlink{hyp:secondterm}{$(ii)$}, treated in \hyperlink{hyp:step3b}{Step 3b} similarly to the analysis in \cite{de2023convergence}.
    \item[\bf \hyperlink{hyp:step4}{Step 4}] We treat the remaining regime $\delta\geq n^{3/2}$. Here the crude bound $\ANin\preceq\diag(n\kappaj^{2})$ suffices, since only integrability of the tail is needed.
    \item[\bf \hyperlink{hyp:step5}{Step 5}] We finish the proof: Steps \hyperlink{hyp:step3}{3} and \hyperlink{hyp:step4}{4} show that the event $\{\lVert\hatgn-g_0\rVert_{L^2(K)}^2\geq\delta^2\}$ is contained in the complement of $S^{\overline{C}^{-1}\delta}$ for some $\overline{C}>0$, so that \hyperlink{hyp:step2}{Step 2} yields the claim.
    \end{itemize}

\paragraph*{Proof.}

{\bf \hypertarget{hyp:step1}{Step 1} (Preparation and Definitions).}
For $x\in\IR$, we denote by $\lceil x\rceil$ the smallest integer greater than or equal to $x$. Following \eqref{eq:omegaassumppolynomial}, \eqref{eq:rtassump} and since $\rr>1+2\rho\geq 1+\rho$, it holds
$$\sumj \kappaj\omegaj\max\{1,\log(j)^{1/2}\}\lesssim\sumj j^{-\rr+\rho}\max\{1,\log(j)^{1/2}\}<\infty$$ and hence the decay property \eqref{eq:strongerdecayproperty} is fulfilled. 
Applying Lemma \ref{lem:calcgaussiangnhat} yields $(\px,\pz)$ $(\IPpX\otimes\IPpE)$-a.s.
\begin{align*}
    \hatgn(\px,\pz)=\limNinfty\bs{\psi}_N^T\ANin(\px)\bN(\px,\pz)
\end{align*}
with convergence in $L^2(K)$ and (abbreviating $\AN=\AN(\px)$ and $\bN=\bN(\px,\pz)$)
\begin{align}
A_{N,j\ell}&=\alphajl+\deltajl n^{-1}\kappaj^{-2},~~
\alphajl=\frac{1}{n}\sumin\langle\psij(x_i),\psi_\ell(x_i)\rangle_{\Y},~~1\leq j,\ell\leq N,\notag\\
b_{N,j}&=\beta_j=\frac{1}{n}\sumin \dup[\sim]{z_i}{\psij(x_i)},~~1\leq j\leq N\label{eq:ANbNgaussianproof}.
\end{align}
For $\delta>0$ we define the sequence $({d_n^\delta})_{n\in\IN}$ as
\begin{align}
{d_n^\delta}=c n^{-\frac {4\rho+2} {2\rr}}\max\left\{\min\{1,n^{\frac{1}{2\rr}}\delta^2\},n^{-3}\delta^2\right\},~~n\in\IN,\label{eq:dn}
\end{align} 
where $c>0$ is a constant to be chosen later. In particular, for $0<\delta\leq n^{3/2}$, 
\begin{align}\label{eq:dnestimate}
n^{\frac{4\rho+1}{2\rr}}{d_n^\delta}\lesssim \delta^2\quad\text{and}\quad n^{\frac{4\rho+2}{2\rr}}{d_n^\delta}\leq c\lesssim 1.
\end{align}
With $\pK= K^n$ and $\omega_j\coloneqq\sup_{x\in K}\lVert\psi_j(x)\rVert_\Y$, $j\in\IN$, we define the set 
\begin{align} 
    S^\delta &\coloneqq\Sdeltal\cap\Sdeltag,\label{eq:Sdelta}\\  
    \Sdeltal&\coloneqq\biggg\{(\px,\boldsymbol{e}) \in (\pK\times\bs{\tilde\Y^*}):~\trm{for }j,\ell\in\{1,\dots,\lceil n^{1/(\rr-\rho-1)}\rceil\},j\neq \ell:\label{eq:Sdeltal}\\
    &\qquad|\alphajl(\px)|^2\leq \omega_j^2\omega_\ell^2{d_n^\delta},~|\alphajj(\px)-1|^2\leq \omega_j^4 {d_n^\delta},~\left|\betaj(\px,d\circ g_0(\px)+\boldsymbol{e})-\Gzj\right|^2\leq \omega_j^2{d_n^\delta}\biggg\},\notag\\
    \Sdeltag&\coloneqq\biggg\{(\px,\boldsymbol{e}) \in (\pK\times\bs{\tilde\Y^*}):~\text{for }j>\lceil n^{1/(\rr-\rho-1)}\rceil:\label{eq:Sdeltag}\\
    &\qquad\left|\betaj(\px,d\circ g_0(\px)+\boldsymbol{e})-\Gzj\right|^2\leq C_\beta j\log(j)^2\omega_j^2\max\left\{1,\lVert g_0\rVert_{\infty}^2 \right\}{d_n^\delta}\biggg\},\qquad \delta >0\notag
\end{align}
with a constant $C_\beta>0$ to be chosen later.
The remaining proof is divided into several steps.
We first show that for $\delta\gtrsim\delta_n$ there exists a constant $C>0$ s.t.
\begin{align}
    (\IPpX\otimes\IP_{\pEE})\left(S^\delta\right)\geq 1-
C\exp\left(-\frac{n{d_n^\delta}}{C}\right).\label{eq:Smeasure}
\end{align}
Then we show
$\lVert \hatgn(\px,d\circ g_0(\px)+\boldsymbol{e})-g_0\rVert_{L^2(K)}^2\lesssim \delta^2$ for $\delta\gtrsim\delta_n$ and $(\px,\boldsymbol{e})\in S^\delta$.
Combining these results gives the claim.

{\bf \hypertarget{hyp:step2}{Step 2} (upper bound $(\IPpX\otimes\IP_{\pEE})\left(S^\delta\right)$).}
Hoeffding's inequality (e.g.~\cite[Theorem 2.2.6]{vershynin2018high}) yields for $\delta_1$,~$\delta_2\geq 0$ and $j, \ell\in\IN$, $j\neq \ell$,
\begin{align}
    \IPpX\left(|\alphajl(\pXX)|\geq \delta_1\right)&=\IPpX\left(\left|\frac{1}{n}\sumin\langle\psij(X_i),\psi_\ell(X_i)\rangle_{\Y}\right|\geq \delta_1\right)\leq 2\exp\left(-\frac{n\delta_1^2}{2\omegaj^2\omega_\ell^2}\right)\label{eq:alphajkproof}\\
    \IPpX\left(|\alphajj(\pXX)-1|\geq \delta_2\right)&=\IPpX\left(\left|\frac{1}{n}\sumin\lVert\psij(X_i)\rVert_{\Y}^2-1\right|\geq \delta_2\right)\leq 2\exp\left(-\frac{n\delta_2^2}{\omegaj^4}\right).\label{eq:alphajjproof}
\end{align}
For the RVs $\beta_j=\beta_j(\pXX,d\circ g_0(\pXX)+\pEE)$ we split off the noise dependence, i.e.
\begin{align}
        &\left(\IPpX\otimes\IPpE\right)\left(|\betaj(\pXX,d\circ g_0(\pXX)+\pEE)-\Gzj|\geq 2\delta_3\right)\label{eq:betasplit}\\
        \leq~&\left(\IPpX\otimes\IPpE\right)\left(|\betaj(\pXX,d\circ g_0(\pXX)+\pEE)-\IE_{\pEE}[\betaj](\pXX)|\geq \delta_3\right)\notag\\
        &~~+\IPpX\left(|\IE_{\pEE}[\betaj](\pXX)-\Gzj|\geq \delta_3\right)  ,\quad \delta_3\geq 0.\notag
\end{align}
By \eqref{eq:tildemapping1} it holds $\IP$-a.s.~$\dup[\sim]{E_i}{\psij(x_i)}=\WY(\psij(x_i))$, so that we have $\IE_{E}[\dup[\sim]{E_i}{\psij(x_i)}]=0$ for all $x_i\in K$, $i=1,\dots, n$ and $j\in\IN$. Using this together with the $\IPE$-a.s.~linearity of $\dup[\sim]{\cdot}{y}$ (Corollary \ref{cor:existencetildemapping}), we get for $\px\in\pK$
\begin{align}
    \IE_{\pEE}[\betaj](\px)&=\frac{1}{n}\sumin \IE_{\pEE}\left[\dup[\sim]{d\circ g_0(x_i)+E_i}{\psij(x_i)}\right]\notag\\
    &=\frac{1}{n}\sumin \IE_{\pEE}\left[\dup[\sim]{d\circ g_0(x_i)}{\psij(x_i)}\right]=\frac{1}{n}\sumin\langle g_0(x_i),\psij(x_i)\rangle_{\Y}.\label{eq:IEbetaj}
\end{align}
The last equality uses $\dup[\sim]{d(y_0)}{y}=\langle y_0,y\rangle_{\Y}$ for $y_0,y\in\Y$, which follows from the definition of $d=\iota^*\circ R_\Y$ and $\dup[\sim]{\cdot}{\cdot}$, cf.~\eqref{eq:disomorphism} and Corollary \ref{cor:existencetildemapping}.
Using this and again the $\IPE$-a.s.~linearity of $\dup[\sim]{\cdot}{y}$ we can bound the first summand in \eqref{eq:betasplit} using Gaussian tails as
\begin{align}
&\left(\IPpX\otimes\IPpE\right)\left(|\betaj(\pXX,d\circ g_0(\pXX)+\pEE)-\IE_{\pEE}[\betaj](\pXX)|\geq \delta_3\right)\notag\\
\leq~& \sup_{\px\in\pK}\IPpE\left(|\betaj(\px,d\circ g_0(\px)+\pEE)-\IE_{\pEE}[\betaj](\px)|\geq \delta_3\right)\notag\\
=~&\sup_{\px\in\pK}\IPpE\left(\left|\frac{1}{n}\sumin\dup[\sim]{E_i}{\psij(x_i)}\right|\geq \delta_3\right)\notag\\
\leq~&2 \sup_{x\in K}\exp\left(-\frac{n\delta_3^2}{2\lVert\psij(x)\rVert_\Y^2}\right)= 2\exp\left(-\frac{n\delta_3^2}{2\omegaj^2}\right),\quad \delta_3\geq 0.\label{eq:betaj1}
\end{align}
For the second summand in \eqref{eq:betasplit} we use \eqref{eq:IEbetaj} and Hoeffding's inequality to show
\begin{align}
    \IPpX\left(|\IE_{\pEE}[\betaj](\pXX)-\Gzj|\geq \delta_3\right)&=\IPpX\left(\left|\frac{1}{n}\sum_{i=1}^n\langle g_0(X_i),\psij(X_i)\rangle_{\Y}-\Gzj\right|\geq \delta_3\right)\notag\\
    &\leq 2\exp\left(-\frac{n\delta_3^2}{2\lVert g_0\rVert_{\infty}^2\omegaj^2}\right),~~\delta_3\geq 0,~~j\in\IN.\label{eq:betaj2}
\end{align}

Inserting \eqref{eq:betaj1} and \eqref{eq:betaj2} into \eqref{eq:betasplit} yields
\begin{align}\label{eq:betajfinal}
\left(\IPpX\otimes\IPpE\right)\left(|\betaj-\Gzj|\geq \delta_3\right)\leq 4\exp\left(-\frac{n\delta_3^2}{8\max\left\{1,\lVert g_0\rVert_{\infty}^2\right\}\omegaj^2}\right).
\end{align}
We now estimate the probability of the sets $\Sdeltal$ and $\Sdeltag$ separately.
Using \eqref{eq:alphajkproof} with $\delta_1^2=\omega_j^2\omega_\ell^2 {d_n^\delta}$, \eqref{eq:alphajjproof} with $\delta_2^2=\omega_j^4{d_n^\delta}$ and \eqref{eq:betajfinal} with $\delta_3^2=\omega_j^2{d_n^\delta}$ shows
\begin{align}
&(\IPpX\otimes\IP_{\pEE})\left(\Sdeltal\right)\notag\\
=~&1-(\IPpX\otimes\IP_{\pEE})\left((\pK\times\bs{\tilde\Y^*})\setminus \Sdeltal\right)\notag\\
\geq~~& 1-\sum_{\substack{j,\ell=1\\j\neq \ell}}^{\lceil n^{1/(\rr-\rho-1)}\rceil}\IPpX\left(|\alphajl|^2\geq \omega_j^2\omega_\ell^2{d_n^\delta}\right)-\sum_{j=1}^{\lceil n^{1/(\rr-\rho-1)}\rceil}\IPpX\left(|\alphajj-1|^2\geq \omega_j^4{d_n^\delta}\right)\notag\\
&~~-\sum_{j=1}^{\lceil n^{1/(\rr-\rho-1)}\rceil}\left(\IPpX\otimes\IPpE\right)\left(|\betaj-\Gzj|^2\geq \omega_j^2{d_n^\delta}\right)\notag\\
\geq~~& 1-2\sum_{\substack{j,\ell=1\\j\neq \ell}}^{\lceil n^{1/(\rr-\rho-1)}\rceil}\exp\left(-\frac{n{d_n^\delta}}{C}\right)-2\sum_{j=1}^{\lceil n^{1/(\rr-\rho-1)}\rceil}\exp\left(-\frac{n{d_n^\delta}}{C}\right)-\sum_{j=1}^{\lceil n^{1/(\rr-\rho-1)}\rceil}4\exp\left(-\frac{n{d_n^\delta}}{C}\right)\notag\\
\geq~~& 1-24n^{\frac{2}{\rr-\rho-1}}\exp\left(-\frac{n{d_n^\delta}}{C}\right)
\geq1-C\exp\left(-\frac{n{d_n^\delta}}{C}\right),\label{eq:Sequality}
\end{align}
where we used at the last inequality that for $\delta\gtrsim\delta_n$ and $\rr>1+2\rho$ we have
\begin{align}
    n{d_n^\delta}\gtrsim n n^{-\frac{4\rho+2}{2\rr}}\min\{1,n^{\frac{1}{2\rr}}\delta^2\}\gtrsim\min\left\{n^{\frac{2\rr-4\rho-2}{2\rr}},n^{\frac{2\rr-4\rho-1}{2\rr}}\delta_n^2\right\}\gtrsim\log(n)^{2},\label{eq:ndnloggrowth}
\end{align}
which implies
\begin{align*}
n^{\frac{2}{\rr-\rho-1}}\exp\left(-\frac{n{d_n^\delta}}{C}\right)&= \exp\left(-\underbrace{\frac{n{d_n^\delta}}{C}}_{\gtrsim \log(n)^2}+\frac{2}{\rr-\rho-1}\log(n)\right)\lesssim \exp\left(-\frac{n{d_n^\delta}}{C}\right).
\end{align*}
Note that the constant $C$ in \eqref{eq:Sequality} changes from line to line.
To estimate the probability of the set $\Sdeltag$, we define $g_0^\infty\coloneqq \max\{1,\lVert g_0\rVert_{\infty}^2\}$. Then, using \eqref{eq:betajfinal} again, we get  
\begin{align}
    (\IPpX\otimes\IP_{\pEE})\left(\Sdeltag\right)
&=1-(\IPpX\otimes\IP_{\pEE})\left((\pK\times\bs{\tilde\Y^*})\setminus \Sdeltag\right)\notag\\
&\geq1-\sum_{j=\lceil n^{1/(\rr-\rho-1)}\rceil}^\infty\left(\IPpX\otimes\IPpE\right)\left(|\betaj-\Gzj|^2 \geq C_\beta j\log(j)^2\omega_j^2g_0^\infty{d_n^\delta}\right) \notag\\
&\geq1-4\sum_{j=\lceil n^{1/(\rr-\rho-1)}\rceil}^\infty\exp\left(-\frac{C_\beta n{d_n^\delta}j\log(j)^2}{8}\right) \notag\\
&\geq1-4\sum_{j=2}^\infty \frac{1}{2j\log(j)^2}\exp\left(-\frac{C_\beta n{d_n^\delta}}{16}\right) \notag\\
&\geq1-C\exp\left(-\frac{n{d_n^\delta}}{C}\right)\label{eq:Sdeltagestimate}.
\end{align}
In \eqref{eq:Sdeltagestimate} we additionally used $\exp(-xy)\leq\exp(-x)/y$ which holds since $x=C_\beta n{d_n^\delta}/16\geq 1$ for $\delta\gtrsim\delta_n$ by \eqref{eq:ndnloggrowth} and choosing $C_\beta$ large enough, and $y=2j\log(j)^2\geq 1$ for $j\geq 2$. Combining \eqref{eq:Sequality} and \eqref{eq:Sdeltagestimate} shows 
\begin{align*}
    (\IPpX\otimes\IP_{\pEE})\left(S^\delta\right)&\geq 1-(\IPpX\otimes\IP_{\pEE})\left((\pK\times\bs{\tilde\Y^*})\setminus \Sdeltal\right)-(\IPpX\otimes\IP_{\pEE})\left((\pK\times\bs{\tilde\Y^*})\setminus \Sdeltag\right)\\
    &\geq 1-C\exp\left(-\frac{n{d_n^\delta}}{C}\right),
\end{align*}
which finishes \hyperlink{hyp:step2}{Step 2}.

{\bf \hypertarget{hyp:step3}{Step 3} (upper bound $\boldsymbol{\lVert\hatgn-g_0\rVert_{L^2(K)}}$ on $\boldsymbol{S^\delta}$ for $\boldsymbol{\delta_n\lesssim\delta\leq n^{3/2}}$).}
Define $\GzeroN\coloneqq(\langle g_0,\psij\rangle_{L^2(K)})_{j=1}^N$. Lemma \ref{lem:calcgaussiangnhat} and Parseval's theorem yields $(\px,\pz)$ $(\IPpX\otimes\IPpE)$-a.s.
\begin{align}
\label{eq:gnhatg0}
\lVert\gnhat(\px,\pz)-g_0\rVert_{L^2(K)}^2&=\limNinfty\lVert \bs{\psi}_N^T\ANin(\px)\bN(\px,\pz)-\bs{\psi}_N^T\GzeroN\rVert_{L^2(K)}^2\notag\\
&=\limNinfty\lVert\ANin(\px)\bN(\px,\pz)-\GzeroN\rVert_2^2.
\end{align} 
Define the matrix-valued maps $\DN:\pX\to\IRNN$ and $\RRN:\pX\to\IRNN$ with 
\begin{align}
\DN(\px)=\diag(\alphajj(\px)+n^{-1}\kappaj^{-2},j=1,\dots,N),~~\RRN(\px)=\AN(\px)-\DN(\px).\label{eq:DNRN}
\end{align}
In the following we write $\DN=\DN(\px)$ and $\RRN=\RRN(\px)$. We split \eqref{eq:gnhatg0} as
\begin{align}
  \label{eq:convonetwosplit}
  \lVert\ANin\bN-\GzeroN\rVert_2^2
  \leq~&2\underbrace{\lVert(\ANin-\DNinv)\bN\rVert_2^2 }_{\coloneqq \text{\hypertarget{hyp:firstterm}{(i)}}}
  +2\underbrace{\|\DNinv\bN-\GzeroN\|_2^2}_{\coloneqq\text{\text{\hypertarget{hyp:secondterm}{(ii)}}}} 
\end{align}
and estimate the summands \hyperlink{hyp:firstterm}{$(i)$} and \hyperlink{hyp:secondterm}{$(ii)$} in \eqref{eq:convonetwosplit} separately.

{\bf \hypertarget{hyp:step3a}{Step 3a} (estimation of {\hyperlink{hyp:firstterm}{$(i)$}}).} To estimate \hyperlink{hyp:firstterm}{$(i)$} we define $\EN\coloneqq -\DNhalfinv\RRN\DNhalfinv$. We first show $\lVert \EN\rVert_2<1/2$ for $n\geq n_0$, $n_0\in\IN$ sufficiently large, $\delta_n\lesssim\delta\leq n^{3/2}$ and $(\px,\pz)\in S^\delta$.
Since $\EN$ is symmetric, we have
 \begin{align}\label{eq:ENtwosums}
     \lVert\EN\rVert_2&\leq\supxN \sumjN\sumlNnotj \frac{n^{1/2}\kappaj |x_j||\alphajl| n^{1/2}\kappa_\ell |x_\ell|}{(1+\alphajj n\kappaj^2)^{1/2}(1+\alpha_{\ell\ell} n\kappa_\ell^2)^{1/2}}\notag\\
     &\leq\supxN \sumjnr\frac{n^{1/2}\kappaj |x_j|}{(1+\alphajj n\kappaj^2)^{1/2}}\sumlnrnotj\frac{|\alphajl| n^{1/2}\kappa_\ell |x_\ell|}{(1+\alpha_{\ell\ell} n\kappa_\ell^2)^{1/2}}\notag\\
     &~+2\supxN\sumjN\frac{n^{1/2}\kappaj |x_j|}{(1+\alphajj n\kappaj^2)^{1/2}}\sumlnrgrnotj\frac{|\alphajl| n^{1/2}\kappa_\ell |x_\ell|}{(1+\alpha_{\ell\ell} n\kappa_\ell^2)^{1/2}}.
 \end{align}
 For the first term in \eqref{eq:ENtwosums} we use that for $(\px,\pz)\in S^\delta$ we have  $\alphajj>1/2$ for $j\leq \lceil n^{1/(2\rr)}\rceil$, see \eqref{eq:alphajjinvbound} below, and $|\alphajl|^2\leq \omega_j^2\omega_\ell^2 {d_n^\delta}$ for $j,\ell\leq \lceil n^{1/(\rr-\rho-1)}\rceil $, $j\neq \ell$. Therefore we can estimate
 \begin{align}
     \label{eq:ENfirstsum}
&\supxN \sumjnr\frac{n^{1/2}\kappaj |x_j|}{(1+\alphajj n\kappaj^2)^{1/2}}\sumlnrnotj\frac{|\alphajl|n^{1/2}\kappa_\ell |x_\ell|}{(1+\alpha_{\ell\ell} n\kappa_\ell^2)^{1/2}}\notag\\
\leq~&{d_n^\delta}^{1/2}\supxN \left(\sumjnr\frac{n^{1/2}\kappaj\omega_j |x_j|}{(1+\alphajj n\kappaj^2)^{1/2}}\right)^2\notag\\
\lesssim~&{d_n^\delta}^{1/2}n\sumjnr\frac{\kappa_j^2\omega_j^2}{1+\alphajj n \kappa_j^2}\notag\\
\lesssim~&{d_n^\delta}^{1/2} n\sum_{j\leq \lceil n^{1/(2\rr)}\rceil} \frac{j^{-2\rr+2\rho}}{1+n j^{-2\rr}}
\lesssim n^{\frac{2\rho+1}{2\rr}}{d_n^\delta}^{1/2}.
 \end{align}
In \eqref{eq:ENfirstsum} we additionally used \cite[Lemma B.2]{de2023convergence} and that for $\bs{a}\in\ell^2(\IN)$ we have 
\begin{align*}
    \sup_{\bs{x}\in\ell^2,\lVert \bs{x}\rVert_{\ell^2}=1}\langle\bs{a},\bs{x}\rangle_{\ell^2}=\lVert \bs{a}\rVert_{\ell^2}
\end{align*}
which follows from Cauchy--Schwarz.
For the tail term in \eqref{eq:ENtwosums} we use $|\alphajl|\leq\omegaj\omega_\ell$ for $j,\ell\in\IN$ and \cite[Lemma B.2]{de2023convergence} to estimate
\begin{align}
    \label{eq:ENsecondsum}
&\supxN\sumjN\frac{n^{1/2}\kappaj |x_j|}{(1+\alphajj n\kappaj^2)^{1/2}}\sumlnrgrnotj\frac{|\alphajl|n^{1/2}\kappa_\ell |x_\ell|}{(1+\alpha_{\ell\ell} n\kappa_\ell^2)^{1/2}}\notag\\
\leq~&\supxN\sumjN\frac{n^{1/2}\kappaj\omega_j |x_j|}{(1+\alphajj n\kappaj^2)^{1/2}}\sup_{\lVert y_N\rVert_2=1}\sumlnrgr\frac{ n^{1/2}\kappa_\ell\omega_\ell y_\ell}{(1+\alpha_{\ell\ell} n\kappa_\ell^2)^{1/2}}\notag\\
\leq~&\left(\sumjN\frac{ n\kappa_j^2\omega_j^2}{1+\alphajj n\kappaj^2}~\sumlnrgr\frac{ n\kappa_\ell^2\omega_\ell^2}{1+\alpha_{\ell\ell} n\kappa_\ell^2}\right)^{1/2}\notag\\
\lesssim~&n\left(\sum_{j\leq \lceil n^{1/(2\rr)}\rceil}\frac{ j^{-2\rr+2\rho}}{(1+nj^{-2\rr})}~\sumlnrgr \ell^{-2\rr+2\rho}\right)^{1/2}\notag\\
\lesssim~&n n^{-\frac{2\rr-2\rho-1}{4\rr}}n^{-\frac{2\rr-2\rho-1}{2\rr-2\rho-2}}\lesssim n^{-\frac{2\rr-2\rho-1}{4\rr}}.
\end{align}
Combining \eqref{eq:ENtwosums}--\eqref{eq:ENsecondsum} and choosing $c$ in \eqref{eq:dn} sufficiently small shows that 
\begin{align}
\lVert\EN\rVert_2<1/2,\quad n\geq n_0\label{eq:ENless1}
\end{align} 
for $n_0\in\IN$ sufficiently large. Using \eqref{eq:ENless1} and the Neumann series for $\EN$ we bound \hyperlink{hyp:firstterm}{$(i)$} as
\begin{align}
    \label{eq:twofactor}
    \hyperlink{hyp:firstterm}{(i)}&=\|(\ANin-\DNinv)\bN\|_2^2 \notag\\
    &=\|\DNhalfinv\DNhalf(\ANin-\DNinv)\DNhalf\DNhalfinv\bN\|_2^2 \notag\\
    &=\|\DNhalfinv((\IIN-\EN)^{-1}-\IIN)\DNhalfinv\bN\|_2^2\notag\\
  &=\|\DNhalfinv(\IIN-\EN)^{-1}\underbrace{\EN\DNhalfinv\bN}_{\coloneqq\vN}\|_2^2\notag\\
    &\leq\lVert\DNhalfinv\rVert_2^2~\lVert (\IIN-\EN)^{-1}\rVert_2^2 ~\lVert \vN\rVert_2^2\lesssim\lVert \vN\rVert_2^2
\end{align}
where we additionally used $\lVert\DNhalfinv\rVert_2^2\lesssim 1$ as for $\delta\leq n^{3/2}$
\begin{align}\label{eq:Dninv2norm}
    \frac{n\kappa_j^2}{1+n\alpha_{jj}\kappa_j^2}\leq
    \begin{cases}
       \alpha_{jj}^{-1}<2,\quad &j<\lceil n^{1/(2\rr)}\rceil,\\
       n\kappa_j^2\lesssim 1,\quad &j\geq \lceil n^{1/(2\rr)}\rceil.
    \end{cases}
\end{align}
In \eqref{eq:Dninv2norm} we used that for $j\leq\lceil n^{1/(2\rr)}\rceil\leq\lceil n^{1/(\rr-\rho-1)}\rceil$, $\delta_n\lesssim\delta\leq n^{3/2}$ and sufficiently small $c>0$ in \eqref{eq:dn} we have 
\begin{align}\label{eq:alphajjinvbound}
|\alpha_{jj}-1|^2\leq\omega_j^4{d_n^\delta}&\leq C_{\omega}^4j^{4\rho}cn^{-\frac{4\rho+2}{2\rr}}<2^{4\rho}C_{\omega}^4cn^{-\frac{1}{\rr}}<\frac 14\implies \alpha_{jj}^{-1}<2
\end{align}
where $C_{\omega}$ denotes the constant hidden in \eqref{eq:omegaassumppolynomial}.
We proceed by estimating the remaining factor $\lVert\vN\rVert_2$ in \eqref{eq:twofactor} which is given by
\begin{align}\label{eq:vnestimate}
 \lVert\vN\rVert_2^2&=\lVert\EN\DNhalfinv\bN\rVert_2^2=\sumjN\left(\frac{n^{1/2}\kappa_j}{(1+\alphajj n\kappa_j^2)^{1/2}}\sumlNnotj\frac{\alphajl n\kappal^2\beta_\ell}{1+\alphall n \kappa_\ell^2} \right)^2.   
\end{align}
Estimating the inner sum first, we split for $j=1,\dots,N$,
\begin{align}
    \label{eq:vNtwoparts}
    \sumlNnotj\frac{|\alphajl| n\kappal^2\beta_\ell}{1+\alphall n \kappa_\ell^2}\leq\sumlNnotj\frac{|\alphajl| n\kappal^2|g_{0,\ell}|}{1+\alphall n \kappa_\ell^2}+\sumlNnotj\frac{|\alphajl| n\kappal^2|\beta_\ell-g_{0,\ell}|}{1+\alphall n \kappa_\ell^2}.
\end{align}
To estimate \eqref{eq:vNtwoparts} we define 
\begin{align}\label{eq:dnjtilde}
    \tilde{d}_{n,j}\coloneqq\begin{cases}
        {d_n^\delta},\quad & j\leq\lceil n^{1/(\rr-\rho-1)}\rceil,\\
        1,\quad & j>\lceil n^{1/(\rr-\rho-1)}\rceil.
    \end{cases}
\end{align}

Then we have $|\alphajl|\leq\tilde{d}_{n,j}^{1/2}\omega_j\omega_\ell$ for $j\leq N$, $j\neq\ell$ and $\ell\leq\lceil n^{1/(\rr-\rho-1)}\rceil$. Using this together with \cite[Lemma B.2]{de2023convergence} we bound the first summand in \eqref{eq:vNtwoparts} for $\qq\neq\rho+1$ as\footnote{In case $\qq=\rho+1$ following \eqref{eq:vnfirst}--\eqref{eq:vnfinal} shows $\lVert\vN\rVert_2^2\lesssim \max\{\delta^2,\log(n)^2 n^{-(2\qq-1)/(2\rr)}\}\lesssim \max\{\delta^2,\delta_n^2\}\lesssim\delta^2$ and thus \eqref{eq:iestimate} holds.}

\begin{align}\label{eq:vnfirst}
    &\sumlNnotj\frac{|\alphajl| n\kappal^2|g_{0,\ell}|}{1+\alphall n \kappa_\ell^2}\notag\\
    \leq~& \omega_j \tilde{d}_{n,j}^{1/2}n\sumlnr\frac{\kappal^2|g_{0,\ell}|\omega_\ell}{1+\alphall n \kappa_\ell^2}+\omegaj n \sumlnrgr\frac{\kappal^2|g_{0,\ell}|\omega_\ell}{1+\alphall n \kappa_\ell^2}\notag\\
    \lesssim~&\omega_j \tilde{d}_{n,j}^{1/2}n\sum_{\ell\leq\lceil n^{1/(2\rr)}\rceil}\frac{\ell^{-2\rr-\qq+\rho}}{1+n \ell^{-2\rr}}+\omegaj n \sumlnrgr \ell^{-2\rr-\qq+\rho}\notag\\
    \lesssim~&\omega_j \tilde{d}_{n,j}^{1/2} \max\{1,n^{-\frac{\qq-\rho-1}{2\rr}}\}+\omegaj \underbrace{n^{-\frac{\rr+\qq}{\rr-\rho-1}}}_{\leq n^{-1/2}\lesssim \tilde{d}_{n,j}^{1/2}}\notag\\
    \lesssim~&\omega_j \tilde{d}_{n,j}^{1/2} \max\{1,n^{-\frac{\qq-\rho-1}{2\rr}}\}.
\end{align}
For the second summand in \eqref{eq:vNtwoparts} we use that for $(\px,\pz)\in S^\delta$ we have
 \begin{align}\label{eq:betaj}
     |\beta_\ell-g_{0,\ell}|^2\lesssim \begin{cases}
     \omega_\ell^2 {d_n^\delta},~~&\ell\leq \lceil n^{1/(\rr-\rho-1)}\rceil,\\
     \ell\log(\ell)^2 \omega_\ell^2 {d_n^\delta}, ~~&\ell>\lceil n^{1/(\rr-\rho-1)}\rceil.
 \end{cases} \end{align}
 Using \eqref{eq:dnjtilde} and again \cite[Lemma B.2]{de2023convergence} we get for $j\neq\ell$,
 \begin{align}\label{eq:vnsecond}
    &\sumlNnotj\frac{|\alphajl|n\kappal^2|\beta_\ell-g_{0,\ell}|}{1+\alphall n \kappa_\ell^2}\notag\\
\leq~&\omega_j  {d_n^\delta}^{1/2}\tilde{d}_{n,j}^{1/2} n\sumlnr\frac{\kappal^2\omega_\ell^2}{1+\alphall n \kappa_\ell^2}+\omega_j{d_n^\delta}^{1/2}n\sumlnrgr\frac{\ell^{1/2}\log(\ell)\kappal^2\omega_\ell^2}{1+\alphall n \kappa_\ell^2}\notag\\
\lesssim~&\omega_j  {d_n^\delta}^{1/2}\tilde{d}_{n,j}^{1/2}n\sum_{\ell\leq\lceil n^{1/(2\rr)}\rceil}\frac{\ell^{-2\rr+2\rho}}{1+n\ell^{-2\rr}}+\omega_j{d_n^\delta}^{1/2}n\sumlnrgr\ell^{-2\rr+2\rho+1/2}\log(\ell)\notag\\
\lesssim~&\omega_j  {d_n^\delta}^{1/2}\tilde{d}_{n,j}^{1/2}n^{\frac{2\rho+1}{2\rr}}+\omega_j{d_n^\delta}^{1/2} \underbrace{n^{-\frac{\rr-\rho-1/2}{\rr-\rho-1}}\log(n)}_{\lesssim 1}\lesssim\omega_j  \tilde{d}_{n,j}^{1/2}.
 \end{align}
In \eqref{eq:vnsecond} we also used ${d_n^\delta}^{1/2}n^{(2\rho+1)/(2\rr)}\lesssim 1$ following \eqref{eq:dnestimate} in the last line. Inserting \eqref{eq:vnfirst} and \eqref{eq:vnsecond} into \eqref{eq:vnestimate} and using \eqref{eq:dnestimate} gives 
\begin{align}\label{eq:vnfinal}
     \lVert\vN\rVert_2^2&\lesssim  n\max\Bigl\{1,n^{-\frac{2\qq-2\rho-2}{2\rr}}\Bigr\}\sumjN\frac{\kappa_j^2\omega_j^2\tilde{d}_{n,j}}{1+\alphajj n\kappa_j^2}\notag\\
     &= n\max\Bigl\{1,n^{-\frac{2\qq-2\rho-2}{2\rr}}\Bigr\}\biggl({d_n^\delta}\sumjnr \frac{\kappa_j^2\omega_j^2}{1+\alphajj n\kappa_j^2}+\sumjnrgr \frac{\kappa_j^2\omega_j^2}{1+\alphajj n\kappa_j^2}\biggr)\notag\\
     &\lesssim n\max\Bigl\{1,n^{-\frac{2\qq-2\rho-2}{2\rr}}\Bigr\}\left({d_n^\delta} n^{-\frac{2\rr-2\rho-1}{2\rr}}+\underbrace{n^{-\frac{2\rr-2\rho-1}{\rr-\rho-1}}}_{\leq n^{-2}\lesssim {d_n^\delta} n^{-1}}\right)\notag\\
     &\lesssim  \max\Bigl\{1,n^{-\frac{2\qq-2\rho-2}{2\rr}}\Bigr\} {d_n^\delta} n^{\frac{2\rho+1}{2\rr}}\notag\\
     &\lesssim \max\Bigl\{1,n^{-\frac{2\qq-2\rho-2}{2\rr}}\Bigr\} \min\Bigl\{n^{-\frac{4\rho+1}{2\rr}}\delta^2,n^{-\frac{4\rho+2}{2\rr}}\Bigr\}n^{\frac{2\rho+1}{2\rr}}\notag\\
     &\lesssim\max\{\delta^2, \underbrace{n^{-\frac{2\qq-1}{2\rr}}}_{\leq\delta_n^2\lesssim\delta^2}\}\lesssim\delta^2.
\end{align}
Equations \eqref{eq:twofactor} and \eqref{eq:vnfinal} show the bound  
\begin{align}
     \hyperlink{hyp:firstterm}{(i)}&\lesssim \delta^2 \label{eq:iestimate}
 \end{align}
 for $(\px,\pz)\in S^\delta$, $\delta_n\lesssim\delta\leq n^{3/2}$ and $n\geq n_0$ and thus finish \hyperlink{hyp:step3a}{Step 3a}.
 
{\bf \hypertarget{hyp:step3b}{Step 3b} (estimation of {\hyperlink{hyp:secondterm}{$(ii)$}}).} To estimate $(ii)$ from \eqref{eq:convonetwosplit} on $S^\delta$, we split
 \begin{align}\label{eq:iisplit}
     (ii)&=\left\lVert\DNinv\bN-\GzeroN\right\rVert_2^2\\
     &=\sumjN\left(\frac{n\kappaj^2\betaj}{1+\alphajj n\kappaj^2}-\Gzj\right)^2\notag\\
     &=\sumjN \left(\frac{n\kappaj^2(\betaj-\Gzj)}{1+\alphajj n\kappaj^2}+\frac{n\kappaj^2(1-\alphajj)\Gzj}{1+\alphajj n\kappaj^2}-\frac{\Gzj}{1+\alphajj n\kappaj^2}\right)^2\notag\\
     &\leq 3\sumjN\left(\frac{n^2\kappaj^4(\betaj-\Gzj)^2}{\left(1+\alphajj n\kappaj^2\right)^2}+\frac{n^2\kappaj^4(1-\alphajj)^2\Gzj^2}{\left(1+\alphajj n\kappaj^2\right)^2}+\frac{\Gzj^2}{\left(1+\alphajj n\kappaj^2\right)^2}\right).\notag
 \end{align}
 We now bound the summands in \eqref{eq:iisplit} separately. For the first summand we use \eqref{eq:betaj} to estimate
 \begin{align}
     &\sumjN \frac{n^2\kappaj^4
     (\betaj-\Gzj)^2}{\left(1+\alphajj n\kappaj^2\right)^2}\notag\\
\lesssim~&n^2{{d_n^\delta}}\sum_{j\leq\lceil n^\frac{1}{\rr-\rho-1}\rceil}\frac{\kappa_j^4\omega_j^2}{\left(1+\alphajj n\kappa_j^2\right)^2}+n^2{d_n^\delta}\sum_{j>\lceil n^\frac{1}{\rr-\rho-1}\rceil}\frac{\kappa_j^4\omega_j^2j\log(j)^2}{\left(1+\alphajj n\kappa_j^2\right)^2} \notag\\
\lesssim~&n^2{{d_n^\delta}}\sum_{j\leq\lceil n^\frac{1}{2\rr}\rceil}\frac{j^{-4\rr+2\rho}}{\left(1+nj^{-2\rr}\right)^2}+n^2{d_n^\delta}\sum_{j>\lceil n^\frac{1}{\rr-\rho-1}\rceil}j^{-4\rr+2\rho+1}     \log(j)^{2}\notag\\
     \lesssim~&n^2n^{-\frac{4\rr-2\rho-1}{2\rr}}{{d_n^\delta}}+n^2 \underbrace{\log(n)^{2}n^{-\frac{4\rr-2\rho-2}{\rr-\rho-1}}}_{\lesssim n^{-\frac{4\rr-2\rho-1}{2\rr}}\text{ for }\rr>1+\rho}{d_n^\delta}\notag\\
     \lesssim~& n^{\frac{2\rho+1}{2\rr}}{d_n^\delta}\lesssim\delta^2,\label{eq:iifirst}
 \end{align}
 where we used \cite[Lemma B.2]{de2023convergence} and \eqref{eq:dnestimate}.

 The second term in \eqref{eq:iisplit} can be estimated in a similar way as \eqref{eq:iifirst} using 
\begin{align}
    |\alphajj-1|^2\lesssim\begin{cases}
       \omegaj^4 {d_n^\delta},~~j\leq \lceil n^{1/(\rr-\rho-1)}\rceil, \\
        4\omegaj^4,~~j> \lceil n^{1/(\rr-\rho-1)}\rceil.
    \end{cases}\label{eq:alphajjestimate}
\end{align}
Applying again \cite[Lemma B.2]{de2023convergence} and using \eqref{eq:dnestimate} shows
 \begin{align}
     &\sumjN \frac{n^2\kappaj^4
     (1-\alphajj)^2 g_{0,j}^2}{\left(1+\alphajj n\kappaj^2\right)^2}\notag\\
\lesssim~&n^2{{d_n^\delta}}\sum_{j\leq\lceil n^{\frac{1}{\rr-\rho-1}}\rceil}\frac{\kappa_j^4 g_{0,j}^2\omega_j^4}{\left(1+\alphajj n\kappa_j^2\right)^2}+n^2\sum_{j>\lceil n^\frac{1}{\rr-\rho-1}\rceil}\frac{\kappa_j^4 g_{0,j}^2\omega_j^4}{\left(1+\alphajj n\kappa_j^2\right)^2}\notag\\
\lesssim~&n^2{{d_n^\delta}}\sum_{j\leq\lceil n^{\frac{1}{\rr-\rho-1}}\rceil}\frac{j^{-4\rr-2\qq+4\rho}}{\left(1+nj^{-2\rr}\right)^2}+n^2\sum_{j>\lceil n^\frac{1}{\rr-\rho-1}\rceil}j^{-4\rr-2\qq+4\rho}\notag\\
     \lesssim~& n^2 \max\left\{n^{-2},n^{-\frac{4\rr+2\qq-4\rho-1}{2\rr}}\right\}{d_n^\delta}+n^2n^{-\frac{4\rr+2\qq-4\rho-1}{\rr-\rho-1}}\notag\\
\lesssim~& n^{\frac{4\rho+1}{2\rr}}{d_n^\delta}+\underbrace{n^{-\frac{2\rr+2\qq-2\rho+1}{\rr-\rho-1}}}_{\leq n^{-1}\leq\delta_n^2\lesssim\delta^2}\lesssim \delta^2.
 \end{align}
 For the third summand we get
 \begin{align}
     \sumjN\frac{\Gzj^2}{\left(1+\alphajj n \kappaj^2\right)^2}&\lesssim\sum_{j\leq\lceil n^{1/(2\rr)}\rceil}\frac{j^{-2\qq}}{(1+nj^{-2\rr})^2}
     \lesssim \max\bigl\{n^{-2},n^{-\frac{2\qq-1}{2\rr}}\bigr\}\leq\delta_n^2\lesssim\delta^2\label{eq:iithird}
 \end{align}
 again using \cite[Lemma B.2]{de2023convergence}. Combining \eqref{eq:iifirst}--\eqref{eq:iithird} shows 
 \begin{align}
     \hyperlink{hyp:secondterm}{(ii)}\lesssim \delta^2
     \label{eq:iiestimate}
 \end{align}
 completing \hyperlink{hyp:step3b}{Step 3b}.
Using \eqref{eq:convonetwosplit}, \eqref{eq:iestimate} and \eqref{eq:iiestimate} shows that for $n\geq n_0$ and $\delta_n\lesssim\delta\leq n^{3/2}$ we have
\begin{align}\label{eq:iandiiestimate}
    \lVert\hatgn-g_0\rVert_{L^2(K)}^2\lesssim \delta^2~~\text{on }S^{\delta}
\end{align}
which finishes \hyperlink{hyp:step3}{Step 3}.

{\bf \hypertarget{hyp:step4}{Step 4} (upper bound $\boldsymbol{\lVert\hatgn-g_0\rVert_{L^2(K)}^2}$ on $\boldsymbol{S^\delta}$ for $\boldsymbol{\delta\geq n^{3/2}}$).}
Following \eqref{eq:gnhatg0}, we have 
\begin{align}\label{eq:Scbound}
\lVert\hatgn-g_0\rVert_{L^2(K)}^2&=\limNinfty\left\|\ANin\bN-\GzeroN \right\|_2^2\notag\\
   &\leq2\limNinfty\left\|\ANin\bN \right\|_2^2
   +2\sum_{j=1}^{\infty}\underbrace{\Gzj^2}_{\lesssim j^{-2\qq}}\lesssim \limNinfty\left\|\ANin\bN\right\|_2^2+1,   
\end{align}
where we used $\qq>1/2$.
Define $v_N'=\DNhalfinv\bN$. Using $\lVert\ANin\rVert_2\lesssim n$ and ${\AN}^{-2}\preceq \lVert\AN^{-1}\rVert \ANin$ (here $\preceq $ denotes the Loewner order, e.g.,~\cite[Section 7.7]{horn2012matrix}) we have
\begin{align}\label{eq:matrixanalysis}
\left\|\ANin\bN\right\|_2^2&\leq\lVert\ANin\rVert_2~\lVert\DNhalf\ANin\DNhalf\rVert_2~\lVert v_N'\rVert_2^2\lesssim n\lVert\DNhalf\ANin\DNhalf\rVert_2~\lVert v_N'\rVert_2^2.
\end{align}
Since $\DNhalf\ANin\DNhalf$ is symmetric positive definite we have
\begin{align}
\label{eq:eigenvalue}
\left\lVert\DNhalf\ANin\DNhalf\right\rVert_2&=\lambda_{\min}\left(\DNhalfinv\AN\DNhalfinv \right)^{-1}\notag\\
&=\left(\inf_{\lVert x_N\rVert_2=1}x_N^T\DNhalfinv\AN\DNhalfinv x_N\right)^{-1}\notag\\
&\leq \left(\inf_{\lVert x_N\rVert_2=1}x_N^T\DNhalfinv\diag\left(n^{-1}\kappaj^{-2}\right)\DNhalfinv x_N\right)^{-1}\notag\\
&\leq \left(\inf_{\lVert x_N\rVert_2=1}x_N^T\diag\left(1+\alphajj n\kappaj^{2}\right)^{-1} x_N\right)^{-1}\lesssim n.
\end{align}
For $v_N'$ we use \eqref{eq:betaj} to get

\begin{align}
    \Vert v_N'\rVert_2^2
    =\sumjN\frac{n\kappaj^2\betaj^2}{1+\alphajj n\kappaj^2}
    \lesssim n\sumjN \kappaj^2\Gzj^2+n{d_n^\delta}\sumjN \max\{1,j\log(j)^2\}\kappaj^2\omegaj^2
    \lesssim n(1+{d_n^\delta}).\label{eq:vNprimeestimate}
\end{align}

Inserting \eqref{eq:vNprimeestimate} and \eqref{eq:eigenvalue} into \eqref{eq:matrixanalysis} and then \eqref{eq:Scbound} shows for $\delta\geq n^{3/2}$
\begin{align}
    \label{eq:scboundfinal}
   \lVert\hatgn-g_0\rVert_{L^2(K)}^2&\lesssim n^3(1+{d_n^\delta})\lesssim\delta^2
\end{align}
and thus finishes \hyperlink{hyp:step4}{Step 4}.

{\bf \hypertarget{hyp:step5}{Step 5} (finish the proof).} Let $n\in\IN$, $n\geq n_0$ for $n_0\in\IN$ sufficiently large and $\delta\geq\delta_n$. Denote the constant hidden in $\lesssim$ from \eqref{eq:iandiiestimate} and \eqref{eq:scboundfinal} by $\overline{C}^2 >0$. 
Then it holds
\begin{align}
&\left(\IPpX\otimes\IP_{\pEE}\right)\left(\left\lVert\hatgn-g_0\right\rVert_{L^2(K)}^2\geq\delta^2\right)\leq1-\left(\IPpX\otimes\IP_{\pEE}\right)\left(S^{\overline{C}^{-1}\delta}\right)\leq C\exp\left({-\frac{n{d_n^\delta}}{C}}\right),\label{eq:concentrationbprooffinal}
\end{align}
where we used \eqref{eq:Smeasure} with $\overline{C}^{-1}\delta\gtrsim\delta_n$ and absorbed $\overline{C}$ into $C$. 
Choosing $C\geq \exp(n_0)>1$ then shows \eqref{eq:convergencegaussian1} and \eqref{eq:convergencegaussian2} for $n\leq n_0$. Since \eqref{eq:scboundfinal} holds for $n\in\IN$ this shows \eqref{eq:concentrationbprooffinal} for $n\in\IN$ with a potential larger $C$.
Writing out the definition of ${d_n^\delta}$ shows \eqref{eq:convergencegaussian}. In particular $C$ depends on $g_0$ only through $\lVert g_0\rVert_\infty$, see \eqref{eq:betajfinal} and \eqref{eq:Sdeltag}, and $\sup_{j\in\IN}j^{\qq}|\Gzj|$, see \eqref{eq:rtassump}.  This finishes the proof
of Theorem \ref{th:convergencetheorem}.

\subsection{Proof of Lemma \ref{lem:abstractlmap}}\label{app:abstractlmap}
Let $g_j\coloneqq \langle g,\psi_j\rangle_{L^2(K)}$ and $C_\kappa\coloneqq \sumj \omega_j^2\kappa_j^2$. Then the decay property \eqref{eq:strongerdecayproperty} shows $C_\kappa<\infty$ and together with Cauchy--Schwarz we get
\begin{align}\label{eq:HpriorCzero}
    \biggl\|\sumj g_j\psi_j\biggr\|_\infty\le\sum_{j\in\IN}\omega_j|g_j|
    \le C_\kappa^{1/2} \|g\|_{H_{\rm{prior}}},\qquad g\in H_{\rm{prior}}
\end{align}
which shows $H_{\rm{prior}}\subseteq C^0(K)$.
With $\tomega_j\coloneqq\omega_j\max\{1,\log(j)^{1/2}\}$, $j\in\IN$,
Lemma \ref{lem:envelopelemma} and \eqref{eq:strongerdecayproperty} show that $(x,z)$ $(\IPX\otimes\IPE)$-a.s.~
\begin{align*}
  C_{x,z}^2\coloneqq\sum_{j\in\IN}\kappa_j^2\langle z,\psi_j(x)\rangle_\sim^2
  \leq M(x,z)^2\sum_{j\in\IN}\kappa_j^2\tilde \omega_j^2<\infty .
\end{align*}
By Cauchy--Schwarz we have $(x,z)$ $(\IPX\otimes\IPE)$-a.s.
\begin{align*}
\sum_{j\in\IN}|g_j\langle z,\psi_j(x)\rangle_\sim|
\leq \left(\sum_{j\in\IN}\kappa_j^2\langle z,\psi_j(x)\rangle_\sim^2\right)^{1/2}
\left(\sum_{j\in\IN}\kappa_j^{-2}g_j^2\right)^{1/2}
\leq C_{x,z}\lVert g\rVert_{H_{\rm{prior}}}
\end{align*}
for all $g\in H_{\rm{prior}}$. Thus $H_{\rm{prior}}\subseteq\Lambda_{x,z}$ and
$\ell_{x,z}$ is bounded on $H_{\rm{prior}}$ with
$\lVert\ell_{x,z}\rVert\leq C_{x,z}$ for $(x,z)$ $(\IPX\otimes\IPE)$-a.s. Clearly $\Lambda_{x,z}$ is a linear subspace and $\IP_G(\Lambda_{x,z})=1$ because, by the argumentation after \eqref{eq:envelopebound},
$\sum_{j\in\IN}\tomega_j|g_j|<\infty$ holds $\IP_G$-a.s.
Lastly, \eqref{eq:lmapidentity} follows from \eqref{eq:pointwiseexpansion}.

\subsection{Proof of Theorem \ref{th:MAPestimator}}\label{app:MAPestimator}
 For $x_i\in K$ and $z_i\in\tilde\Y^*$, $i=1,\dots,n$, $n\in\IN$, consider the maps
\begin{align*}
 &B:C^0(K)\times C^0(K)\to\IR,\quad B(g,h)=\sum_{i=1}^n\langle g(x_i),h(x_i)\rangle_\Y,\\
 &\ell: \Lambda_{\px,\pz}\to\IR,\quad \ell(g)=\sum_{i=1}^n\ell_{x_i,z_i}(g)
\end{align*}
with $\ell_{x,z}$ and $\Lambda_{\px,\pz}=\bigcap_{i=1,\dots, n}\Lambda_{x_i,z_i}$ from Lemma \ref{lem:abstractlmap}. In particular $\IP_G(\Lambda_{\px,\pz})=1$ and Theorem
\ref{th:PosteriorGaussianNoise} and Lemma \ref{lem:abstractlmap} give $g$ $\IP_G$-a.s.
\begin{align}\label{eq:posteriorquadraticdensity}
    \varrho(g)\coloneqq\frac{\dd\IP_{G|(\pXX,\pZZ)}}{\dd\IP_G}(g)=\mathfrak Z(\px,\pz)^{-1}
    \exp\bigl(-\tfrac12 B(g,g)+\ell(g)\bigr).
\end{align}
Then \cite[Example 5.10.17]{bogachev1998gaussian} shows that $\IP_{G|(\pXX,\pZZ)}$ is Gaussian and \cite[p.~179]{bogachev2007measure} yields $\IP_{G|(\pXX,\pZZ)}\sim\IP_G$ (since $\varrho>0$ on $\Lambda_{\px,\pz}$) which gives $H_{\text{post}}=H_{\rm{prior}}$ as sets by
\cite[Proposition 2.7.3 (i)]{bogachev1998gaussian}. 

Now fix $h\in H_{\text{post}}=H_{\rm{prior}}$ and let $\hat h_{\text{post}}$ denote the
measurable linear extension of $\dup[H_{\text{post}}]{\cdot}{h}$ onto $C^0(K)$. Since $\IP_{G|(\pXX,\pZZ)}$ is Gaussian with mean $\gnhat$ and CM space $\Hpost$, the Cameron--Martin formula (e.g.~\cite[Corollary 2.4.3]{bogachev1998gaussian}) reads
\begin{align*}
    \frac{\dd\IP_{G|(\pXX,\pZZ)}(\cdot-h)}{\dd\IP_{G|(\pXX,\pZZ)}}(g)
    =\exp\Bigl(\hat h_{\text{post}}(g-\gnhat)
    -\tfrac12\lV h\rV^2_{H_{\text{post}}}\Bigr),
\end{align*}
and, as $h\mapsto\hat h_{\text{post}}$ is linear, it holds $g$ $\IP_G$-a.s.
\begin{align}\label{eq:CMpost}
    \frac{\dd\IP_{G|(\pXX,\pZZ)}(\cdot-h)}{\dd\IP_{G|(\pXX,\pZZ)}}(g)
    \frac{\dd\IP_{G|(\pXX,\pZZ)}(\cdot+h)}{\dd\IP_{G|(\pXX,\pZZ)}}(g)
    =\exp\bigl(-\lV h\rV^2_{H_{\text{post}}}\bigr).
\end{align}
On the other hand, $B(g\pm h,g\pm h)=B(g,g)\pm2B(g,h)+B(h,h)$ and the linearity
of $\ell$ yield
$\varrho(g-h)\varrho(g+h)/\varrho(g)^2=\exp(-B(h,h))$ for
$g$ $\IP_G$-a.s., so the Cameron--Martin formula for
$\IP_G$ and the chain rule for Radon--Nikodym derivatives
(\cite[p.~179]{bogachev2007measure}) show $g$ $\IP_G$-a.s.
\begin{align}
    \frac{\dd\IP_{G|(\pXX,\pZZ)}(\cdot-h)}{\dd\IP_{G|(\pXX,\pZZ)}}(g)
     \frac{\dd\IP_{G|(\pXX,\pZZ)}(\cdot+h)}{\dd\IP_{G|(\pXX,\pZZ)}}(g)
    &=\frac{\varrho(g-h)\varrho(g+h)}{\varrho(g)^2}\,
     \frac{\dd\IP_G(\cdot-h)}{\dd\IP_G}(g)\,\frac{\dd\IP_G(\cdot+h)}{\dd\IP_G}(g)\notag\\
    &=\exp\bigl(-B(h,h)-\lV h\rV^2_{H_{\rm{prior}}}\bigr).\label{eq:CMprior}
\end{align}
Comparing \eqref{eq:CMpost} with \eqref{eq:CMprior} and using the polarization identity shows
\begin{align}\label{eq:CMnormpost}
    \dup[H_{\text{post}}]{g}{h}=\dup[H_{\rm{prior}}]{g}{h}
    +\sum_{i=1}^n\dup[\Y]{g(x_i)}{h(x_i)},\qquad g,h\in H_{\rm{prior}}.
\end{align}
In particular $\lV\cdot\rV_{H_{\rm{prior}}}\le\lV\cdot\rV_{H_{\text{post}}}$
and, by \eqref{eq:HpriorCzero},
$\lV\cdot\rV^2_{H_{\text{post}}}\le(1+nC_\kappa)\lV\cdot\rV^2_{H_{\rm{prior}}}$,
so the two norms are equivalent.

Now, for $N\in\IN$, consider $\AN$ and $\bN$ from \eqref{eq:alphajk} and \eqref{eq:betanuj}. By \eqref{eq:CMnormpost} we
have $\dup[\Hpost]{\psi_j}{\psi_k}=nA_{N,jk}$ for $j,k\le N$, while
$\ell(\psi_j)=n\beta_j$ by \eqref{eq:lmap}. Since $\ell$ is
bounded on $\Hpost$, the Riesz representation theorem provides a unique
$m\in\Hpost$ with $\dup[\Hpost]{m}{\cdot}=\ell$, and
\begin{align*}
    \dup[\Hpost]{\bs{\psi}_N^\top\AN^{-1}\bN}{\psi_j}
    =n\bigl(\AN\AN^{-1}\bN\bigr)_j=n\beta_j=\dup[\Hpost]{m}{\psi_j},
    \qquad j=1,\dots,N,
\end{align*}
i.e.\ $\bs{\psi}_N^\top\AN^{-1}\bN$ is the $\Hpost$-orthogonal projection of $m$ onto
$\operatorname{span}\{\psi_1,\dots,\psi_N\}$. As $(\kappa_j\psi_j)_{j\in\IN}$ is an
ONB of $\Hprior$ and the norms $\lV\cdot\rV_{H_{\rm{prior}}}$ and $\lV\cdot\rV_{H_{\text{post}}}$ are equivalent, these projections converge to $m$ in
$\Hpost$, hence in $\Hprior$ and, by \eqref{eq:HpriorCzero}, in $C^0(K)$. Then Lemma
\ref{lem:calcgaussiangnhat} shows $\gnhat=m\in\Hpost$ and
\begin{align}\label{eq:gnhatRiesz}
    \dup[\Hpost]{\gnhat}{h}=\ell(h),\qquad h\in\Hpost.
\end{align}
Using e.g.~\cite[Corollary 5.4]{ayanbayev2022gamma} yields that
\begin{align}
   \tilde I: \mathcal G\to \IR\cup\{\infty\},~~ \tilde I (g)=\begin{cases}
    \frac{1}{2}\lVert g-\gnhat\rVert_{H_{\text{post}}}^2,~~ &g \in H_{\rm{prior}}\\
   \infty,~~&\text{otherwise},
   \end{cases}\label{eq:tikhonovOMtilde}
\end{align}
defines an OM functional of the posterior $\IP_{G|(\pXX,\pZZ)}$.%
Setting $I\coloneqq\tilde I-\tfrac12\lVert\gnhat\rVert^2_{H_{\text{post}}}$,
\eqref{eq:OMfunctional} shows that $I$ is also an OM functional of the posterior. 
Using \eqref{eq:CMnormpost} and \eqref{eq:gnhatRiesz} we get for $g\in H_{\rm{prior}}$,
\begin{align*}
2I(g)=\lV g\rV_{H_{\text{post}}}^2-2\langle \gnhat,g\rangle_{\Hpost}=  \sum_{i=1}^n\left(\lVert g(x_i)\rVert_\Y^2-2\ell_{x_i,z_i}(g)\right)+\lVert g\rVert_{H_{\rm{prior}}}^2.
\end{align*}
It remains to show that the posterior mean $\gnhat$ is the unique MAP estimator. Define $\mu\coloneqq\IP_{G|(\pXX,\pZZ)}(\cdot,\px,\pz)$ and $\mu_0\coloneqq\mu(\cdot+\gnhat)$. As $B_\delta(0)$ is convex and symmetric, Anderson's inequality \cite[Theorem~2.8.10]{bogachev1998gaussian} gives
\[
J_\delta(g)=\mu_0\bigl(B_\delta(0)+g-\gnhat\bigr)\le\mu_0\bigl(B_\delta(0)\bigr)=J_\delta(\gnhat),\qquad g\in C^0(K),\ \delta>0 .
\]
Hence $\gnhat$ maximizes $J_\delta$ for every $\delta>0$, and thus $\gnhat$ is a MAP estimator. Conversely, any MAP estimator $g^\star$ is a so-called \emph{global weak mode} of $\mu$ (see \cite[Definition~3.7]{ayanbayev2022gamma} for a definition) by \cite[Lemma~3.9]{ayanbayev2022gamma}. Then \cite[Proposition~4.1]{ayanbayev2022gamma} and \cite[Lemma 3.7]{dashti2013bayesian} show that $g^\star$ minimizes $I$, and therefore $g^\star=\gnhat$.

\subsection{\texorpdfstring{Proof of Theorem \ref{th:minimax} }{Proof of Minimax Optimality}}\label{app:lowerbound}

{\bf Step 1 (proof of \ref{item:minimaxupper}).} Since $\qq>1$ and $\rho=0$, we have $\Theta_{\qq}\subset C^0(K)$ with $\sup_{g\in\Theta_{\qq}}\lVert g\rVert_{C^0(K)}\leq M\sumj j^{-\qq}<\infty$. Moreover, \eqref{eq:rtassump} holds for all $g_0\in\Theta_\qq$ with hidden constant $1$. Hence the constant in \eqref{eq:convergencegaussianIE} can be chosen uniformly over $\Theta_{\qq}$, and Corollary \ref{cor:Expectationconvergence} with $\rho=0$ and $\rr=\qq>1$ yields \eqref{eq:upperboundhyperrect}.

{\bf Step 2 (proof of \ref{item:minimaxlower}).} For $g\in C^0(K)$ denote by $\IP^{g}$ the law of $(\pXX,\pZZ)=(\pXX,d\circ g(\pXX)+\pEE)$ under $\IPpX\otimes\IPpE$. Following Corollary \ref{cor:existencetildemapping} and Theorem \ref{th:PosteriorGaussianNoise}, $\IP^{g}$ has density $\exp(-\pPhi(g,\cdot))$ w.r.t.~$\IPpX\otimes\IPpE$ with $\pPhi$ from \eqref{eq:phigeneral}, and $\langle d(y)+E,\bar y\rangle_\sim=\langle y,\bar y\rangle_\Y+\IW_\Y(\bar y)$ almost surely for $y,\bar y\in\Y$ and $E\sim\IPE$. We calculate the Kullback--Leibler (KL) divergence $D_{KL}$ (see \cite[p.~84]{tsybakov2009introduction})
\begin{align}    D_{KL}(\IP^{g},\IP^{g'})&=\int\log\left(\frac{\dd\IP^{g}}{\dd\IP^{g'}}\right)\dd\IP^{g}=\int\log\left(\frac{\dd\IP^{g}}{\dd(\IPpX\otimes\IPpE)}\frac{\dd(\IPpX\otimes\IPpE)}{\dd\IP^{g'}}\right)\dd\IP^{g}\notag\\
    &=\IE_{(\pXX,\pEE)}\left[\pPhi\big(g',\pXX,d\circ g(\pXX)+\pEE\big)-\pPhi\big(g,\pXX,d\circ g(\pXX)+\pEE\big)\right]\notag\\
    &=n\,\IE_{X\sim\IPX}\left[\frac12\lVert g'(X)\rVert_\Y^2-\frac12\lVert g(X)\rVert_\Y^2+\big\langle g(X),g(X)-g'(X)\big\rangle_\Y\right]\notag\\
    &=\frac n2\lVert g-g'\rVert^2_{L^2(K)},\quad g,g'\in C^0(K).\label{eq:KLwhitenoise}
\end{align}

Fix $n\in\IN$ and set
\begin{align*}
    m\coloneqq\left\lceil n^{\frac{1}{2\qq}}\right\rceil,\qquad a\coloneqq m^{-\qq},\qquad \T\coloneqq\{0,1\}^m,\qquad g_\tau\coloneqq a\sum_{j=1}^m\tau_j\psij,\quad\tau\in\T.
\end{align*}
In the following let $d_H$ denote the Hamming distance (see \cite[p.~103]{tsybakov2009introduction}),
\begin{align*}
d_H:\T\times\T\to\{0,1,\dots,m\},\quad d_H(\tau,\tau')=\sum_{j=1}^m\lvert \tau_j-\tau_j'\rvert.
\end{align*} 
Denote by $\IE^{g_\tau}$ the expectation w.r.t.~$\IP^{g_\tau}$.
Observe $g_\tau\in\Theta_q$ for all $\tau\in\T$. By \eqref{eq:KLwhitenoise} we have
\begin{align*}
    D_{KL}(\IP^{g_\tau},\IP^{g_{\tau'}})=\frac n2\lVert g_\tau-g_{\tau'}\rVert^2_{L^2(K)}, \quad \tau,\tau'\in\T.
\end{align*}
As $(\psij)_{j\in\IN}$ is orthonormal, we have $\lVert g_\tau-g_{\tau'}\rVert^2_{L^2(K)}= a^2d_H(\tau,\tau')$ and thus
\begin{align*}
    D_{KL}(\IP^{g_\tau},\IP^{g_{\tau'}})=\frac n2\,m^{-2\qq}\leq\frac12\qquad \text{for }\tau,\tau' \text{ with } d_H(\tau,\tau')=1
\end{align*}
using $m=\lceil n^{1/(2\qq)}\rceil$.
Applying \cite[Theorem 2.12 (iv)]{tsybakov2009introduction} shows
\begin{align}\label{eq:hamminglowerbound}
    \inf_{\tilde\tau}\max_{\tau\in\T}\IE^{g_\tau}\left[d_H(\tilde\tau,\tau)\right]\geq\frac m4.
\end{align}
Now let $\tilde g_n$ be an arbitrary estimator and define
\begin{align*}
    \tilde\tau\coloneqq(\tilde\tau_1,\dots,\tilde\tau_m)\in\T,\quad \tilde\tau_j\coloneqq\mathds 1_{\left\{\langle\tilde g_n,\psij\rangle_{L^2(K)}> a/2\right\}},\quad j=1,\dots,m.
\end{align*}
If $\tilde\tau_j\neq\tau_j$ for $\tau\in\T$, then $|\langle\tilde g_n,\psij\rangle_{L^2(K)}- a\tau_j|\geq a/2$. Parseval yields
\begin{align*}
    \lVert\tilde g_n-g_\tau\rVert^2_{L^2(K)}\geq\sum_{j=1}^m\left(\langle\tilde g_n,\psij\rangle_{L^2(K)}- a\tau_j\right)^2\geq\frac{ a^2}{4}d_H(\tilde\tau,\tau),\qquad\tau\in\T.
\end{align*}
Taking expectations and using $\{g_\tau:\tau\in\T\}\subset\Theta_{\qq}$ shows with \eqref{eq:hamminglowerbound} that
\begin{align}
    \sup_{g_0\in\Theta_{\qq}}\IE_{(\pXX,\pEE)}\left[\lVert\tilde g_n-g_0\rVert^2_{L^2(K)}\right]&\geq\max_{\tau\in\T}\IE^{g_\tau}\left[\lVert\tilde g_n-g_\tau\rVert^2_{L^2(K)}\right]\notag\\
    &\geq\frac{ a^2}{4}\inf_{\tilde\tau}\max_{\tau\in\T}\IE^{g_\tau}\left[d_H(\tilde\tau,\tau)\right]\geq \frac{ a^2m}{16}.\label{eq:lowerboundfinitedim}
\end{align}
Using the definition of $m$ and $ a$ and taking the infimum over $\tilde g_n$ in
\eqref{eq:lowerboundfinitedim} shows \eqref{eq:lowerboundhyperrect}. This finishes the proof of Theorem \ref{th:minimax}.

\subsection{Envelope Lemma}
\begin{lemma}\label{lem:envelopelemma}
    Consider the map $\langle\cdot,\cdot\rangle_\sim$ from Corollary \ref{cor:existencetildemapping} and a continuous ONB $(\psi_j)_{j\in\IN}$ of $L^2(K)$ with $\omega_j\coloneqq\lVert \psi_j\rVert_{\infty}$ for $j\in\IN$. Then it holds for $x\in K$ and $z\in\tilde\Y^*$,
    \begin{align*}
        \lvert\langle z,\psi_j(x)\rangle_\sim\rvert\leq M\omega_j\max\left\{1,\log(j)^{1/2}\right\},\quad j\in\IN
    \end{align*}
    where $M=M(x,z)$ is $(x,z)$ $(\IPX\otimes\IPE)$-a.s.~finite.
\end{lemma}
\begin{proof}
    We aim to use the Borel--Cantelli lemma, e.g.~\cite[Theorem 4.18]{kallenberg1997foundations}. Let therefore $x\in K$ be arbitrary. Abbreviate $m_j\coloneqq\max\left\{1,\log(j)^{1/2}\right\}$ for $j\in\IN$. We define the sets
\begin{align*}
    A_j\coloneqq \left\{z\in\tilde\Y^*: \left|\frac{\langle z,\psi_j(x)\rangle_\sim}{\omega_j m_j}\right|\geq 2\right\},\quad j \in\IN.
\end{align*}
Since $\langle\cdot,\psi_j(x)\rangle_\sim$ has law $\N(0,\lVert\psi_j(x)\rVert_\Y^2)$ under $\IPE$, see \eqref{eq:WN} and \eqref{eq:tildemapping1}, we have for $\psi_j(x)\neq 0$
\begin{align*}
    \IPE(A_j)&\leq2\exp\left(-\frac{2\omega_j^2m_j^2}{\lVert\psi_j(x)\rVert_\Y^2}\right)
    \leq2j^{-2},
\end{align*}
where we used $\lVert\psi_j(x)\rVert_\Y\leq\omega_j$. Thus $\sum_{j\in\IN}\IPE(A_j)<\infty$. The Borel--Cantelli lemma (\cite[Theorem 4.18]{kallenberg1997foundations}) yields
\begin{align*}
    \IPE(z\in\tilde\Y^*:z\in A_j \text{ for finitely many }j\in\IN)=1.
\end{align*} 
Thus, for $(x,z)$ $\IP_X\otimes\IP_E$-a.s.~there exists $J=J(x,z)\in\IN$ such that
\begin{align}\label{eq:Jdefinition}
    \bigl\lvert\langle z,\psi_j(x)\rangle_\sim\bigr\rvert\leq 2\omega_j m_j,\quad j>J.
\end{align}
Choosing the $\IP_X\otimes\IP_E$-a.s.~finite constant
\begin{align*}
    M_0=M_0(x,z)=\max_{j=1,\dots,J} \frac{\bigl\lvert\langle z,\psi_j(x)\rangle_\sim\bigr\rvert}{\omega_j m_j}
\end{align*}
and setting $M\coloneqq\max\{2,M_0\}$ shows
\begin{align*}
    \bigl\lvert\langle z,\psi_j(x)\rangle_\sim\bigr\rvert\leq M\omega_j m_j,\quad j\in\IN.
\end{align*}
Choosing the minimum $J$ in \eqref{eq:Jdefinition}, and $M=\infty$ if no such $J$ exists, $(x,z)\mapsto J(x,z)$ and thus $(x,z)\mapsto M(x,z)$ are jointly measurable (by Corollary \ref{cor:existencetildemapping} and continuity of $\psi_j$) and the claim follows with Tonelli as in the proof of Lemma \ref{lem:posteriorexpectationexistence}.
\end{proof}

\section{Proofs of Section \ref{sec:applications}}\label{app:proofapplication}

\subsection{Proof of Lemma \ref{lem:ltwobasis}}\label{app:Ltwobasisproof}
    Since $\Y$ is a Hilbert space, e.g.~\cite[Theorem 12.6.1]{aubin2000applied} shows that $L^2(\CRr(\X),\IPX;\Y)$ can be isometrically identified with $L^2(\CRr(\X),\IPX;\IR)\otimes\Y$, where $\otimes$ denotes the tensor product. As $\sigmaRr$ is injective with measurable inverse, it suffices to show that $\Lnu\circ\left(\sigmaRr\right)^{-1}$ is an ONB of $L^2(\CRr(\X),\IPX;\IR)$. Let therefore $F\in L^2(\CRr(\X),\IPX;\IR) $. Since $\IPX=\left(\sigmaRr\right)_\#\pi$ we have $F\circ\sigmaRr\in L^2(U,\pi,\IR)\eqqcolon L^2(\pi)$. Because $(\Lnu)_{\pnu\in\F}$ is an ONB of $L^2(\pi)$, see, e.g.,~\cite[Theorem 2.12]{schwab2011sparse}, there exists $\left(\xi_{\pnu}\right)_{\pnu\in\F}\in\ell^2(\F)$ s.t.
    \begin{align*}
        \limNinfty \Bigl\lVert F\circ\sigmaRr-\sum_{\substack{\pnu\in\F,\\|\pnu|_1\leq N}}\xi_{\pnu}\Lnu\Bigr\rVert_{L^2(\pi)}=0.
    \end{align*}
    This is equivalent to
    \begin{align*}
        \limNinfty \Bigl\lVert F-\sum_{\substack{\pnu\in\F,\\|\pnu|_1\leq N}}\xi_{\pnu}\Lnu\circ\left(\sigmaRr\right)^{-1}\Bigr\rVert_{L^2(\CRr(\X),\IPX;\IR)}=0,
    \end{align*}
    which yields that $\Lnu\circ\left(\sigmaRr\right)^{-1}$ is an ONB of $L^2(\CRr(\X),\IPX;\IR)$. This shows the claim.
\subsection{Proof of Theorem \ref{th:holodecay}}\label{app:proofholodecay}
For a proof of \ref{item:holodecayseq}, see \cite[Theorem 5 (iii)]{herrmann2024neural} which repeats arguments from \cite[Theorem 2.2.10 (iv)]{zech2018sparse}.
For \ref{item:holodecaylp} we use $\IPX=\left(\sigmaRr\right)_\#\pi$ and the definition of $\psinuk$ to show
\begin{align*}
    \Gznuk&=\int_{C_R^s(\X)}\langle g_0(x),\psinuk(x)\rangle_{\Y}\dd\IPX(x)\\
    &=\int_U\langle g_0\circ\sigmaRr(\pu),\psinuk\circ\sigmaRr(\pu)\rangle_{\Y}\dd\pi(\pu)=\int_U\Lnu(\pu)\langle g_0\circ\sigmaRr(\pu),\eta_k\rangle_{\Y}\dd\pi(\pu).
\end{align*}
The result is now an application of \cite[Proposition 2]{herrmann2024neural}.

\subsection{Proof of Lemma \ref{lem:holomakesconvergence}}\label{app:holomakesconvergence}
We first prove \ref{item:g0rate} by constructing an enumeration $i:\IN\to\F\times\IN$, $i(j)\eqqcolon (\nu,k)_j$ and then showing that $\omegaj=\omega_{(\nu,k)_j}$ and $\Gzj\coloneqq g_{0,(\nu,k)_j}$ satisfy \eqref{eq:g0rate}. We then prove \ref{item:generalrate}.

{\bf Step 1 (construction of enumeration).} Consider the sequence $(\anu)_{\nu\in\F}$ from Theorem \ref{th:holodecay} and let $1/2<\qq<\min\{\ssv,\ttv+1/2\}$. We define the sequence
$$(\cnuk)_{\pnu\in\F,k\in\IN},\quad\cnuk=\max\left\{\left\lvert\Gznuk\right\rvert,\anu k^{-\qq}\right\}.$$
Now let $i:\IN\to\F\times\IN$, $i(j)\coloneqq(\pnu,k)_j$ be any monotonically decreasing rearrangement of $\cnuk$, i.e.~
$$j\leq j'\implies c_j\coloneqq c_{(\pnu,k)_j}\geq c_{(\pnu,k)_ {j'}}\eqqcolon c_{j'}.$$
Let $p\in (1/\qq,2)$ and $\ttv'\in[\qq-1/2,\ttv)$, which is possible since $\qq<\ttv+1/2$. Since $q<\ssv$, we have $p>1/\qq>1/\ssv$ and thus $(\anu)_{\nu\in\F}\in\ell^p(\F)$ following Theorem \ref{th:holodecay}, applied with some $p_0\in(1/\ssv,1)$, $p_0\le p$. Moreover, since $\qq p>1$, we have $(\anu k^{-\qq})_{\pnu\in\F,k\in\IN}\in\ell^p(\F\times\IN)$. Theorem \ref{th:holodecay} \ref{item:holodecaylp} and $\theta_k\lesssim k^{-1}$ yield $\sumk k^{2\ttv'}\Gznuk^2\lesssim\anu^2$, so H\"older's inequality with exponents $2/p>1$ and $2/(2-p)>1$ gives
\begin{align}
    \sumk\left\lvert\Gznuk\right\rvert^p=\sumk\left(k^{\ttv'}\left\lvert\Gznuk\right\rvert\right)^pk^{-\ttv'p}\lesssim\anu^p\left(\sumk k^{-\frac{2\ttv'p}{2-p}}\right)^{\frac{2-p}{2}}\lesssim\anu^p,\label{eq:gzjestimate}
\end{align}
where the last series converges since $2\ttv'p/(2-p)>1$ by $t'\geq \qq-1/2$ and $p\qq>1$. Hence $(\Gznuk)_{\pnu\in\F,k\in\IN}\in\ell^p(\F\times\IN)$, and thus $(\cnuk)_{\pnu\in\F,k\in\IN}\in\ell^p(\F\times\IN)$. Because $(c_j)_{j\in\IN}$ is monotonically decreasing, \cite[Lemma 1.4.8]{zech2018sparse} yields $c_j\lesssim j^{-1/p}$.

Moreover, with the zero vector $\boldsymbol{0}=(0,0,\dots)\in\F$, we have  
\begin{align}
    \label{eq:cnulowerbound}
    c_j\geq \min_{k\leq j}~c_{\boldsymbol{0},k}\geq a_{\boldsymbol{0}}j^{-\qq}\gtrsim j^{-\qq}.
\end{align}

{\bf Step 2 (proof of \eqref{eq:g0rate}).}
Following \eqref{eq:omeganu}, we have $$\omeganu\coloneqq\omega_{\pnu,k}=\lVert\psinuk\rVert_{C^0(K)}\leq\prod_{\ell\in\supp\pnu}(1+2\nu_\ell)^{1/2}$$ and aim to bound $\omega_{(\pnu,k)_j}$ in terms of $j$ using the enumeration from above.

Fix $\beta>0$ and recall $\tau_j=\max\{T_1,T_2\theta_j^{\ssv(p_0-1)}\}$ and $\enu=\tau^{\pnu}$ from Theorem \ref{th:holodecay} \ref{item:holodecayseq}. Since $\theta_j\to 0$ and $p_0<1$, we have $\tau_j\to\infty$, and $\tau_j\geq T_1>1$ for all $j\in\IN$. Hence
\begin{align*}
    d_\beta(j)\coloneqq\sup_{n\in\IN}\,(1+2n)^{\beta/2}\tau_j^{-n}<\infty,\qquad j\in\IN,
\end{align*}
and $1+2n\leq 3^n$ gives $(1+2n)^{\beta/2}\tau_j^{-n}\leq\bigl(3^{\beta/2}\tau_j^{-1}\bigr)^{n}\leq 1$ whenever $\tau_j\geq 3^{\beta/2}$. Thus, as $\tau_j\to\infty$, there is $j_\beta\in\IN$ with $d_\beta(j)\leq 1$ for all $j\geq j_\beta$. Therefore
\begin{align*}
    D_\beta\coloneqq\prod_{j<j_\beta}\max\{1,d_\beta(j)\}<\infty\quad\text{ and }\quad\prod_{j\in\supp\pnu}d_\beta(j)\leq D_\beta, \quad \pnu\in\F.
\end{align*}
Using \eqref{eq:gzjestimate} in the form $\lvert\Gznuk\rvert\lesssim\anu$, we obtain
\begin{align}
    \cnuk\lesssim\anu
    &=\underbrace{\anu\enu}_{\leq C\text{ since }(\anu\enu)_{\pnu\in\F}\in\ell^1(\F)}\enu^{-1}\notag\\
    &\lesssim\prod_{j\in\supp\pnu}\max\left\{T_1,T_2\theta_j^{\ssv(p_0-1)}\right\}^{-\nu_j}=\prod_{j\in\supp\pnu}\tau_j^{-\nu_j}\notag\\
    &\leq D_\beta\prod_{j\in\supp\pnu}(1+2\nu_j)^{-\beta/2}\leq D_\beta\,\omeganu^{-\beta},\label{eq:maxcnuesst}
\end{align}
where the last inequality follows by inserting $1=(1+2\nu_j)^{\beta/2}(1+2\nu_j)^{-\beta/2}$ into each factor and applying the definition of $d_\beta(j)$.

Let $\rho>0$ and set $\beta\coloneqq\qq/\rho$. Combining \eqref{eq:cnulowerbound} and \eqref{eq:maxcnuesst} shows
\begin{align*}
    j^{-\qq}\lesssim c_{(\pnu,k)_j}\lesssim\omega_j^{-\qq/\rho},
\end{align*}
that is, $\omegaj\lesssim j^{\rho}$. Since $\rho>0$ was arbitrary, \eqref{eq:omegaassumppolynomial} holds for every $\rho>0$.

By definition of $\cnuk$ we have $\lvert\Gznuk\rvert\leq\cnuk$ and hence
$$\lvert g_{0,j}\rvert=\lvert g_{0,(\pnu,k)_j}\rvert\leq c_{(\pnu,k)_j}\lesssim j^{-1/p}.$$
Choosing $1/p=q-\eps/2=\min\{s,t+1/2\}-\eps$ for arbitrarily small $\eps>0$ shows the claim.

{\bf Step 3 (proof of \ref{item:generalrate}).} Consider the setup from above with $0<\qq<\min\{\ssv,\ttv\}$, $p>1/\qq$ and $\ttv'\in[\qq,\ttv)$. Replace $\cnuk$ by $\bar c_{\pnu,k}\coloneqq\anu k^{-\qq}$ and consider a monotonically decreasing rearrangement $i$ of $(\bar c_{\pnu,k})_{\pnu\in\F,k\in\IN}$. Since $(\anu)_{\pnu\in\F}$ in Theorem \ref{th:holodecay} does not depend on $g_0$, neither does $i$. Then $\bar c_j\coloneqq\bar c_{(\pnu,k)_j}\lesssim j^{-1/p}$ and $\omegaj\lesssim j^{\rho}$, $p\in (1/\qq,2)$, $\rho>0$, follow as in Steps 1 and 2. Moreover, Theorem \ref{th:holodecay} \ref{item:holodecaylp}, $\theta_k\lesssim k^{-1}$ and $\ttv'\geq\qq$ yield $$\lvert\Gznuk\rvert\lesssim\anu\theta_k^{\ttv'}\lesssim\anu k^{-\qq}=\bar c_{\pnu,k},$$ and thus $\lvert\Gzj\rvert\lesssim\bar c_j\lesssim j^{-1/p}$ for $p>1/\qq$. Choosing $$1/p=\qq-\eps/2=\min\{\ssv,\ttv\}-\eps$$ shows \ref{item:generalrate}.

\subsection{Proof of Corollary \ref{cor:holomorphlearning}}\label{app:holomorphlearning}
We only prove \ref{item:g0explicitrate}.
We aim to apply Corollary \ref{cor:Expectationconvergence} and check its assumptions.
Let $\eps>0$ and set $m\coloneqq\min\{\ssv,\ttv+1/2\}>1/2$, $M\coloneqq\min\{\rr,m\}$ and
\begin{align*}
    \qq\coloneqq m-\min\left\{\tfrac{\rr\eps}{2},\tfrac{m-1/2}{2}\right\},\qquad
    \rho\coloneqq\min\left\{\tfrac{\rr\eps}{4},\tfrac{\rr-1}{4}\right\}.
\end{align*}
Then $\qq\geq(m+1/2)/2>1/2$ and $2\rho\leq(\rr-1)/2<\rr-1$, and Lemma
\ref{lem:holomakesconvergence} yields \eqref{eq:omegaassumppolynomial} and
\eqref{eq:rtassump} for this pair $(\qq,\rho)$. Since $\rr-\rho>1$,
\begin{align*}
    \sumj\kappaj\omegaj\max\left\{1,\log(j)^{1/2}\right\}\lesssim\sumj j^{-\rr+\rho}\log(j)^{1/2}<\infty,
\end{align*}
so \eqref{eq:strongerdecayproperty} holds.
Moreover, as $K=\CRr(\X)$ is compact with
$\supp(\IPX)=K$ and $g_0\in C^0(K,\Y)$ by Assumption \ref{ass:holo} \ref{item:holoassumptionholo}, Setting
\ref{set:gaussiansetup} is satisfied by Lemma \ref{lem:decay}. Hence the assumptions of Theorem
\ref{th:convergencetheorem} are fulfilled. From $M\leq m$, $M\leq\rr$ we get
$\min\{\qq,\rr-2\rho\}\geq M-\rr\eps/2$, so \eqref{eq:deltan} gives
\begin{align*}
    \delta_n^2\leq\max\{1,\log(n)\}^2\,n^{-\frac{2M-1}{2\rr}+\frac{\eps}{2}}\lesssim n^{-\frac{2M-1}{2\rr}+\eps}.
\end{align*}
Then \ref{item:g0explicitrate} follows from Corollary \ref{cor:Expectationconvergence}; \ref{item:generalexplicitrate} follows by the same proof with $m\coloneqq\min\{\ssv,\ttv\}>1/2$.

\section{Posterior Mean Convergence in \texorpdfstring{$C^0(K)$}{C0(K)}}\label{app:C0Kconvergence}

One can use the proof strategy of Theorem \ref{th:convergencetheorem} to get concentration inequalities w.r.t.~the $C^0(K)$-norm.

\begin{theorem}\label{th:concentrationC0}
Consider the setting of Theorem \ref{th:convergencetheorem}, in particular let \eqref{eq:omegaassumppolynomial} and \eqref{eq:rtassump} hold. 
Moreover, let $\qq>1+\rho$ and $\rr>3/2+3\rho$. Define the sequence $(\delta_n)_{n\in\IN}$ with 
\begin{align}
    \label{eq:deltanC0}
    \delta_n=\max\{1,\log(n)\}~n^{-\frac{\min\{\qq,\rr-2\rho\}-\rho-1}{2\rr}},\qquad n\in\IN.
\end{align}
 
Then there exists a constant $C>0$ s.t.~for all $n\in\IN$ and $\delta\geq \delta_n$ 
\begin{subequations}\label{eq:convergencegaussianC0}
\begin{fleqn}[2em]
\begin{align}
 &\left(\IPpX\otimes\IP_{\pEE}\right)\left(\left\lVert\hatgn(\pXX,d\circ g_0(\pXX)+\pEE)-g_0\right\rVert_{C^0(K)}^2\geq\delta^2\right)\notag
\end{align}
\end{fleqn}
\vspace{-\abovedisplayskip}
\begin{empheq}[left={\leq\empheqlbrace}]{alignat=2}
    &C\exp\left(-C^{-1}n^{(\rr-2\rho-1)/\rr}\delta^2\right),
      &&\;\;\delta_n\leq\delta\leq n^{-\frac{1}{4\rr}},\label{eq:convergencegaussianC01}\\
    &C\exp\left(-C^{-1}n^{(2\rr-4\rho-3)/(2\rr)}\right),
      &&\;\;n^{-\frac{1}{4\rr}}\leq\delta\leq n^{2},\label{eq:convergencegaussianC02}\\
    &C\exp\left(-C^{-1}n^{-(6\rr+4\rho+3)/(2\rr)}\delta^2\right),
      &&\;\;\delta\geq n^{2}\label{eq:convergencegaussianC03}.
\end{empheq}
\end{subequations}
\end{theorem}

\subsection*{Proof of Theorem \ref{th:concentrationC0}}
We follow the proof of Theorem \ref{th:convergencetheorem}, in particular Steps \hyperlink{hyp:step1}{1}--\hyperlink{hyp:step5}{5}.

{\bf \hypertarget{hyp:C0step1}{Step 1} (preparation and definitions).}
Following \eqref{eq:omegaassumppolynomial} and \eqref{eq:rtassump}, the decay property \eqref{eq:strongerdecayproperty} is fulfilled. Thus we can apply Lemma \ref{lem:calcgaussiangnhat} which shows
\begin{align*}
    \hatgn(\px,\pz)=\limNinfty\bs{\psi}_N^T\ANin(\px)\bN(\px,\pz)
\end{align*}
with convergence in $C^0(K)$ and $\AN$ and $\bN$ from \eqref{eq:ANbNgaussianproof}. Similar to \eqref{eq:dn}, for $\delta>0$, we define the sequence $({d_n^\delta})_{n\in\IN}$ as
\begin{align}
    {d_n^\delta}=c n^{-\frac {4\rho+3} {2\rr}}\max\left\{\min\{1,n^{\frac{1}{2\rr}}\delta^2\},n^{-4}\delta^2\right\},~~n\in\IN,\label{eq:dnc0}
\end{align}
with $c>0$ sufficiently small, see \hyperlink{hyp:C0step3a}{Step 3a}.
Consider the set $S^\delta$ from \eqref{eq:Sdelta} with ${d_n^\delta}$ from \eqref{eq:dnc0} instead of \eqref{eq:dn}.

{\bf \hypertarget{hyp:C0step2}{Step 2} (upper bound $(\IPpX\otimes\IP_{\pEE})\left(S^\delta\right)$).}
The proof is identical to the \hyperlink{hyp:step2}{$L^2(K)$}-case showing
\begin{align}
    (\IPpX\otimes\IP_{\pEE})\left(S^\delta\right)\geq 1-
C\exp\left(-\frac{n{d_n^\delta}}{C}\right)\label{eq:SmeasureC0}
\end{align}
for $\delta\gtrsim\delta_n$ with $\delta_n$ from \eqref{eq:deltanC0} and ${d_n^\delta}$ from \eqref{eq:dnc0}.

{\bf \hypertarget{hyp:C0step3}{Step 3} (upper bound $\boldsymbol{\lVert\hatgn-g_0\rVert_{C^0(K)}}$ on $\boldsymbol{S^\delta}$ for $\boldsymbol{\delta_n\lesssim\delta\leq n^{2}}$).}
Define $$\GzeroN\coloneqq(\langle g_0,\psij\rangle_{L^2(K)})_{j=1}^N$$ and $\bs{w}_N\coloneqq(\lVert\psi_j\rVert_{\infty})_{j=1}^N$. Then Lemma \ref{lem:calcgaussiangnhat} and the triangle inequality yields $(\px,\pz)$ $(\IPpX\otimes\IPpE)$-a.s.
\begin{align}
\label{eq:gnhatg0C0}
\lVert\gnhat(\px,\pz)-g_0\rVert_{\infty}^2&=\limNinfty\lVert \bs{\psi}_N^T\ANin(\px)\bN(\px,\pz)-\bs{\psi}_N^T\GzeroN\rVert_{\infty}^2\notag\\
&\leq\limNinfty(\wN^T|\ANin(\px)\bN(\px,\pz)-\GzeroN|)^2.
\end{align} 
Defining $\DN$ and $\RRN$ as in \eqref{eq:DNRN} and suppressing $\px$ and $\pz$ in the notation from now on we split \eqref{eq:gnhatg0C0} as
\begin{align}
  \label{eq:convonetwosplitC0}
  (\wN^T|\ANin\bN-\GzeroN|)^2
  \leq~&2\underbrace{(\wN^T|(\ANin-\DNinv)\bN|)^2 }_{\coloneqq \text{\hypertarget{hyp:C0firstterm}{(i)}}}
  +2\underbrace{(\wN^T|\DNinv\bN-\GzeroN|)^2}_{\coloneqq\text{\text{\hypertarget{hyp:C0secondterm}{(ii)}}}} .
\end{align}
As in the proof of Theorem \ref{th:convergencetheorem} we now estimate the summands \hyperlink{hyp:C0firstterm}{$(i)$} and \hyperlink{hyp:C0secondterm}{$(ii)$} in \eqref{eq:convonetwosplitC0} separately.

{\bf \hypertarget{hyp:C0step3a}{Step 3a} (estimation of {\hyperlink{hyp:C0firstterm}{$(i)$}}).}
Let $\EN=-\DNhalfinv\RRN\DNhalfinv$. Then, by following \eqref{eq:ENtwosums}--\eqref{eq:ENsecondsum}, it holds
\begin{align}
    \label{eq:ENboundC0}
    \lVert\EN\rVert_2\lesssim n^{\frac{2\rho+1}{2\rr}}{d_n^\delta}^{1/2}+n^{-\frac{2\rr-2\rho-1}{4\rr}}.
\end{align}
Thus, by \eqref{eq:dnc0}, we have $\lVert\EN\rVert_2<1/2$ for $n\geq n_0$ and sufficiently small $c$ in \eqref{eq:dnc0}.
Similar to \eqref{eq:twofactor} we bound \hyperlink{hyp:C0firstterm}{$(i)$} as 
\begin{align}
    \label{eq:twofactorC0}
    \hyperlink{hyp:C0firstterm}{(i)}=(\wN^T|(\ANin-\DNinv)\bN|)^2&\leq \lVert\underbrace{\wN^T\DNhalfinv}_{\coloneqq\uN}\rVert_2^2~\lVert (\IIN-\EN)^{-1}\rVert_2^2 ~\lVert \vN\rVert_2^2\notag\\
    &\lesssim   \lVert \uN\rVert_2^2~\lVert \vN\rVert_2^2
\end{align}
with $\vN=\EN\DNhalfinv\bN$ and using \eqref{eq:ENboundC0}. We now estimate the two factors in \eqref{eq:twofactorC0}. For $\lVert\uN\rVert_2^2$ we use \cite[Lemma B.2]{de2023convergence} to estimate
\begin{align}\label{eq:uNC0}
    \lVert\uN\rVert_2^2&=\sumjN \omega_j^2\frac{n\kappaj^2}{1+\alphajj n\kappa_j^2}\notag\\
    &=n\sumjnr \frac{\kappaj^2\omega_j^2}{1+\alphajj n\kappa_j^2}+n\sumjnrgr\kappaj^2\omega_j^2\notag\\
    &\lesssim n\sumjntwor \frac{j^{-2\rr+2\rho}}{1+nj^{-2\rr}}+n\sumjnrgr j^{-2\rr+2\rho}\notag\\
    &\lesssim n n^{-\frac{2\rr-2\rho-1}{2\rr}}+ n n^{-\frac{2\rr-2\rho-1}{\rr-\rho-1}}\lesssim n^{\frac{2\rho+1}{2\rr}}.
\end{align}
For $\lVert\vN\rVert_2^2$ following \eqref{eq:vnestimate}--\eqref{eq:vnfinal}\footnote{The estimate ${d_n^\delta}^{1/2}n^{\frac{2\rho+1}{2\rr}}\lesssim 1$, used in \eqref{eq:vnsecond}, is still valid for ${d_n^\delta}$ from \eqref{eq:dnc0} instead of \eqref{eq:dn}.} shows 
\begin{align}
    \label{eq:vNC0}
    \lVert\vN\rVert_2^2\lesssim {d_n^\delta} n^{\frac{2\rho+1}{2\rr}}.
\end{align}
Inserting \eqref{eq:uNC0} and \eqref{eq:vNC0} into \eqref{eq:twofactorC0} shows

\begin{align}
    \hyperlink{hyp:C0firstterm}{(i)}&\lesssim {d_n^\delta} n^{\frac{4\rho+2}{2\rr}} 
\lesssim \min\{n^{-\frac{4\rho+2}{2\rr}}\delta^2,n^{-\frac{4\rho+3}{2\rr}}\}n^{\frac{4\rho+2}{2\rr}} 
\lesssim\max\{\delta^2, \underbrace{n^{-\frac{2\qq-2\rho-1}{2\rr}}}_{\leq\delta_n^2\lesssim\delta^2}\}\lesssim\delta^2\label{eq:iestimateC0}
\end{align}
finishing \hyperlink{hyp:C0step3a}{Step 3a}.

{\bf \hypertarget{hyp:C0step3b}{Step 3b} (estimation of {\hyperlink{hyp:C0secondterm}{$(ii)$}}).}
We split \hyperlink{hyp:C0secondterm}{$(ii)$} from \eqref{eq:convonetwosplitC0} similar to \eqref{eq:iisplit} as
\begin{align}
    \label{eq:iisplitc0}
\hyperlink{hyp:C0secondterm}{(ii)}^{1/2}&=\wN^T|\DNinv\bN-\GzeroN|\\
&=\sumjN\omega_j\left|\frac{n\kappaj^2\betaj}{1+\alphajj n\kappaj^2}-\Gzj\right|\notag\\
     &\leq\sumjN\left|\frac{n\kappaj^2\omega_j(\betaj-\Gzj)}{1+\alphajj n\kappaj^2}\right|+\left|\frac{n\kappaj^2\omega_j(1-\alphajj)\Gzj}{1+\alphajj n\kappaj^2}\right|+\left|\frac{\omega_j\Gzj}{1+\alphajj n\kappaj^2}\right|.\notag
\end{align}
We estimate the first term in \eqref{eq:iisplitc0} similar to \eqref{eq:iifirst} using \eqref{eq:betaj} and \cite[Lemma B.2]{de2023convergence} and get
\begin{align}
     &\sumjN\left|\frac{n\kappaj^2\omega_j(\betaj-\Gzj)}{1+\alphajj n\kappaj^2}\right|\notag\\
\lesssim~&n {d_n^\delta}^{1/2} \sum_{j\leq\lceil n^\frac{1}{\rr-\rho-1}\rceil}\frac{\kappa_j^2\omega_j^2}{1+\alphajj n\kappa_j^2}+n{d_n^\delta}^{1/2}\sum_{j>\lceil n^\frac{1}{\rr-\rho-1}\rceil}\frac{\kappa_j^2\omega_j^2j^{1/2}\log(j)}{1+\alphajj n\kappa_j^2} \notag\\
\lesssim~&nn^{-\frac{2\rr-2\rho-1}{2\rr}}{d_n^\delta}^{1/2}+n \underbrace{\log(n)n^{-\frac{2\rr-2\rho-3/2}{\rr-\rho-1}}}_{\lesssim n^{-\frac{2\rr-2\rho-1}{2\rr}}\text{ for }\rr>1+\rho}{d_n^\delta}^{1/2}\notag\\
     \lesssim~& n^{\frac{2\rho+1}{2\rr}}{d_n^\delta}^{1/2}\lesssim\delta.\label{eq:iifirstC0}
\end{align}
For the second term in \eqref{eq:iisplitc0}, we use \eqref{eq:alphajjestimate} and $\qq>1+\rho$ to get
\begin{align}
    &\sumjN\left|\frac{n\kappaj^2\omega_j(1-\alphajj)\Gzj}{1+\alphajj n\kappaj^2}\right|\notag\\
\lesssim~&n{d_n^\delta}^{1/2}\sum_{j\leq\lceil n^{\frac{1}{\rr-\rho-1}}\rceil}\frac{\kappa_j^2 |g_{0,j}|\omega_j^3}{1+\alphajj n\kappa_j^2}+n\sum_{j>\lceil n^\frac{1}{\rr-\rho-1}\rceil}\frac{\kappa_j^2 |g_{0,j}|\omega_j^3}{1+\alphajj n\kappa_j^2}\notag\\
\lesssim~& n \max\left\{n^{-1},n^{-\frac{2\rr+\qq-3\rho-1}{2\rr}}\right\}{d_n^\delta}^{1/2}+nn^{-\frac{2\rr+\qq-3\rho-1}{\rr-\rho-1}}\notag\\
\lesssim~& n^{\frac{2\rho+1}{2\rr}}{d_n^\delta}^{1/2}+\underbrace{n^{-\frac{\rr+\qq-2\rho}{\rr-\rho-1}}}_{\leq n^{-1/2}\leq\delta_n\lesssim\delta}\lesssim \delta.\label{eq:iisecondsummandC0}
\end{align}
Finally, for the third summand in \eqref{eq:iisplitc0} it holds
\begin{align}
    \label{eq:iithirdsummandC0}
    \sumjN\left|\frac{\omega_j\Gzj}{1+\alphajj n\kappaj^2}\right|\lesssim\sumjntwor \frac{j^{-\qq+\rho}}{1+nj^{-2\rr}}\lesssim n^{-\frac{\qq-\rho-1}{2\rr}}\lesssim\delta_n\lesssim\delta.
\end{align}
Combining \eqref{eq:iifirstC0}--\eqref{eq:iithirdsummandC0} shows $\hyperlink{hyp:C0secondterm}{(ii)}\lesssim\delta^2$
finishing \hyperlink{hyp:C0step3b}{Step 3b}. Together with \eqref{eq:iestimateC0} this finishes \hyperlink{hyp:C0step3}{Step 3}.

{\bf \hypertarget{hyp:C0step4}{Step 4} (upper bound $\boldsymbol{\lVert\hatgn-g_0\rVert_{C^0(K)}}$ on $\boldsymbol{S^\delta}$ for $\boldsymbol{\delta\geq n^{2}}$).}
Following \eqref{eq:gnhatg0C0} and using $\qq>1+\rho$, we have 
\begin{align}\label{eq:ScboundC0}
\lVert\hatgn-g_0\rVert_{\infty}^2\notag
   &\leq2\limNinfty\left(\wN^T|\ANin\bN|\right)^2
   +2\Bigl(\sum_{j=1}^{\infty}\underbrace{|\Gzj|\omega_j}_{\lesssim j^{-\qq+\rho}}\Bigr)^2\notag\\
   &\lesssim \limNinfty\left(\wN^T|\ANin\bN|\right)^2+1.   
\end{align}
Using $\lVert\uN\rVert_2^2\lesssim n$, \eqref{eq:eigenvalue}, \eqref{eq:vNprimeestimate} and the definition of ${d_n^\delta}$ in \eqref{eq:dnc0} we bound 

\begin{align}
    \left(\wN^T|\ANin\bN|\right)^2\lesssim \lVert \uN\rVert_2^2~\lVert\DN^{1/2}\AN^{-1}\DN^{1/2}\rVert_2^2~\lVert\vN'\rVert_2^2 
    \lesssim n^{4}(1+{d_n^\delta})\lesssim\delta^2\label{eq:roughboundC0}
\end{align}
where we additionally used $2\rr>2\rho+1$.
Inserting \eqref{eq:roughboundC0} into \eqref{eq:ScboundC0} and using $1\leq n^{4}\lesssim\delta^2$ shows $\lVert\hatgn-g_0\rVert_\infty^2\lesssim\delta^2$ which finishes \hyperlink{hyp:C0step4}{Step 4}.

{\bf \hypertarget{hyp:C0step5}{Step 5} (finish the proof).}
Identical argumentation as in \eqref{eq:concentrationbprooffinal} lets us conclude 
\begin{align}
    &\left(\IPpX\otimes\IP_{\pEE}\right)\left(\left\lVert\hatgn-g_0\right\rVert_{\infty}^2\geq\delta^2\right) \leq C\exp\left({-\frac{n{d_n^\delta}}{C}}\right),\label{eq:concentrationbprooffinalC0}
\end{align}
with ${d_n^\delta}$ from \eqref{eq:dnc0}. Writing out ${d_n^\delta}$ shows \eqref{eq:convergencegaussianC0} and thus finishes the proof.
\end{appendix}
\end{document}